\documentclass[secthm,seceqn,amsthm,ussrhead,reqno, 12pt]{amsart}
\usepackage[utf8]{inputenc}
\usepackage[english]{babel}
\usepackage[symbol]{footmisc}
\usepackage{amssymb,amsmath,amsthm,amsfonts,xcolor,enumerate,hyperref,comment,longtable,cleveref}

\usepackage{times}
\usepackage{cite}
\usepackage{pdflscape}
\usepackage{ulem}
\usepackage[mathcal]{euscript}
\usepackage{tikz}
\usepackage{hyperref}
\usepackage{cancel}
\usepackage{stmaryrd}

\usetikzlibrary{arrows}

\newtheorem{theorem}{Theorem}
\newtheorem{lemma}[theorem]{Lemma}

\newtheorem{definition}[theorem]{Definition}

\newtheorem*{theoremA}{Theorem A}

\usepackage{stmaryrd}
\usepackage{xcolor}

\begin{document}

\noindent{\Large
The algebraic  classification of   nilpotent Novikov   algebras}\footnote{The authors thank
  Renato Fehlberg J\'{u}nior 
      \& 
    Crislaine Kuster   for their active stimulation to write this paper.
The first part of this work is supported by 
FCT   UIDB/MAT/00212/2020 and UIDP/MAT/00212/2020.
The second part of this work is supported by the Russian Science Foundation under grant 22-71-10001.  
}
%\footnote{Corresponding author: Ivan Kaygorodov   (  kaygorodov.ivan@gmail.com)}

 \bigskip

\begin{center}

 {\bf
Kobiljon Abdurasulov\footnote{Institute of Mathematics Academy of
Sciences of Uzbekistan, Tashkent, Uzbekistan; \ abdurasulov0505@mail.ru}, 
Ivan Kaygorodov\footnote{CMA-UBI, Universidade da Beira Interior, Covilh\~{a}, Portugal;   \
Moscow Center for Fundamental and Applied Mathematics, Moscow,   Russia; \
Saint Petersburg  University, Russia; \
    kaygorodov.ivan@gmail.com}
  \&
Abror Khudoyberdiyev\footnote{Institute of Mathematics Academy of
Sciences of Uzbekistan, Tashkent, Uzbekistan; \
National University of Uzbekistan, Tashkent, Uzbekistan, \ khabror@mail.ru}

}

\end{center}

 \bigskip

\noindent{\bf Abstract}:
{\it This paper is devoted to the complete algebraic classification of complex $5$-dimensional nilpotent Novikov
algebras.
%In particular,  we proved that the variety of complex $5$-dimensional  nilpotent novikov algebras has dimension $??$, it is defined by $??$ irreducible components and it has $??$ rigid algebras.
}

 \bigskip


 \bigskip
 

 \medskip

\section*{Introduction}

One of the classical problems in the theory of non-associative algebras is to classify (up to isomorphism) the algebras of dimension $n$ from a certain variety defined by some family of polynomial identities. It is typical to focus on small dimensions, and there are two main directions for the classification: algebraic and geometric. Varieties as Jordan, Lie, Leibniz or Zinbiel algebras have been studied from these two approaches (\!\cite{kkl20,cfk18,bfk23, kkp20,klp20,  fkkv22, kv16, jkk21, kkl22  } and 
\cite{ kv16,  kkl20,ikp20,   fkkv22}, respectively).
In the present paper, we give the algebraic   classification of
$5$-dimensional nilpotent Novikov   algebras.

The variety of Novikov algebras is defined by the following identities:
\[
\begin{array}{rcl}
(xy)z &=& (xz)y,\\
(xy)z-x(yz) &=& (yx)z-y(xz).
\end{array} \]
It contains commutative associative algebras as a subvariety.
On the other hand, the variety of Novikov algebras is the intersection of
the variety of right commutative algebras (defined by the first Novikov identity)
and
the variety of left symmetric (Pre-Lie) algebras
(defined by the second Novikov identity, 
see about it \cite{bur06} and references therein).
Also, a Novikov algebra with the commutator multiplication gives a Lie algebra,
and Novikov algebras are related to
Tortken and Novikov-Poisson algebras  \cite{xu97,zerui18}.
The class of Novikov (known as  Gelfand-Dorfman-Novikov algebras) 
appeared in papers of  Gelfand ---  Dorfman \cite{gd79} and Novikov --- Balinsky \cite{bn85}.

The systematic study of Novikov algebras from algebraic view started after the paper of Zelmanov where
all complex  finite-dimensional simple Novikov algebras were classified  \cite{ze87}.
The first nontrivial examples of  infinite-dimensional simple  Novikov algebras were constructed by Filippov in \cite{fi89}.
Also, simple Novikov algebras (under some special conditions) were described in the infinite-dimensional case and over fields of positive characteristic 
in some papers by Osborn and Xu \cite{os94, xu96, xu01}.

Many other purely algebraic properties of Novikov algebras were studied in a series by papers of Dzhumadildaev \cite{di14, dt05, dz11, dl02}.
So, Dzhumadildaev and   Löfwall described the basis of free Novikov algebras \cite{dl02};
Dzhumadildaev proved  that the Novikov operad is not Koszul \cite{dz11};
Dzhumadildaev and  Ismailov  found the $S_n$-module structure of the multilinear component of degree $n$ of the $n$-generated free Novikov algebra over a field of characteristic 0 \cite{di14}.
Makar-Limanov and Umirbaev  proved  The Freiheitssatz for Novikov algebras \cite{mlu10},
and Duĭsengalieva and  Umirbaev constructed a wild automorphism of the three-generated free Novikov algebra \cite{du17}.  
Novikov central extensions of $n$-dimensional restricted polynimonial algebras are studied by Kaygorodov,
 Lopes and P\'{a}ez-Guill\'{a}n in \cite{klp20}.
Chen, Niu  and Meng gave  
 some new realizations of two Novikov algebras \cite{chen08}. 
Lebzioui studied 
  pseudo-Euclidean Novikov algebras in \cite{l20}.

Filippov proved that each Novikov nilalgebra is nilpotent \cite{fi01}.
Dzhumadildaev and  Tulenbaev proved that 
 if each left multiplication of a Novikov algebra over $K$  (char $K=p$, $p=0$ or $p > n+1$) 
 has nil-index $n$, then $A^2$ is nilpotent with nilpotency index less than or equal to $n$ \cite{dt05}.
 Shestakov and Zhang proved analogues of Itô's and Kegel's theorems for Novikov algebras \cite{zerui20}.
 Another interesting direction in the algebraic study of Novikov algebras is the description of possible Novikov structures on a certain Lie algebra \cite{bdv08,tang12}.

The algebraic classification of $3$-dimensional Novikov algebras was given in \cite{bc01},
and for some classes of $4$-dimensional algebras, it was given in \cite{bg13};
 $4$-dimensional  and one-generated $6$-dimesional complex nilpotent Novikov algebras are described in \cite{kkk19, ckkk20}, respectively. 
The geometric classification of $3$-dimensional Novikov algebras was given in \cite{bb14} and of
$4$-dimensional nilpotent Novikov algebras in \cite{kkk19}.

Our method for classifying nilpotent  Novikov  algebras is based on the calculation of central extensions of nilpotent algebras of smaller dimensions from the same variety. 
The algebraic study of central extensions of   algebras has been an important topic for years \cite{  klp20,hac16,  ss78}.
First, Skjelbred and Sund used central extensions of Lie algebras to obtain a classification of nilpotent Lie algebras  \cite{ss78}.
Note that the Skjelbred-Sund method of central extensions is an important tool in the classification of nilpotent algebras.
Using the same method,  
 small dimensional nilpotent 
(associative, 
 terminal  \cite{kkp20}, Jordan,
  Lie, 
 anticommutative  \cite{kkl20}) algebras,
and some others have been described. Our main results related to the algebraic classification of the  variety of Novikov algebras are summarized below.

\begin{theoremA}
Up to isomorphism, there are infinitely many isomorphism classes of  
complex  non-split non-one-generated $5$-dimensional   nilpotent (non-2-step nilpotent) non-commutative  Novikov  algebras, 
described explicitly  in  section \ref{secteoA} in terms of 
$82$ one-parameter families, $27$ two-parameter families, $5$ three-parameter families and 
$104$ additional isomorphism classes.
\end{theoremA}

 \newpage

\section{The algebraic classification of nilpotent Novikov algebras}
\subsection{Method of classification of nilpotent algebras}

Throughout this paper, we use the notations and methods well written in \cite{hac16},
which we have adapted for the Novikov case with some modifications.
Further in this section, we give some important definitions.

Let $({\bf A}, \cdot)$ be a Novikov  algebra over  $\mathbb C$
and $\mathbb V$ a vector space over ${\mathbb C}$. The $\mathbb C$-linear space ${\rm Z^{2}}\left(
\bf A,\mathbb V \right) $ is defined as the set of all  bilinear maps $\theta  \colon {\bf A} \times {\bf A} \longrightarrow {\mathbb V}$ such that

\begin{longtable}{rcl}
$\theta(xy,z)$&$=$&$\theta(xz,y),$\\
$\theta(xy,z)-\theta(x,yz)$&$=$&$ \theta(yx,z)-\theta(y,xz).$
\end{longtable}

These elements will be called {\it cocycles}. For a
linear map $f$ from $\bf A$ to  $\mathbb V$, if we define $\delta f\colon {\bf A} \times
{\bf A} \longrightarrow {\mathbb V}$ by $\delta f  (x,y ) =f(xy )$, then $\delta f\in {\rm Z^{2}}\left( {\bf A},{\mathbb V} \right) $. We define ${\rm B^{2}}\left({\bf A},{\mathbb V}\right) =\left\{ \theta =\delta f\ : f\in {\rm Hom}\left( {\bf A},{\mathbb V}\right) \right\} $.
%One can easily check that ${\rm B^{2}}(\bf A,\mathbb V)$ is a linear subspace of ${\rm Z^{2}}\left( {\bf A},{\mathbb V}\right) $; %its elements are called {\it coboundaries}. 
We define the {\it second cohomology space} ${\rm H^{2}}\left( {\bf A},{\mathbb V}\right) $ as the quotient space ${\rm Z^{2}}
\left( {\bf A},{\mathbb V}\right) \big/{\rm B^{2}}\left( {\bf A},{\mathbb V}\right) $.
% The equivalence class of $%
%\theta \in {\rm Z^{2}}\left( {\bf A},{\mathbb V}\right) $ will be denoted by $\left[
%\theta \right] \in {\rm H^{2}}\left( {\bf A},{\mathbb V}\right) $.

\

Let $\operatorname{Aut}({\bf A}) $ be the automorphism group of  ${\bf A} $ and let $\phi \in \operatorname{Aut}({\bf A})$. For $\theta \in
{\rm Z^{2}}\left( {\bf A},{\mathbb V}\right) $ define  the action of the group $\operatorname{Aut}({\bf A}) $ on ${\rm H^{2}}\left( {\bf A},{\mathbb V}\right) $ by $\phi \theta (x,y)
=\theta \left( \phi \left( x\right) ,\phi \left( y\right) \right) $.  It is easy to verify that
 ${\rm B^{2}}\left( {\bf A},{\mathbb V}\right) $ is invariant under the action of $\operatorname{Aut}({\bf A}).$  
 So, we have an induced action of  $\operatorname{Aut}({\bf A})$  on ${\rm H^{2}}\left( {\bf A},{\mathbb V}\right)$.

\

Let $\bf A$ be a Novikov  algebra of dimension $m$ over  $\mathbb C$ and ${\mathbb V}$ be a $\mathbb C$-vector
space of dimension $k$. For $\theta \in {\rm Z^{2}}\left(
{\bf A},{\mathbb V}\right) $, define on the linear space ${\bf A}_{\theta } = {\bf A}\oplus {\mathbb V}$ the
bilinear product `` $\left[ -,-\right] _{{\bf A}_{\theta }}$'' by $\left[ x+x^{\prime },y+y^{\prime }\right] _{{\bf A}_{\theta }}=
 xy +\theta(x,y) $ for all $x,y\in {\bf A},x^{\prime },y^{\prime }\in {\mathbb V}$.
The algebra ${\bf A}_{\theta }$ is called an $k$-{\it dimensional central extension} of ${\bf A}$ by ${\mathbb V}$. One can easily check that ${\bf A_{\theta}}$ is a Novikov 
algebra if and only if $\theta \in {\rm Z^2}({\bf A}, {\mathbb V})$.

Call the
set $\operatorname{Ann}(\theta)=\left\{ x\in {\bf A}:\theta \left( x, {\bf A} \right)+ \theta \left({\bf A} ,x\right) =0\right\} $
the {\it annihilator} of $\theta $. We recall that the {\it annihilator} of an  algebra ${\bf A}$ is defined as
the ideal $\operatorname{Ann}(  {\bf A} ) =\left\{ x\in {\bf A}:  x{\bf A}+ {\bf A}x =0\right\}$. Observe
 that
$\operatorname{Ann}\left( {\bf A}_{\theta }\right) =(\operatorname{Ann}(\theta) \cap\operatorname{Ann}({\bf A}))
 \oplus {\mathbb V}$.

\

The following result shows that every algebra with a non-zero annihilator is a central extension of a smaller-dimensional algebra.

\begin{lemma}
Let ${\bf A}$ be an $n$-dimensional Novikov algebra such that $\dim (\operatorname{Ann}({\bf A}))=m\neq0$. Then there exists, up to isomorphism, a unique $(n-m)$-dimensional Novikov  algebra ${\bf A}'$ and a bilinear map $\theta \in {\rm Z^2}({\bf A}, {\mathbb V})$ with $\operatorname{Ann}({\bf A})\cap\operatorname{Ann}(\theta)=0$, where $\mathbb V$ is a vector space of dimension m, such that ${\bf A} \cong {{\bf A}'}_{\theta}$ and
 ${\bf A}/\operatorname{Ann}({\bf A})\cong {\bf A}'$.
\end{lemma}

\begin{proof}
Let ${\bf A}'$ be a linear complement of $\operatorname{Ann}({\bf A})$ in ${\bf A}$. Define a linear map $P \colon {\bf A} \longrightarrow {\bf A}'$ by $P(x+v)=x$ for $x\in {\bf A}'$ and $v\in\operatorname{Ann}({\bf A})$, and define a multiplication on ${\bf A}'$ by $[x, y]_{{\bf A}'}=P(x y)$ for $x, y \in {\bf A}'$.
For $x, y \in {\bf A}$, we have
\[P(xy)=P((x-P(x)+P(x))(y- P(y)+P(y)))=P(P(x) P(y))=[P(x), P(y)]_{{\bf A}'}. \]

Since $P$ is a homomorphism $P({\bf A})={\bf A}'$ is a Novikov algebra and
 ${\bf A}/\operatorname{Ann}({\bf A})\cong {\bf A}'$, which gives us the uniqueness. Now, define the map $\theta \colon {\bf A}' \times {\bf A}' \longrightarrow\operatorname{Ann}({\bf A})$ by $\theta(x,y)=xy- [x,y]_{{\bf A}'}$.
  Thus, ${\bf A}'_{\theta}$ is ${\bf A}$ and therefore $\theta \in {\rm Z^2}({\bf A}, {\mathbb V})$ and $\operatorname{Ann}({\bf A})\cap\operatorname{Ann}(\theta)=0$.
\end{proof}
 
\begin{definition}
Let ${\bf A}$ be an algebra and $I$ be a subspace of $\operatorname{Ann}({\bf A})$. If ${\bf A}={\bf A}_0 \oplus I$
then $I$ is called an {\it annihilator component} of ${\bf A}$.
\end{definition}
\begin{definition}
A central extension of an algebra $\bf A$ without annihilator component is called a {\it non-split central extension}.
\end{definition}

Our task is to find all central extensions of an algebra $\bf A$ by
a space ${\mathbb V}$.  In order to solve the isomorphism problem we need to study the
action of $\operatorname{Aut}({\bf A})$ on ${\rm H^{2}}\left( {\bf A},{\mathbb V}
\right) $. To do that, let us fix a basis $e_{1},\ldots ,e_{s}$ of ${\mathbb V}$, and $
\theta \in {\rm Z^{2}}\left( {\bf A},{\mathbb V}\right) $. Then $\theta $ can be uniquely
written as $\theta \left( x,y\right) =
\displaystyle \sum_{i=1}^{s} \theta _{i}\left( x,y\right) e_{i}$, where $\theta _{i}\in
{\rm Z^{2}}\left( {\bf A},\mathbb C\right) $. Moreover, $\operatorname{Ann}(\theta)=\operatorname{Ann}(\theta _{1})\cap\operatorname{Ann}(\theta _{2})\cap\cdots \cap\operatorname{Ann}(\theta _{s})$. Furthermore, $\theta \in
{\rm B^{2}}\left( {\bf A},{\mathbb V}\right) $\ if and only if all $\theta _{i}\in {\rm B^{2}}\left( {\bf A},
\mathbb C\right) $.
It is not difficult to prove (see \cite[Lemma 13]{hac16}) that given a Novikov algebra ${\bf A}_{\theta}$, if we write as
above $\theta \left( x,y\right) = \displaystyle \sum_{i=1}^{s} \theta_{i}\left( x,y\right) e_{i}\in {\rm Z^{2}}\left( {\bf A},{\mathbb V}\right) $ and 
$\operatorname{Ann}(\theta)\cap \operatorname{Ann}\left( {\bf A}\right) =0$, then ${\bf A}_{\theta }$ has an
annihilator component if and only if $\left[ \theta _{1}\right] ,\left[
\theta _{2}\right] ,\ldots ,\left[ \theta _{s}\right] $ are linearly
dependent in ${\rm H^{2}}\left( {\bf A},\mathbb C\right) $.

Let ${\mathbb V}$ be a finite-dimensional vector space over $\mathbb C$. The {\it Grassmannian} $G_{k}\left( {\mathbb V}\right) $ is the set of all $k$-dimensional
linear subspaces of $ {\mathbb V}$. Let $G_{s}\left( {\rm H^{2}}\left( {\bf A},\mathbb C\right) \right) $ be the Grassmannian of subspaces of dimension $s$ in
${\rm H^{2}}\left( {\bf A},\mathbb C\right) $. There is a natural action of $\operatorname{Aut}({\bf A})$ on $G_{s}\left( {\rm H^{2}}\left( {\bf A},\mathbb C\right) \right) $.
Let $\phi \in \operatorname{Aut}({\bf A})$. For $W=\left\langle
\left[ \theta _{1}\right] ,\left[ \theta _{2}\right] ,\dots,\left[ \theta _{s}
\right] \right\rangle \in G_{s}\left( {\rm H^{2}}\left( {\bf A},\mathbb C
\right) \right) $ define $\phi W=\left\langle \left[ \phi \theta _{1}\right]
,\left[ \phi \theta _{2}\right] ,\dots,\left[ \phi \theta _{s}\right]
\right\rangle $. We denote the orbit of $W\in G_{s}\left(
{\rm H^{2}}\left( {\bf A},\mathbb C\right) \right) $ under the action of $\operatorname{Aut}({\bf A})$ by $\operatorname{Orb}(W)$. Given
\[
W_{1}=\left\langle \left[ \theta _{1}\right] ,\left[ \theta _{2}\right] ,\dots,
\left[ \theta _{s}\right] \right\rangle ,W_{2}=\left\langle \left[ \vartheta
_{1}\right] ,\left[ \vartheta _{2}\right] ,\dots,\left[ \vartheta _{s}\right]
\right\rangle \in G_{s}\left( {\rm H^{2}}\left( {\bf A},\mathbb C\right)
\right),
\]
we easily have that if $W_{1}=W_{2}$, then $ \bigcap\limits_{i=1}^{s}\operatorname{Ann}(\theta _{i})\cap \operatorname{Ann}\left( {\bf A}\right) = \bigcap\limits_{i=1}^{s}
\operatorname{Ann}(\vartheta _{i})\cap\operatorname{Ann}( {\bf A}) $, and therefore we can introduce
the set
\[
{\bf T}_{s}({\bf A}) =\left\{ W=\left\langle \left[ \theta _{1}\right] ,
\left[ \theta _{2}\right] ,\dots,\left[ \theta _{s}\right] \right\rangle \in
G_{s}\left( {\rm H^{2}}\left( {\bf A},\mathbb C\right) \right) : \bigcap\limits_{i=1}^{s}\operatorname{Ann}(\theta _{i})\cap\operatorname{Ann}({\bf A}) =0\right\},
\]
which is stable under the action of $\operatorname{Aut}({\bf A})$.

Now, let ${\mathbb V}$ be an $s$-dimensional linear space and let us denote by
${\bf E}\left( {\bf A},{\mathbb V}\right) $ the set of all {\it non-split $s$-dimensional central extensions} of ${\bf A}$ by
${\mathbb V}$. By above, we can write
\[
{\bf E}\left( {\bf A},{\mathbb V}\right) =\left\{ {\bf A}_{\theta }:\theta \left( x,y\right) = \sum_{i=1}^{s}\theta _{i}\left( x,y\right) e_{i} \ \ \text{and} \ \ \left\langle \left[ \theta _{1}\right] ,\left[ \theta _{2}\right] ,\dots,
\left[ \theta _{s}\right] \right\rangle \in {\bf T}_{s}({\bf A}) \right\} .
\]
We also have the following result, which can be proved as in \cite[Lemma 17]{hac16}.

\begin{lemma}
 Let ${\bf A}_{\theta },{\bf A}_{\vartheta }\in {\bf E}\left( {\bf A},{\mathbb V}\right) $. Suppose that $\theta \left( x,y\right) =  \displaystyle \sum_{i=1}^{s}
\theta _{i}\left( x,y\right) e_{i}$ and $\vartheta \left( x,y\right) =
\displaystyle \sum_{i=1}^{s} \vartheta _{i}\left( x,y\right) e_{i}$.
Then the Novikov algebras ${\bf A}_{\theta }$ and ${\bf A}_{\vartheta } $ are isomorphic
if and only if
$$\operatorname{Orb}\left\langle \left[ \theta _{1}\right] ,
\left[ \theta _{2}\right] ,\dots,\left[ \theta _{s}\right] \right\rangle =
\operatorname{Orb}\left\langle \left[ \vartheta _{1}\right] ,\left[ \vartheta
_{2}\right] ,\dots,\left[ \vartheta _{s}\right] \right\rangle .$$
\end{lemma}

This shows that there exists a one-to-one correspondence between the set of $\operatorname{Aut}({\bf A})$-orbits on ${\bf T}_{s}\left( {\bf A}\right) $ and the set of
isomorphism classes of ${\bf E}\left( {\bf A},{\mathbb V}\right) $. Consequently, we have a
procedure that allows us, given a Novikov algebra ${\bf A}'$ of
dimension $n-s$, to construct all non-split central extensions of ${\bf A}'$. This procedure is:

{\centerline {\textsl{Procedure}}}

\begin{enumerate}
\item For a given Novikov algebra ${\bf A}'$ of dimension $n-s $, determine ${\rm H^{2}}( {\bf A}',\mathbb {C}) $, $\operatorname{Ann}({\bf A}')$ and $\operatorname{Aut}({\bf A}')$.

\item Determine the set of $\operatorname{Aut}({\bf A}')$-orbits on $T_{s}({\bf A}') $.

\item For each orbit, construct the Novikov algebra associated with a
representative of it.
\end{enumerate}

\subsection{Notations}
Let ${\bf A}$ be a Novikov algebra with
a basis $e_{1},e_{2},\dots,e_{n}$. Then by $\Delta _{ij}$\ we denote the
bilinear form
$\Delta _{ij} \colon {\bf A}\times {\bf A}\longrightarrow \mathbb C$
with $\Delta _{ij}\left( e_{l},e_{m}\right) = \delta_{il}\delta_{jm}$.
Then the set $\left\{ \Delta_{ij}:1\leq i, j\leq n\right\} $ is a basis for the space of
the bilinear forms on ${\bf A}$. Then every $\theta \in
{\rm Z^{2}}\left( {\bf A},\mathbb C\right) $ can be uniquely written as $
\theta = \displaystyle \sum_{1\leq i,j\leq n} c_{ij}\Delta _{{i}{j}}$, where $
c_{ij}\in \mathbb C$.
Let us fix the following notations:

$$\begin{array}{lll}
{\mathcal N}^{i*}_j& \mbox{---}& j \mbox{th }i\mbox{-dimensional nilpotent  Novikov algebra with identity $xyz=0$} \\
{\mathcal N}^i_j& \mbox{---}& j \mbox{th }i\mbox{-dimensional nilpotent "pure" Novikov algebra (without identity $xyz=0$)} \\
{\mathfrak{N}}_i& \mbox{---}& i\mbox{th }4\mbox{-dimensional $2$-step nilpotent algebra} \\
{\rm{N}}_i& \mbox{---}& i\mbox{th  non-split non-one-generated }  
5\mbox{-dimensional 
 nilpotent} \\ 
&&\mbox{(non-$2$-step nilpotent) non-commutative Novikov algebra} \\

\end{array}$$

\subsection{$1$-dimensional central extensions of 
$4$-dimensional $2$-step nilpotent  Novikov algebras}

\subsubsection{The description of second cohomology space}
In the following table, we give the description of the second cohomology space of   $4$-dimensional $2$-step nilpotent Novikov algebras.

\begin{longtable}{ll llllll}
\hline

{${\mathfrak N}_{01}$} $:$ \quad  $e_1e_1 = e_2$ \\ 
\multicolumn{8}{l}{
${\rm H}^2_{com}({\mathfrak N}_{01})=
\Big\langle  
[\Delta_{12}+\Delta_{21}],[\Delta_{13}+\Delta_{31}],[\Delta_{14}+\Delta_{41}],[\Delta_{33}],[\Delta_{34}+\Delta_{43}],[\Delta_{44}]
\Big\rangle $}\\
 
${\rm H}^2({\mathfrak N}_{01})={\rm H}^2_{com}({\mathcal N}^4_{01})\oplus \langle [\Delta _{21}], [\Delta _{31}],[\Delta _{41}],[\Delta _{43}]\rangle$ \\

\hline
${\mathfrak N}_{02}$ $:$ \quad $e_1e_1 =e_3$ \quad $e_2e_2=e_4$  \\ 

\multicolumn{8}{l}{
${\rm H}^2_{com}({\mathfrak N}_{02})=
\Big\langle  
[\Delta_{12}+\Delta_{21}],[\Delta_{13}+\Delta_{31}],[\Delta_{24}+\Delta_{42}]\Big\rangle $}\\
 
${\rm H}^2({\mathfrak N}_{01})={\rm H}^2_{com}({\mathcal N}_{01})\oplus \langle [\Delta _{21}], [\Delta _{31}],[\Delta _{42}]\rangle$ \\

\hline
{${\mathfrak N}_{03}$} $:$ \quad  $e_1e_2=e_3$ \quad $e_2e_1=-e_3$  \\ 

\multicolumn{8}{l}{
${\rm H}^2({\mathfrak N}_{03})=
\Big\langle 
[\Delta_{11}],[\Delta_{14}],[\Delta_{21}],[\Delta_{22}],[\Delta_{24}], [\Delta_{41}], [\Delta_{42}], [\Delta_{44}] 
\Big\rangle $}\\

\hline
${\mathfrak N}_{04}^{\alpha}$ $:$ \quad $e_1e_1=e_3$\quad  $e_1e_2=e_3$\quad $e_2e_2=\alpha e_3$ \\  

\multicolumn{8}{l}{
${\rm H}^2({\mathfrak N}_{04}^{\alpha\neq0})=
\Big\langle [\Delta_{12}],[\Delta_{14}],[\Delta_{21}],[\Delta_{22}],[\Delta_{24}], [\Delta_{41}], [\Delta_{42}], [\Delta_{44}]
 \Big\rangle = \Phi_\alpha $}\\

\multicolumn{8}{l}{
${\rm H}^2({\mathfrak N}_{04}^{0})=
\Phi_0  \oplus \Big\langle  [\Delta_{13}], 
[\Delta_{31}+\Delta_{32}-\Delta_{23}]
 \Big\rangle $}\\

\hline
${\mathfrak N}_{05}$ $:$ \quad $e_1e_1=  e_3$ \quad $e_1e_2=e_3$ \quad  $e_2e_1=e_3$ \\ 

\multicolumn{8}{l}{
${\rm H}^2({\mathfrak N}_{05})=
\Big\langle  [\Delta_{12}],[\Delta_{14}],[\Delta_{21}],[\Delta_{22}],[\Delta_{24}], [\Delta_{41}], [\Delta_{42}], [\Delta_{44}]
\Big\rangle $}\\

\hline
${\mathfrak N}_{06}$ $:$ \quad $e_1e_2 = e_4$ \quad $e_3e_1 = e_4$   \\ 

\multicolumn{8}{l}{
${\rm H}^2({\mathfrak N}_{06})=
\Big\langle   [\Delta_{11}],[\Delta_{13}],[\Delta_{21}],[\Delta_{22}],[\Delta_{23}], [\Delta_{31}], [\Delta_{32}], [\Delta_{33}]
\Big\rangle $}\\
 
\hline
{${\mathfrak N}_{07}$} $:$ \quad $e_1e_2 = e_3$ \quad $e_2e_1 = e_4$   \quad $e_2e_2 = -e_3$ \\ 

 \multicolumn{8}{l}{
${\rm H}^2({\mathfrak N}_{07})=
\Big\langle   [\Delta_{11}],[\Delta_{22}],[\Delta_{23}-\Delta_{13}],[\Delta_{24}], [\Delta_{32}-\Delta_{13}], [\Delta_{41}-\Delta_{14}]
\Big\rangle $}\\

\hline
${\mathfrak N}_{08}^{\alpha}$ $:$ \quad $e_1e_1 = e_3$ \quad $e_1e_2 = e_4$ \quad $e_2e_1 = -\alpha e_3$ \quad $e_2e_2 = -e_4$ \\ 
 
\multicolumn{8}{l}{
${\rm H}^2({\mathfrak N}_{08}^{\alpha\neq1})=
\Big\langle
\relax 
\begin{array}{l}
[\Delta_{12}], [\Delta_{21}], [\Delta_{13}-\alpha\Delta_{23}], 
[\Delta_{14}-\Delta_{24}], \\ \relax [\Delta_{31}-\Delta_{13}],[\Delta_{42}-\Delta_{24}]
\end{array}
\Big\rangle = \Phi_{\alpha} $}\\

\multicolumn{8}{l}{
${\rm H}^2({\mathfrak N}_{08}^{1})=
\Phi_{1} \oplus\Big\langle 
 [\Delta_{32}+\Delta_{41}-\Delta_{23}-\Delta_{14}]
\Big\rangle $}\\

\hline
${\mathfrak N}_{09}^{\alpha}$ $:$ \quad $e_1e_1 = e_4$ \quad $e_1e_2 = \alpha e_4$ \quad  $e_2e_1 = -\alpha e_4$ \quad $e_2e_2 = e_4$ \quad  $e_3e_3 = e_4$\\

\multicolumn{8}{l}{
${\rm H}^2({\mathfrak N}_{09}^{\alpha})=
\Big\langle   [\Delta_{12}] [\Delta_{12}], [\Delta_{21}], [\Delta_{22}],[\Delta_{23}], [\Delta_{31}],[\Delta_{32}],[\Delta_{33}]
\Big\rangle $}\\

\hline
${\mathfrak N}_{10}$ $:$ \quad  $e_1e_2 = e_4$ \quad $e_1e_3 = e_4$ \quad $e_2e_1 = -e_4$ \quad $e_2e_2 = e_4$ \quad $e_3e_1 = e_4$\\ 

\multicolumn{8}{l}{
${\rm H}^2({\mathfrak N}_{10})=
\Big\langle   [\Delta_{12}], [\Delta_{12}], [\Delta_{21}], [\Delta_{22}],[\Delta_{23}], [\Delta_{31}],[\Delta_{32}],[\Delta_{33}]
\Big\rangle $}\\

\hline
${\mathfrak N}_{11}$ $:$ \quad  $e_1e_1 = e_4$ \quad $e_1e_2 = e_4$ \quad $e_2e_1 = -e_4$ \quad $e_3e_3 = e_4$ \quad  \\

\multicolumn{8}{l}{
${\rm H}^2({\mathfrak N}_{11})=
\Big\langle  [\Delta_{12}], [\Delta_{12}], [\Delta_{21}], [\Delta_{22}],[\Delta_{23}], [\Delta_{31}],[\Delta_{32}],[\Delta_{33}]
\Big\rangle $}\\

\hline
{${\mathfrak N}_{12}$} $:$ \quad  $e_1e_2 = e_3$ \quad $e_2e_1 = e_4$  \\ 

 \multicolumn{8}{l}{
${\rm H}^2({\mathfrak N}_{12})=
\Big\langle  [\Delta_{11}], [\Delta_{13}], [\Delta_{14}-\Delta_{41}], [\Delta_{22}],[\Delta_{23}-\Delta_{32}],[\Delta_{24}]
\Big\rangle $}\\

\hline
${\mathfrak N}_{13}$ $:$ \quad $e_1e_1 = e_4$ \quad $e_1e_2 = e_3$ \quad $e_2e_1 = -e_3$ \quad 
$e_2e_2=2e_3+e_4$\\
 
\multicolumn{8}{l}{
${\rm H}^2({\mathfrak N}_{13})=
\Big\langle
\relax \begin{array}{l}
[\Delta_{21}], [\Delta_{22}], [\Delta_{14}+\Delta_{23}], [-\Delta_{13}+2\Delta_{14}+\Delta_{24}], \\ \relax [\Delta_{31}-2\Delta_{13}-2\Delta_{32}+\Delta_{42}],
[\Delta_{41}-2\Delta_{14}-\Delta_{32}]
\end{array}
\Big\rangle $}\\

\hline
{${\mathfrak N}_{14}^{\alpha}$} $:$ \quad   $e_1e_2 = e_4$ \quad $e_2e_1 =\alpha e_4$ \quad $e_2e_2 = e_3$ \\

\multicolumn{8}{l}{
${\rm H}^2({\mathfrak N}_{14}^{\alpha\neq0,1})=
\Big\langle \relax
\begin{array}{l}
[\Delta_{11}], [\Delta_{21}], [\Delta_{23}], [\Delta_{13}+\Delta_{24}],\\ \relax [\Delta_{32}], [(\alpha-1)\Delta_{24}+\alpha\Delta_{31}+\Delta_{42}] 
\end{array}
\Big\rangle  = \Phi_{\alpha}$}\\ 

\multicolumn{8}{l}{
${\rm H}^2({\mathfrak N}_{14}^{0})= \Phi_0 \oplus
\Big\langle  [\Delta_{14}]
\Big\rangle $}\\

\multicolumn{8}{l}{
${\rm H}^2_{com}({\mathfrak N}_{14}^{1})=
\Big\langle  
[\Delta_{11}],\ [\Delta_{13}+\Delta_{31}+\Delta_{24}+\Delta_{42}],
[\Delta_{23}+\Delta_{32}] \Big\rangle $}\\

\multicolumn{8}{l}{
${\rm H}^2({\mathfrak N}_{14}^{1})={\rm H}^2_{com}({\mathcal N}_{14}^{1})\oplus
\Big\langle [\Delta_{21}], [\Delta_{32}], [\Delta_{31}+\Delta_{42}]
\Big\rangle $}\\

\hline

${\mathfrak N}_{15}$ $:$ \quad  $e_1e_2 = e_4$ \quad $e_2e_1 = -e_4$ \quad $e_3e_3 = e_4$  \\

\multicolumn{8}{l}{
${\rm H}^2({\mathfrak N}_{15})=
\Big\langle   [\Delta_{11}], [\Delta_{13}], [\Delta_{21}], [\Delta_{22}],[\Delta_{23}], [\Delta_{31}],[\Delta_{32}],[\Delta_{33}]
\Big\rangle $}\\
\hline

\end{longtable}

\subsubsection{Central extensions of ${\mathfrak N}_{01}$}
	Let us use the following notations:
	\begin{longtable}{lllllll} 
	$\nabla_1 = [\Delta_{12}+\Delta_{21}],$ & $\nabla_2 = [\Delta_{13}+\Delta_{31}],$ & $\nabla_3 = [\Delta_{14}+\Delta_{41}],$ & $\nabla_4 = [\Delta_{33}],$ & $\nabla_5 = [\Delta_{34}+\Delta_{43}],$ \\
$\nabla_6 = [\Delta_{44}],$ & $\nabla_7 = [\Delta_{21}],$ &$\nabla_8 = [\Delta_{31}],$ & $\nabla_9 = [\Delta_{41}],$ & $\nabla_{10}= [\Delta_{43}]$ 
	\end{longtable}	
	
Take $\theta=\sum\limits_{i=1}^{10}\alpha_i\nabla_i\in {\rm H^2}({\mathfrak N}_{01}).$
	The automorphism group of ${\mathfrak N}_{01}$ consists of invertible matrices of the form
	$$\phi=
	\begin{pmatrix}
	x &  0  & 0 & 0\\
	q &  x^2& r & u\\
	w &  0  & t & k\\
    z &  0  & y & l
	\end{pmatrix}.
	$$
	Since
	$$
	\phi^T\begin{pmatrix}
	0 & \alpha_1   & \alpha_2 & \alpha_3\\
	\alpha_1 +\alpha_7  & 0 & 0 & 0\\
	\alpha_2+\alpha_8 &  0  & \alpha_4 & \alpha_5\\
	\alpha_3+\alpha_9 &  0  & \alpha_5+\alpha_{10} & \alpha_6
	\end{pmatrix} \phi=	\begin{pmatrix}
	\alpha^* & \alpha_1^*   & \alpha_2^* & \alpha_3^*\\
	\alpha_1^* +\alpha_7^*  & 0 & 0 & 0\\
	\alpha_2^*+\alpha_8^* &  0  & \alpha_4^* & \alpha_5^*\\
	\alpha_3^*+\alpha_9^* &  0  & \alpha_5^*+\alpha_{10}^* & \alpha_6^*
	\end{pmatrix},
	$$
	 we have that the action of ${\rm Aut} ({\mathfrak N}_{01})$ on the subspace
$\langle \sum\limits_{i=1}^{10}\alpha_i\nabla_i  \rangle$
is given by
$\langle \sum\limits_{i=1}^{10}\alpha_i^{*}\nabla_i\rangle,$
where
\begin{longtable}{lcl}
$\alpha^*_1$&$=$&$x^3 \alpha _1,$ \\
$\alpha^*_2$&$=$&$r x \alpha _1+y \left(x \alpha _3+w \alpha _5+z \alpha _6\right)+t \left(x \alpha _2+w \alpha _4+z \left(\alpha _5+\alpha _{10}\right)\right),$ \\
$\alpha^*_3$&$=$&$u x \alpha _1+l \left(x \alpha _3+w \alpha _5+z \alpha _6\right)+k \left(x \alpha _2+w \alpha _4+z \left(\alpha _5+\alpha _{10}\right)\right),$ \\
$\alpha_4^*$&$=$&$t^2 \alpha _4+y \left(2 t \alpha _5+y \alpha _6+t \alpha _{10}\right),$\\
$\alpha_5^*$&$=$&$k t \alpha _4+(l t+k y) \alpha _5+y \left(l \alpha _6+k \alpha _{10}\right)$\\
$\alpha_6^*$&$=$&$k^2 \alpha _4+l \left(2 k \alpha _5+l \alpha _6+k \alpha _{10}\right),$\\
$\alpha_7^*$&$=$&$x^3 \alpha _7,$\\
$\alpha_8^*$&$=$&$r x \alpha _7+t x \alpha _8+x y \alpha _9+w y \alpha _{10}-t z \alpha _{10},$\\
$\alpha_9^*$&$=$&$u x \alpha _7+k x \alpha _8+l x \alpha _9+l w \alpha _{10}-k z \alpha _{10},$\\
$\alpha_{10}^*$&$=$&$(l t-k y) \alpha _{10}.$\\
\end{longtable}

 We are interested only in the cases with 
 \begin{center}
$(\alpha_1,\alpha_7)\neq (0,0),$ 
$(\alpha_2,\alpha_4,\alpha_5,\alpha_8, \alpha_{10}) \neq (0,0,0,0,0),$\\
$(\alpha_3,\alpha_5,\alpha_6,\alpha_9, \alpha_{10}) \neq (0,0,0,0,0),$
$(\alpha_7,\alpha_8,\alpha_9, \alpha_{10}) \neq (0,0,0,0).$ 
 \end{center} 

\begin{enumerate}
    \item $\alpha_1=0,\ \alpha_{7}\neq0,$ then choosing $r=-\frac{t x \alpha _8+x y \alpha _9+(w y-t z) \alpha _{10}}{x \alpha _7},$ $u=-\frac{k x \alpha _8+l x \alpha _9+(l w-k z) \alpha _{10}}{x \alpha _7},$ we have $\alpha_8^*=\alpha_9^*=0.$ 
  
  The family of orbits $\langle\alpha_4\nabla_4+\alpha_5\nabla_5+\alpha_6\nabla_6+\alpha_{10}\nabla_{10}\rangle$ gives us characterised structure of three dimensional ideal whose a one dimensional extension of two dimensional subalgebra with basis $\{e_3, e_4\}.$ 
Let us remember the classification of algebras of this type.

\begin{longtable}{ll llllllllllll}
${\mathcal N}^{3*}_{01}$ &:&  $e_1 e_1 = e_2$\\

${\mathcal N}^{3*}_{02}$ &:&  $e_1 e_1 = e_3$ &  $e_2 e_2=e_3$ \\

${\mathcal N}^{3*}_{03}$ &:&   $e_1 e_2=e_3$ & $e_2 e_1=-e_3$   \\

${\mathcal N}^{3*}_{04}(\lambda)$ &:&
$e_1 e_1 = \lambda e_3$  & $e_2 e_1=e_3$  & $e_2 e_2=e_3$   \\

\end{longtable}

Using the classification of three dimensional nilpotent algebras, we may consider following cases. 
   \begin{enumerate}
	\item $\alpha_4=\alpha_5=\alpha_6=\alpha_{10}=0,$ i.e., three dimensional ideal is abelian. Then 
	we may suppose $\alpha_2 \neq 0$ and choosing $y=0,$ $l=\alpha_2,$ $k =-\alpha_3,$ we obtain that  $\alpha_3^*=0,$ which implies $(\alpha_3^*,\alpha_5^*,\alpha_6^*, \alpha_9^*, \alpha_{10}^*) = (0,0,0,0,0).$ Thus, in this case we do not have new algebras. 
	
	\item $\alpha_4=1,$ $\alpha_5=\alpha_6=\alpha_{10}=0,$ i.e., three dimensional ideal is isomorphic to ${\mathcal N}_{01}^{3*}$. Then $\alpha_3 \neq 0$ and choosing $x=1,$ $t=1,$ $k=0,$ $y=0,$ $w = -\alpha_2,$ $l = \frac{\alpha_7}{\alpha_3}$ and  $t=\sqrt{\alpha_7},$ we have the representative 
	$\langle \nabla_3+ \nabla_4+\nabla_7\rangle.$

\item $\alpha_4=\alpha_6=1,$ $\alpha_5=\alpha_{10}=0,$ i.e., three dimensional ideal is isomorphic to ${\mathcal N}_{02}^{3*}$. Then choosing 
$x=\frac{1}{\sqrt[3]{\alpha_7}},$ $k=y=0,$ $l=t=1,$ $w = -\frac{\alpha_2}{\sqrt[3]{\alpha_7}},$ $z = -\frac{\alpha_3}{\sqrt[3]{\alpha_7}},$ 
we have the representative 
	$\langle \nabla_4+ \nabla_6+\nabla_7\rangle.$
	
\item $\alpha_4= \alpha_6=0,$ $\alpha_5=1, \alpha_{10}=-2, $ i.e., three dimensional ideal is isomorphic to ${\mathcal N}_{03}^{3*}$.  Then choosing 
$x=\frac{1}{\sqrt[3]{\alpha_7}},$ $k=y=0,$ $l=t=1,$ $w = -\frac{\alpha_3}{\sqrt[3]{\alpha_7}},$ $z = \frac{\alpha_2}{\sqrt[3]{\alpha_7}},$ 
we have the representative 
	$\langle \nabla_5+ \nabla_7-2\nabla_{10}\rangle.$
	
	\item $\alpha_4=\lambda, \alpha_5=0, \alpha_6=1, \alpha_{10}=1,$ i.e., three dimensional ideal is isomorphic to ${\mathcal N}_{04}^{3*}(\lambda)$. 
	\begin{enumerate}

\item If  $\lambda \neq 0,$ then choosing $x=\frac{1}{\sqrt[3]{\alpha_7}},$ $k=0,$ $y=0,$ 
$l=t=1,$
$z = \frac{\alpha_3}{\sqrt[3]{\alpha_7}},$
 and 
$w = \frac{\alpha_2-\alpha_3}{\lambda \sqrt[3]{\alpha_7}},$ we have the family of representatives 
	$\langle \lambda \nabla_4+ \nabla_6+\nabla_7+\nabla_{10}\rangle_{\lambda\neq 0}.$

	\item If $\lambda = 0$  
and  $\alpha_2=\alpha_3,$ then choosing  
$x=\frac{1}{\sqrt[3]{\alpha_7}},$ $k=0,$ $y=0,$ 
$l=t=1$ and
$z = \frac{\alpha_3}{\sqrt[3]{\alpha_7}},$
   we have the representative 
	$\langle \nabla_6+\nabla_7+\nabla_{10}\rangle.$

\item If $\lambda = 0$  
and  $\alpha_2\neq\alpha_3,$ then choosing  
$x=\frac{(\alpha_2-\alpha_3)^2}{\alpha_7},$ $k=0,$ $y=0,$ 
$l=t=\frac{(\alpha_2-\alpha_3)^3}{\alpha_7}$ and
$z = -\frac{\alpha_3(\alpha_2-\alpha_3)^2}{\alpha_7},$
   we have the representative 
	$\langle \nabla_2+ \nabla_6+\nabla_7+\nabla_{10}\rangle.$ 
	
 	\end{enumerate}
	
 \end{enumerate}
 
  \item $\alpha_1\neq 0,$ then choosing 
 \begin{longtable}{lcl}
  $r$&$=$&$ - \frac{t x \alpha _2+x y \alpha _3+t w \alpha _4+w y \alpha _5+t z \alpha _5+y z \alpha _6+t z \alpha _{10}}{x \alpha _1},$\\ 
  $u$&$ =$&$- \frac{k x \alpha _2+l x \alpha _3+k w \alpha _4+l w \alpha _5+k z \alpha _5+l z \alpha _6+k z \alpha _{10}}{x \alpha _1},$ 
\end{longtable}  
  we have $\alpha_2^*=\alpha_3^*=0.$ 
   
   \begin{enumerate}
	\item $\alpha_4=\alpha_5=\alpha_6=\alpha_{10}=0,$ i.e., three dimensional ideal is abelian. Then 
	we may suppose $\alpha_8 \neq 0$ and choosing $y=0,$ $l=\alpha_8,$ $k =-\alpha_9,$ we obtain that  $\alpha_9^*=0,$ which implies $(\alpha_3^*,\alpha_5^*,\alpha_6^*, \alpha_9^*, \alpha_{10}^*) = (0,0,0,0,0).$ Thus, in this case we do not have new algebras.

	\item $\alpha_4=1,$ $\alpha_5=\alpha_6=\alpha_{10}=0,$ i.e., three dimensional ideal is isomorphic to ${\mathcal N}_{01}^{3*}$. 
	Then  $\alpha_9\neq 0,$ and choosing  
$x=1,$ $k=0,$ $t=\sqrt{\alpha_1},$ $y=-\frac{\sqrt{\alpha_1}\alpha_8}{\alpha_9},$ 
$l=\frac{\alpha_1}{\alpha_9}$ and $w=0,$
   we have the family of representatives 
	$\langle \nabla_1+ \nabla_4+\alpha \nabla_7+\nabla_{9}\rangle.$
	
\item $\alpha_4=\alpha_6=1,$ $\alpha_5=\alpha_{10}=0,$ i.e., three dimensional ideal is isomorphic to ${\mathcal N}_{02}^{3*}$. 
	\begin{enumerate}
	\item $\alpha_7=0,$ then  $\alpha _8^2+\alpha _9^2\neq 0,$ then choosing $x=\frac{\alpha _8^2+\alpha _9^2}{\alpha _1},$ $t=\frac{\alpha _8(\alpha _8^2+\alpha _9^2)}{\alpha _1},$ $y=\frac{\alpha _9(\alpha _8^2+\alpha _9^2)}{\alpha _1},$
	$l=\frac{\alpha _8(\alpha _8^2+\alpha _9^2)}{\alpha _1},$
	$k=-\frac{\alpha _9(\alpha _8^2+\alpha _9^2)}{\alpha _1},$ we have the representative 
	$\langle \nabla_1+ \nabla_4+\nabla_6+\nabla_{8}\rangle.$
	
	\item $\alpha_7=0,$ $\alpha _8^2+\alpha _9^2= 0,$ i.e., $\alpha_9 = \pm i \alpha _8 \neq 0,$ 
	then choosing $x=\sqrt{\alpha_8},$ 
	$t=\frac{\alpha_1} 2,$ $y=\pm \frac{\alpha_1} {2i},$
	$l=\pm ix\alpha_8,$ $k=x\alpha_8,$
	have the representative 
	$\langle \nabla_1+ \nabla_5+\nabla_8\rangle.$
	
	\item $\alpha_7\neq0,$ then choosing $x=1,$ $y=k=0,$
	$t=l=\sqrt{\alpha_7},$
	$z=\frac{\alpha _1 \alpha _9}{\alpha _7},$ $w=\frac{\alpha _1 \alpha _8}{\alpha _7}$
	we have the family of representatives 
	$\langle \alpha \nabla_1+ \nabla_4+\nabla_6+\nabla_{7}\rangle_{\alpha\neq0}.$
	\end{enumerate}
\item $\alpha_4= \alpha_6=0,$ $\alpha_5=1, \alpha_{10}=-2, $ i.e., three dimensional ideal is isomorphic to ${\mathcal N}_{03}^{3*}$.
	\begin{enumerate}
	\item $2 \alpha _1+\alpha _7\neq 0,$ then choosing $x=1,$ $y=k=0,$  $t=\sqrt{\alpha_1},$ $l=1,$ $z-\frac{ \alpha _1 \alpha _8}{2 \alpha _1+\alpha _7}$ and $w=\frac{ \alpha _1 \alpha _9}{2 \alpha _1+\alpha _7},$ we have the family of representatives 
	$\langle  \nabla_1+ \nabla_5+\alpha\nabla_{7}-2\nabla_{10}\rangle_{\alpha\neq-2}.$
	\item $2 \alpha _1+\alpha _7= 0,$ then in case of $(\alpha_8,\alpha_9) = (0,0),$ we have the representative 
	$\langle  \nabla_1+ \nabla_5-2\nabla_{7}-2\nabla_{10}\rangle$ and case of $(\alpha_8,\alpha_9) \neq (0,0),$ without loss of generality we may assume $\alpha_8 \neq 0$ and choosing $x=1,$ $y=0,$
	$l=\alpha_8,$ $k=-\alpha_9,$
	$t=\frac{\alpha_1}{\alpha_8},$  we have the representative 
	$\langle  \nabla_1+ \nabla_5-2\nabla_{7}+\nabla_{8}-2\nabla_{10}\rangle.$

\end{enumerate}
\item $\alpha_4=\lambda, \alpha_5=0, \alpha_6=1, \alpha_{10}=1,$ i.e., three dimensional ideal is isomorphic to ${\mathcal N}_{04}^{3*}(\lambda)$. 
	\begin{enumerate}

\item  $\alpha _1^2+\alpha _1 \alpha _7+\lambda \alpha _7^2 \neq 0,$ then choosing $x=1,$ $y=k=0$ and 
\begin{center}$t=l=\sqrt{\alpha_7},$
$z=\frac{ \alpha _1 \left(\alpha _1 \alpha _8+\alpha _4 \alpha _7 \alpha _9\right)}{\alpha _1^2+\alpha _1 \alpha _7+\alpha _4 \alpha _7^2},$ 
$w=\frac{\alpha _1 \left(\alpha _7 \left(\alpha _8-\alpha _9\right)+\alpha _1 \alpha _9\right)}{\alpha _1^2+\alpha _1 \alpha _7+\alpha _4 \alpha _7^2},$ 
\end{center} we have the representative
	$\langle \alpha \nabla_1+ \lambda \nabla_4+\nabla_{6}+\nabla_{7}+\nabla_{10}\rangle.$

\item  $\alpha _1^2+\alpha _1 \alpha _7+\lambda \alpha _7^2 = 0,$ then choosing $y=k=0,$ $w=\frac{z \alpha _7}{\alpha _1}-x \alpha _9,$ we have 
$\alpha_9^*=0,$ $\alpha_8^*=\frac{tx}{\alpha_1}(\alpha_1\alpha_8 - \lambda \alpha_7 \alpha_9).$ Thus, in this case we have the representatives $\langle \frac{-1 \pm \sqrt{1-4\lambda}}{2} \nabla_1+ \lambda \nabla_4+\nabla_{6}+\nabla_{7}+\nabla_{10}\rangle$  and $\langle \frac{- 1 \pm \sqrt{1-4\lambda}}{2} \nabla_1+ \lambda \nabla_4+\nabla_{6}+\nabla_{7}+\nabla_{8}+\nabla_{10}\rangle$ depending on $\alpha_1\alpha_8 = \lambda \alpha_7 \alpha_9$ or not.

\end{enumerate}
\end{enumerate}

\end{enumerate}

Summarizing all cases, we have the following distinct orbits
\begin{center}
$\langle \nabla_3+ \nabla_4+\nabla_7\rangle,$  
$\langle \nabla_5+ \nabla_7-2\nabla_{10}\rangle,$  
$\langle \nabla_2+ \nabla_6+\nabla_7+\nabla_{10}\rangle$ 
$\langle \nabla_1+ \nabla_4+\alpha \nabla_7+\nabla_{9}\rangle,$ 
$\langle \nabla_1+ \nabla_4+\nabla_6+\nabla_{8}\rangle,$ 
$\langle \nabla_1+ \nabla_5+\nabla_8\rangle,$ 
	$\langle \alpha \nabla_1+ \nabla_4+\nabla_6+\nabla_{7}\rangle,$ 
$\langle  \nabla_1+\nabla_5+\alpha\nabla_{7}-2\nabla_{10}\rangle,$ \\
$\langle  \nabla_1+ \nabla_5-2\nabla_{7}+\nabla_{8}-2\nabla_{10}\rangle,$  
$\langle \alpha \nabla_1+ \lambda \nabla_4+\nabla_{6}+\nabla_{7}+\nabla_{10}\rangle,$ \\
$\langle \frac{-1 \pm \sqrt{1-4\lambda}}{2} \nabla_1+ \lambda\nabla_4+\nabla_{6}+\nabla_{7}+\nabla_{8}+\nabla_{10}\rangle,$
\end{center}
which gives the following new algebras (see section \ref{secteoA}):

\begin{center}

${\rm N}_{01},$ 
${\rm N}_{02},$  
${\rm N}_{03},$
${\rm N}_{04}^{\alpha},$
${\rm N}_{05},$
${\rm N}_{06},$ 
${\rm N}_{07}^{\alpha},$ 
${\rm N}_{08}^{\alpha},$  
${\rm N}_{09},$  
${\rm N}_{10}^{\alpha, \lambda},$  
${\rm N}_{11}^{\lambda},$ 
${\rm N}_{12}^{\lambda\neq \frac{1}{4}}.$ 

\end{center}

\subsubsection{Central extensions of ${\mathfrak N}_{02}$}
	Let us use the following notations:
	\begin{longtable}{lll} 
	$\nabla_1 = [\Delta_{12}+\Delta_{21}],$ & $\nabla_2 = [\Delta_{13}+\Delta_{31}],$ & $\nabla_3 = [\Delta_{24}+\Delta_{42}],$ \\ $\nabla_4 = [\Delta_{21}],$ & $\nabla_5 = [\Delta_{31}],$ & $\nabla_6 = [\Delta_{42}].$
	\end{longtable}	
	
Take $\theta=\sum\limits_{i=1}^{6}\alpha_i\nabla_i\in {\rm H^2}({\mathfrak N}_{02}).$
	The automorphism group of ${\mathfrak N}_{02}$ consists of invertible matrices of the form
		$$\phi_1=
	\begin{pmatrix}
	x &  0  & 0 & 0\\
	0 &  y  & 0 & 0\\
	z &  u  & x^2 & 0\\
    t &  v  & 0 & y^2
	\end{pmatrix}, \quad \phi_2=
	\begin{pmatrix}
	0 &  x  & 0 & 0\\
	y &  0  & 0 & 0\\
	z &  u  & 0 & x^2\\
    t &  v  & y^2 & 0
	\end{pmatrix}.
	$$
	Since
	$$
	\phi_1^T\begin{pmatrix}
	0 & \alpha_1   & \alpha_2 & 0\\
	\alpha_1 +\alpha_4  & 0 & 0 & \alpha_3\\
	\alpha_2+\alpha_5 &  0  & 0 & 0\\
	0 &  \alpha_3+\alpha_6  & 0 & 0
	\end{pmatrix} \phi_1=	\begin{pmatrix}
	\alpha^* & \alpha_1^*   & \alpha_2^* & 0\\
	\alpha_1^* +\alpha_4^*  & \alpha^{**} & 0 & \alpha_3^*\\
	\alpha_2^*+\alpha_5^* &  0  & 0 & 0\\
	0 &  \alpha_3^*+\alpha_6^*  & 0 & 0
	\end{pmatrix},
	$$
	 we have that the action of ${\rm Aut} ({\mathfrak N}_{02})$ on the subspace
$\langle \sum\limits_{i=1}^{10}\alpha_i\nabla_i  \rangle$
is given by
$\langle \sum\limits_{i=1}^{10}\alpha_i^{*}\nabla_i\rangle,$
where
\begin{longtable}{lcllcllcl}
$\alpha^*_1$&$=$&$x y \alpha _1+u x \alpha _2+t y \left(\alpha _3+\alpha _6\right),$ &
$\alpha^*_2$&$=$&$x^3 \alpha _2,$ &
$\alpha^*_3$&$=$&$y^3 \alpha _3,$ \\
$\alpha_4^*$&$=$&$x y \alpha _4+u x \alpha _5-t y \alpha _6,$&
$\alpha_5^*$&$=$&$x^3 \alpha _5,$&
$\alpha_6^*$&$=$&$y^3 \alpha _6.$\\
\end{longtable}

 We are interested only in the cases with 
 \begin{center}
$(\alpha_3,\alpha_6)\neq (0,0),$  $(\alpha_2,\alpha_5)\neq (0,0),$ $(\alpha_4,\alpha_5,\alpha_6) \neq (0,0,0).$ 
 \end{center} 

\begin{enumerate}
    \item $(\alpha_5,\alpha_6)=(0,0),$ then $\alpha_2\alpha_3\alpha_4\neq0 $ and by choosing 
    $ x=\frac{\alpha_4}{\sqrt[3]{\alpha_2^2\alpha_3}}, $ 
    $ y=\frac{\alpha_4}{\sqrt[3]{\alpha_2\alpha_3^2}}$ 
    and 
        $t=-\frac{x \left(y \alpha _1+u \alpha _2\right)}{y \alpha _3},$ 
we have the representative $ \left\langle \nabla_2+\nabla_3+\nabla_4\right\rangle;$

    \item $(\alpha_5,\alpha_6)\neq (0,0),$ then without loss of generality (maybe with an action of a suitable $\phi_2$), we can suppose $\alpha_5\neq0$ and choosing $u=\frac{t y \alpha _6-x y \alpha _4}{x \alpha _5},$ we have $\alpha_4^*=0$.
    \begin{enumerate}
        \item $\alpha_3\alpha_5+(\alpha_2+\alpha_5)\alpha_6=0,$ then $\alpha_6\neq0.$ 
        \begin{enumerate}
        
            \item if $\alpha_1\neq0,$ then choosing $x=\frac{\alpha_1}{\sqrt[3]{\alpha_5^2\alpha_6}},$ 
            $ y=\frac{\alpha_1}{\sqrt[3]{\alpha_5\alpha_6^2}},$ we have the family of representatives
            $\left\langle\nabla_1+\alpha\nabla_2-(1+ \alpha)\nabla_3+ \nabla_5+\nabla_6\right\rangle;$
        
            \item if $\alpha_1=0,$ then choosing $x=y\sqrt[3]{\frac{\alpha_6}{\alpha_5}},$ we have the family if representatives
            $\left\langle \alpha\nabla_2-(1+ \alpha)\nabla_3+ \nabla_5+\nabla_6\right\rangle.$
        \end{enumerate}
        
        \item $\alpha_3\alpha_5+(\alpha_2+\alpha_5)\alpha_6\neq0,$ then choosing $t=-\frac{x \alpha _1 \alpha _5}{\alpha _3 \alpha _5+\left(\alpha _2+\alpha _5\right) \alpha _6},$ we have $\alpha_1^*=0.$

        \begin{enumerate}
            \item if $\alpha_6=0,$ then choosing $x=y\sqrt[3]{\frac{\alpha_3}{\alpha_5}},$ we have the representative
            $\left\langle\alpha\nabla_2+\nabla_3+ \nabla_5\right\rangle;$
            \item if $\alpha_6\neq0,$ then choosing $x=y\sqrt[3]{\frac{\alpha_6}{\alpha_5}},$ we have the representative
            $\left\langle \alpha\nabla_2+\beta\nabla_3+ \nabla_5+\nabla_6\right\rangle_{\beta\neq-(1+\alpha)}.$
        \end{enumerate}
    \end{enumerate}
\end{enumerate}

Summarizing all cases, we have the following distinct orbits 

\begin{center} 
$\langle \nabla_2+\nabla_3+\nabla_4\rangle,$ 
$\langle\nabla_1+\alpha\nabla_2-(1+ \alpha)\nabla_3+ \nabla_5+\nabla_6\rangle,$

$\langle\alpha\nabla_2+\nabla_3+ \nabla_5\rangle,$ 
$\langle \alpha\nabla_2+\beta\nabla_3+ \nabla_5+\nabla_6\rangle^{O(\alpha, \beta) \simeq O(\beta,\alpha)},$
\end{center}
which gives the following new algebras (see section \ref{secteoA}):

\begin{center}

${\rm N}_{13},$ 
${\rm N}_{14}^{\alpha},$ 
${\rm N}_{15}^{\alpha},$ 
${\rm N}_{16}^{\alpha, \beta}.$
 
\end{center}

\subsubsection{Central extensions of ${\mathfrak N}_{04}^0$}
	Let us use the following notations:
	\begin{longtable}{lllllll} 
	$\nabla_1 = [\Delta_{12}],$ & $\nabla_2 = [\Delta_{13}],$ & $\nabla_3 = [\Delta_{14}],$ & $\nabla_4 = [\Delta_{21}],$ & $\nabla_5 = [\Delta_{22}],$\\
    $\nabla_6 = [\Delta_{24}],$	& $\nabla_7 = [\Delta_{41}],$ & $\nabla_8 = [\Delta_{42}],$ & $\nabla_9 = [\Delta_{44}],$ & $\nabla_{10} = [\Delta_{31}+\Delta_{32}-\Delta_{23}],$ 
	\end{longtable}	
	
Take $\theta=\sum\limits_{i=1}^{10}\alpha_i\nabla_i\in {\rm H^2}({\mathfrak N}_{04}^0).$
	The automorphism group of ${\mathfrak N}_{04}^0$ consists of invertible matrices of the form
		$$\phi=
	\begin{pmatrix}
	x &  0  & 0 & 0\\
	y &  x+y  & 0 & 0\\
	z &  t  & x(x+y) & w\\
    u &  v  & 0 & r
	\end{pmatrix}. 
	$$
	Since
	$$
	\phi^T\begin{pmatrix}
	0 & \alpha_1   & \alpha_2 & \alpha_3\\
	\alpha_4  & \alpha_5 & -\alpha_{10} & \alpha_6\\
	\alpha_{10} &  \alpha_{10}  & 0 & 0\\
	\alpha_7 &  \alpha_8  & 0 & \alpha_9
	\end{pmatrix} \phi=	\begin{pmatrix}
	\alpha^{*} & \alpha_1^{*}+\alpha^{*}   & \alpha_2^{*} & \alpha_3^{*}\\
	\alpha_4^{*}  & \alpha_5^{*} & -\alpha_{10}^{*} & \alpha_6^{*}\\
	\alpha_{10}^{*} &  \alpha_{10}^{*}  & 0 & 0\\
	\alpha_7^{*} &  \alpha_8^{*}  & 0 & \alpha_9^{*}
	\end{pmatrix},
	$$
	 we have that the action of ${\rm Aut} ({\mathfrak N}_{04}^0)$ on the subspace
$\langle \sum\limits_{i=1}^{10}\alpha_i\nabla_i  \rangle$
is given by
$\langle \sum\limits_{i=1}^{10}\alpha_i^{*}\nabla_i\rangle,$
where
\begin{longtable}{lcl}
$\alpha^*_1$&$=$&$x^2\alpha_1+x(t+y-z)\alpha_2-xy\alpha_4-ux(\alpha_7-\alpha _8)-$\\
&& \multicolumn{1}{r}{$(u-v)(x\alpha_3+y\alpha_6+u\alpha_9)-y(t-z)\alpha_{10},$} \\
$\alpha^*_2$&$=$&$x (x+y)(x \alpha _2-y \alpha _{10}),$ \\
$\alpha^*_3$&$=$&$w x \alpha _2+r x \alpha _3+r y \alpha _6+r u \alpha _9-w y \alpha _{10},$ \\
$\alpha_4^*$&$=$&$u((x+y) \alpha _6+v \alpha _9)-(x+y) z \alpha _{10} +x((x+y) \alpha _4+v \alpha _7+t \alpha _{10})+$\\
&&\multicolumn{1}{r}{$y((x+y) \alpha _2+v \alpha _8+t \alpha _{10}),$}\\
$\alpha_5^*$&$=$&$(x+y)^2\alpha_5+v((x+y)\alpha_6+(x+y)\alpha_8+v\alpha_9),$\\
$\alpha_6^*$&$=$&$r (x+y) \alpha _6+r v \alpha _9-w (x+y) \alpha _{10},$\\
$\alpha_7^*$&$=$&$rx\alpha_7+ry\alpha_8+ru\alpha _9+w(x+y)\alpha _{10},$\\
$\alpha_8^*$&$=$&$r v \alpha _9+(x+y)(r \alpha _8+w \alpha _{10}),$\\
$\alpha_9^*$&$=$&$r^2 \alpha _9,$\\
$\alpha_{10}^*$&$=$&$x (x+y)^2 \alpha _{10}.$\\
\end{longtable}

Since we are interested only in the cases with 
\begin{center}$(\alpha_2,\alpha_{10})\neq (0,0),\quad  (\alpha_3,\alpha_6,\alpha_7,\alpha_8,\alpha_9) \neq (0,0,0,0,0),$
\end{center}
consider following subcases:

\begin{enumerate}
    \item $\alpha_{10}=0,$ then $\alpha_2\neq0$ and
  choosing  
  $w=-\frac{r \left(x \alpha _3+y \alpha _6+u \alpha _9\right)}{x \alpha _2}$ and
  \begin{center}$t=\frac{x z \alpha _2-x^2 \alpha _1+u x \alpha _3-v x \alpha _3+x y \alpha _4-x y \alpha _5+u y \alpha _6-v y \alpha _6+u x \alpha _7-u x \alpha _8+u^2 \alpha _9-u v \alpha _9}{x \alpha _2},$ \end{center} we have $\alpha_1^{*}=\alpha_3^{*}=0.$  
    \begin{enumerate}
        \item $\alpha_9\neq0,$ then choosing $u=-\frac{x \alpha _7+y \alpha _8}{\alpha _9},\ v=-\frac{(x+y) \alpha _8}{\alpha _9},$ we have $\alpha_7^*=\alpha_8^*=0.$ 
        \begin{enumerate}
            \item $\alpha_5=\alpha_4=\alpha_6=0,$ then choosing $r=\sqrt{\frac{\alpha_2}{\alpha_9}},$ we have the representative $ \left\langle \nabla_2+ \nabla_9\right\rangle;$ 
            
            \item $\alpha_5=\alpha_4=0,\ \alpha_6\neq0,$ then choosing $x=1,\ y=\frac{\alpha _2 \alpha _9-\alpha _6^2}{\alpha _6^2}, $ $ r=\frac{\alpha _2}{\alpha _6},$ we have the representative $ \left\langle \nabla_2+\nabla_6+ \nabla_9 \right\rangle;$
            
            \item $\alpha_5=0,\ \alpha_4\neq0,\ \alpha_6=0$ then choosing $x=\frac{ \alpha _4}{\alpha _2},\ y=0,$ $ r=\frac{\sqrt{\alpha_4^{3}}}{\alpha _2 \sqrt{\alpha _9}},$ we have the representative $ \left\langle \nabla_2+\nabla_4+ \nabla_9 \right\rangle;$ 
            
            \item $\alpha_5=0,\ \alpha_4\neq0,\ \alpha_6\neq0,$ then choosing $x=\frac{\alpha _4}{\alpha _2},$ $ y=\frac{\alpha _4 \left(\alpha _4 \alpha _9-\alpha _6^2\right)}{\alpha _2 \alpha _6^2},$ $ r=\frac{\alpha _4^2}{\alpha _2 \alpha _6},$ we have the representative $ \left\langle \nabla_2+\nabla_4+\nabla_6+\nabla_9 \right\rangle;$ 
            
            \item $\alpha_5\neq0,\ \alpha_4=\alpha_5,$ then choosing $x=1,\ y= \frac{\alpha _2-\alpha _5}{\alpha _5},$ $ r=\frac{\alpha _2}{\sqrt{\alpha _5\alpha _9}},$ we have the family of  representatives $ \left\langle \nabla_2+\nabla_4+\nabla_5+\alpha\nabla_6+ \nabla_9 \right \rangle;$ 
            
            \item $\alpha_5\neq0,\ \alpha_4\neq\alpha_5,$ then choosing $x=\frac{\alpha _5-\alpha _4}{\alpha _2},$ $ y=\frac{\alpha _4 \left(\alpha _4-\alpha _5\right)}{\alpha _2 \alpha _5},$ $ r=\frac{(\alpha _4-\alpha _5)^2}{\alpha _2 \sqrt{\alpha_5\alpha _9}}$  we have the family of representatives 
            $ \left\langle \nabla_2+\nabla_5+\alpha\nabla_6+\nabla_9\right \rangle;$
        \end{enumerate}

\item $\alpha_9=0,\ \alpha_8\neq0,\ \alpha_7=\alpha_8$, then choosing $v=-\frac{x \alpha _4+y \alpha _5+u \alpha _6}{\alpha _8},$ we have $\alpha_4^*=0.$ 
    \begin{enumerate}
         \item $\alpha_6=-\alpha_8, \alpha_5=0,$ then choosing $x=1,\ y=0,\ r=\frac{\alpha_2}{\alpha_5},$ we have the representative $\left\langle \nabla_2-\nabla_6 +\nabla_7+\nabla_8\right\rangle;$
        \item $\alpha_6=-\alpha_8, \alpha_5\neq0,$ then choosing $x=1,\ y=\frac{\alpha _2-\alpha _5}{\alpha _5},\ r=\frac{\alpha_2}{\alpha_8},$ we have the representative $\left\langle \nabla_2+\nabla_5 -\nabla_6 +\nabla_7+\nabla_8\right\rangle;$
        \item $\alpha_6\neq-\alpha_8$, $\alpha_6=0, \ \alpha_5=0,$ then choosing $x=1,\ y=0,\ r=\frac{\alpha_2}{\alpha _8},$  we have the representative $ \left\langle \nabla_2+\nabla_7+ \nabla_8\right\rangle;$
        \item $\alpha_6\neq-\alpha_8$, $\alpha_6=0, \ \alpha_5\neq0,$ then choosing $x=\frac{\alpha _5}{\alpha _2},\ r=\frac{\alpha _5^2}{\alpha_2\alpha_8},$ we have the representative $\left\langle \nabla_2+ \nabla_5+\nabla_7 +\nabla_8 \right\rangle;$ 
        \item $\alpha_6\neq-\alpha_8$, $\alpha_6\neq0,$ then choosing $x=1,\ y=0,\ r=\frac{\alpha_2}{\alpha_8},$ we have the family of representatives $\left\langle \nabla_2+\alpha\nabla_6+\nabla_7+\nabla_8\right \rangle_{\alpha\neq0, -1}.$
          \end{enumerate}

    \item $\alpha_9=0,\ \alpha_8\neq0,\ \alpha_7\neq\alpha_8$, then choosing $y=-\frac{x \alpha _7}{\alpha _8},$ we have $\alpha_7^*=0.$ Hence,
    \begin{enumerate}
        \item $\alpha_6=-\alpha_8,\ \alpha_5=0,$ then choosing $x=1,\ u=\frac{\alpha _4}{\alpha _8},\ r=\frac{\alpha_2}{\alpha_8},$ we have the representative $\left\langle\nabla_2- \nabla_6+\nabla_8 \right\rangle;$
        \item $\alpha_6=-\alpha_8,\ \alpha_5\neq0,$ then choosing $x=\frac{\alpha_5}{\alpha_2},\ u=\frac{\alpha_4\alpha_5}{\alpha_2\alpha_8},\ r=\frac{\alpha_5^2}{\alpha_2\alpha_8},$ we have the representative $\left\langle\nabla_2+\nabla_5- \nabla_6+\nabla_8\right\rangle;$
        \item $\alpha_6\neq-\alpha_8,$ $\alpha_6=\alpha_4=0,$ then choosing $x=1,\ v=-\frac{\alpha_5}{\alpha_8},\ r=\frac{\alpha_2}{\alpha_8},$ we have the representative $\left\langle\nabla_2+\nabla_8 \right\rangle;$
        \item $\alpha_6\neq-\alpha_8,$ $\alpha_6=0,\ \alpha_4\neq0,$ then choosing $x=\frac{\alpha _4}{\alpha _2},\ v=-\frac{\alpha _4 \alpha _5}{\alpha _2 \alpha _8},\ r=\frac{\alpha _4^2}{\alpha _2 \alpha _8},$ we have the representative $\left\langle\nabla_2+\nabla_4 +\nabla_8 \right\rangle;$
        \item $\alpha_6\neq-\alpha_8,$ $\alpha_6\neq0,$ then choosing $x=1,\ u=-\frac{\alpha_4}{\alpha_6},\ v=-\frac{\alpha _5}{\alpha _6+\alpha _8},\ r=\frac{\alpha_2}{\alpha_8},$ we have the family of representatives $\left\langle\nabla_2+\alpha \nabla_6 +\nabla_8\right\rangle_{\alpha\neq 0, -1}.$
        \end{enumerate}
    \item $\alpha_9=\alpha_8=0,\ \alpha_7\neq0,$  then choosing $v=-\frac{(x+y) \left(x \alpha _4+y \alpha _5+u \alpha _6\right)}{x \alpha _7}$ we have $\alpha_4^*=0.$ Hance, 
    \begin{enumerate}
        \item $\alpha_6=\alpha_5=0,$ then choosing $x=1,\ y=0,\ r=\frac{\alpha_2}{\alpha_7},$ we have the representative $\left\langle\nabla_2+\nabla_7\right \rangle;$
        \item $\alpha_6=0,\ \alpha_5\neq0,$ then choosing $x=\frac{\alpha_5}{\alpha_2},\ y=0,\ r=\frac{\alpha_5^2}{\alpha_2\alpha_7},$ we have the representative $\left\langle\nabla_2+\nabla_5 +\nabla_7\right\rangle;$
        \item $\alpha_6\neq0,$ then choosing $x=1,\ y=\frac{\alpha_7-\alpha_6}{\alpha_6},\ u=\frac{\alpha_5}{\alpha_6},\ r=\frac{\alpha_2}{\alpha_6},$ we have the representative $\left\langle\nabla_2+\nabla_6 +\nabla_7\right\rangle.$
    \end{enumerate}
    \item $\alpha_9=\alpha_8=\alpha_7=0,\ \alpha_6\neq0,$ then choosing $x=1,$ 
    $ y=0,$ 
    $ u=-\frac{\alpha _4}{\alpha _6},$ 
    $ v=-\frac{\alpha_5}{\alpha _6},\ r=\frac{\alpha_2}{\alpha_6},$ we have the representative $\left\langle\nabla_2+\nabla_6\right\rangle.$
    \end{enumerate}
    
\item $\alpha_{10}\neq0,$ then choosing $t=-\frac{x (x+y) \alpha _4+y (x+y) \alpha _5+u x \alpha _6+u y \alpha _6+v x \alpha _7+v y \alpha _8+u v \alpha _9-x z \alpha _{10}-y z \alpha _{10}}{(x+y) \alpha _{10}}$ and  $ w=-\frac{r \left((x+y) \alpha _8+v \alpha _9\right)}{(x+y) \alpha _{10}},$ we have  $\alpha_{4}^*=\alpha_{8}^*=0$. Now we consider following subcases:
    \begin{enumerate}
    \item $\alpha_9\neq0,$ then  choosing $u=-\frac{(x+y) \alpha _6+2 x \alpha _7}{2 \alpha _9},\ v=-\frac{(x+y) \alpha _6}{2 \alpha _9},$ we have  $\alpha_{6}^*=\alpha_{7}^*=0$. Hence,
        \begin{enumerate}
    
        \item $\alpha_2=-\alpha_{10},\ \alpha_5=\alpha_1=0,\ \alpha_3=0,$ then choosing $ x=1,$ $y=0, $ $ r=\sqrt{\frac{\alpha_{10}}{\alpha_9}},$ we have the representative $\left\langle-\nabla_2+\nabla_9+\nabla_{10}\right \rangle;$

        \item $\alpha_2=-\alpha_{10},\ \alpha_5=\alpha_1=0,\ \alpha_3\neq0,$ then choosing $x=1,$ $ y=-1+\frac{\alpha_3}{\sqrt{\alpha_9\alpha_{10}}}, $ $  r=\frac{\alpha_3}{\alpha _9},$ we have the representative $\left\langle-\nabla_2+\nabla_3+\nabla_9+\nabla_{10}\right \rangle;$
        \item $\alpha_2=-\alpha_{10},\ \alpha_5=0,\ \alpha_1\neq0,$ then choosing $x=1,\ y=-1+\sqrt{\frac{\alpha _1}{\alpha _{10}}},\ r=\sqrt{\frac{\alpha _1}{\alpha _{9}}},$ we have the family of representatives $\left\langle\nabla_1-\nabla_2 +\alpha\nabla_3+\nabla_9+\nabla_{10}\right \rangle,$ 
 where $O(\alpha)\simeq O(-\alpha);$ 
        \item $\alpha_2=-\alpha_{10},\ \alpha_5\neq0,$ $ \alpha_1=\alpha_5,$ $ \alpha_3=0,$ then choosing $x=\frac{\alpha _5}{\alpha _{10}},\ y=0,\  r=\frac{\alpha_5 \sqrt{\alpha_5}}{\alpha_{10}\sqrt{\alpha _9}},$ we have the representative $\left\langle\nabla_1-\nabla_2+ \nabla_5+\nabla_9+\nabla_{10}\right \rangle;$
        \item $\alpha_2=-\alpha_{10},\ \alpha_5\neq0,$ $ \alpha_1=\alpha_5,$ $ \alpha_3\neq0,$ then choosing 
        $x=\frac{\alpha _5}{\alpha _{10}},$  $y=\frac{\alpha_3 \sqrt{\alpha_5}-\alpha_5 \sqrt{\alpha_9}} {\alpha_{10}\sqrt{\alpha _9}},$ $  r=\frac{\alpha _3 \alpha _5}{\alpha _9 \alpha _{10}},$ we have the representative $\left\langle\nabla_1-\nabla_2+\nabla_3+ \nabla_5+ \nabla_9 +\nabla_{10}\right\rangle;$

        \item $\alpha_2=-\alpha_{10},\ \alpha_5\neq0,\ \alpha_1\neq\alpha_5,$ then choosing $x=\frac{\alpha _5}{\alpha _{10}},$ $ y=\frac{-\alpha _5+\sqrt{\alpha_5 (-\alpha _1+\alpha_5)}}{\alpha_{10}},$ 
        $  r=\frac{\alpha_5 \sqrt{\alpha_5-\alpha_1}}{\alpha_{10}\sqrt{\alpha_9}},$ we have the 
 family of representatives $\left\langle-\nabla_2+\alpha \nabla_3+\nabla_5+ \nabla_9 +\nabla_{10}\right\rangle,$ where $O(\alpha)\simeq O(-\alpha);$ 
        \item $\alpha_2\neq-\alpha_{10},\ \alpha_5=\alpha_3=\alpha_1=0,$ then choosing $x=1,\ y=\frac{\alpha_2}{\alpha _{10}},\  r=\frac{\alpha _2+\alpha _{10}}{\sqrt{\alpha_9\alpha _{10}}},$ we have the representative $\left\langle\nabla_9+\nabla_{10}\right \rangle;$ 

        \item $\alpha_2\neq-\alpha_{10},\ \alpha_5=\alpha_3=0,\ \alpha_1\neq0,$ then choosing $x=\frac{\alpha _1 \alpha _{10}}{\left(\alpha _2+\alpha _{10}\right)^2},$  $y=\frac{\alpha_1\alpha_2}{(\alpha _2+\alpha _{10})^2},$ $ r=\frac{\alpha _1\alpha_{10}\sqrt{\alpha_1}}{\sqrt{\alpha _9} (\alpha _2+\alpha _{10})^2},$ we have the representative $\left\langle\nabla_1+\nabla_9+ \nabla_{10} \right\rangle;$ 
        \item $\alpha_2\neq-\alpha_{10},\ \alpha_5=0,\ \alpha_3\neq0,$ then choosing $x=\frac{\alpha _3^2 \alpha _{10}}{\alpha_9(\alpha _2+\alpha _{10})^2},$ $ y=\frac{\alpha_2 \alpha_3^2}{\alpha_9(\alpha_2+\alpha _{10})^2},$ $  r=\frac{\alpha _3^3 \alpha _{10}}{\alpha_9^2 (\alpha _2+\alpha _{10})^2},$ we have the family of representatives $\left\langle\alpha\nabla_1+\nabla_3+\nabla_9+ \nabla_{10} \right\rangle;$
        \item $\alpha_2\neq-\alpha_{10},\ \alpha_5\neq0,$ then choosing $x=\frac{\alpha _5}{\alpha _{10}},$ 
        $y=\frac{\alpha_2\alpha_5}{\alpha _{10}^2},$ 
        $r=\frac{\alpha _5\sqrt{\alpha_5}(\alpha _2+\alpha _{10})}{\sqrt{\alpha_9}\alpha_{10}^2},$ we have the 
 family of representatives $\left\langle\alpha\nabla_1+\beta\nabla_3+\nabla_5+\nabla_9+ \nabla_{10} \right\rangle,$ where $O(\alpha,\beta)\simeq O(\alpha,-\beta).$
        \end{enumerate}
    \item $\alpha_9=0,\ \alpha_6\neq0,$ then  choosing $v=-\frac{(x+y) \alpha _5}{\alpha _6},$ we have  $\alpha_{5}^*=0$. Hence,

    \begin{enumerate}
        \item $\alpha_2=-\alpha_{10},\ \alpha_3=\alpha_{6},\ \alpha_7=0,\ \alpha_1=0,$ then choosing $x=1,\  y=0,$ $r=\frac{\alpha_{10}}{\alpha_{6}},$ we have the representative $\left\langle-\nabla_2+\nabla_3+\nabla_6 +\nabla_{10}\right\rangle;$
        \item $\alpha_2=-\alpha_{10},\ \alpha_3=\alpha_{6},\ \alpha_7=0,\ \alpha_1\neq0,$ then choosing $x=1,$ $  y=-1+\sqrt{\frac{\alpha _1}{\alpha _{10}}},$ $ r=\frac{\sqrt{\alpha_1\alpha_{10}}}{\alpha _6},$ we have the representative $\left\langle\nabla_1-\nabla_2+\nabla_3 +\nabla_6 +\nabla_{10}\right\rangle;$

        \item $\alpha_2=-\alpha_{10},\ \alpha_3=\alpha_{6},\ \alpha_7\neq0,$ then choosing $x=1,\  y=\frac{\alpha_1 -\alpha_6}{\alpha_6},$ 
        $u=\frac{\alpha_1}{\alpha_7},$ 
        $r=\frac{\alpha _7 \alpha _{10}}{\alpha _6^2},$ we have the representative $\left\langle-\nabla_2+\nabla_3 +\nabla_6 +\nabla_7+\nabla_{10}\right\rangle;$
        \item $\alpha_2=-\alpha_{10},\ \alpha_3\neq\alpha_{6},\ \alpha _3+\alpha _7=\alpha _6,\ \alpha_1=0,$ then choosing $x=1,$ 
        $  y=-\frac{\alpha_3}{\alpha_6},$ 
        $ r=-\frac{\left(\alpha _3-\alpha _6\right) \alpha _{10}}{\alpha _6^2},$ we have the representative $\left\langle-\nabla_2+\nabla_6 +\nabla_7+\nabla_{10}\right\rangle;$

        \item $\alpha_2=-\alpha_{10},\ \alpha_3\neq\alpha_{6},\ \alpha _3+\alpha _7=\alpha _6,\ \alpha_1\neq0,$ then choosing $x=\frac{\alpha _1 \alpha _6^2}{\left(\alpha _3-\alpha _6\right){}^2 \alpha _{10}},$ $  y=-\frac{\alpha _1 \alpha_3\alpha_6}{\left(\alpha _3-\alpha _6\right){}^2 \alpha _{10}},$ $r=\frac{\alpha _1^2 \alpha _6^2}{(\alpha_6-\alpha_3)^3 \alpha _{10}},$ we have the representative  $\left\langle\nabla_1-\nabla_2+\nabla_6 +\nabla_7+\nabla_{10}\right\rangle;$
    
        \item $\alpha_2=-\alpha_{10},$ $ \alpha_3\neq\alpha_{6},$ $ \alpha _3+\alpha _7\neq\alpha _6,$ then choosing $x=1,$ 
        $  y=-\frac{\alpha_3}{\alpha _6},$ 
        $ u=\frac{\alpha _1}{\alpha _3-\alpha _6+\alpha _7},$ 
        $ r=-\frac{\left(\alpha _3-\alpha _6\right) \alpha _{10}}{\alpha _6^2},$ we have the family of representatives  $\left\langle-\nabla_2+\nabla_6 +\alpha\nabla_7 +\nabla_{10}\right\rangle_{\alpha\neq1};$
        \item $\alpha_2\neq-\alpha_{10},$ then choosing $y=\frac{x\alpha_2}{\alpha_{10}},$ we have $\alpha_2^*=0.$
        \begin{enumerate}
            \item $\alpha_3=- \alpha_7,\ \alpha_1=0,$ then choosing $x=1,$ $ r=\frac{\alpha_{10}}{\alpha _6},$ we have the family of representatives  $\left\langle\alpha\nabla_3+ \nabla_6-\alpha\nabla_7+\nabla_{10}\right\rangle;$

            \item $\alpha_3=-\alpha_{7},\  \alpha_1\neq0,$ then choosing $x=\frac{\alpha _1}{\alpha _{10}}, $ 
            $ r=\frac{\alpha _1^2}{\alpha _6 \alpha _{10}},$ we have the family of  representatives  $\left\langle\nabla_1+\alpha \nabla_3+\nabla_6-\alpha\nabla_7+\nabla_{10}\right\rangle
            ;$
            
            \item $\alpha_3\neq-\alpha_{7},$ then choosing $x=1,$ $ u=\frac{x \alpha _1}{\alpha _3+\alpha _7},$ $ r=\frac{\alpha_{10}}{\alpha _6},$ we have the family of representatives  $\left\langle\alpha\nabla_3+ \nabla_6+ \beta\nabla_7+\nabla_{10}\right\rangle_{\alpha+\beta\neq0}.$
        \end{enumerate}
    \end{enumerate}
    \item $\alpha_9=\alpha_6=0,\ \alpha_7\neq0.$

    \begin{enumerate}
        \item $\alpha_2=-\alpha_{10},\ \alpha_3=-\alpha_7,\ \alpha_5=\alpha_1=0,$ then choosing $x=1,\ y=0, \ r=\frac{\alpha_{10}}{\alpha_7},$ we have the representative $\left\langle-\nabla_2-\nabla_3+\nabla_7+\nabla_{10}\right\rangle;$
     
        \item $\alpha_2=-\alpha_{10},\ \alpha_3=-\alpha_{7},\ \alpha_5=0,\ \alpha_1\neq0,$ then choosing $x=1,$ 
        $ y=\sqrt{\frac{\alpha_{1}}{\alpha_{10}}}-1,$ 
        $  r=\frac{\alpha _1}{\alpha _7},$ we have the representative $\left\langle \nabla_1-\nabla_2-\nabla_3+\nabla_7+\nabla_{10}\right\rangle;$

        \item $\alpha_2=-\alpha_{10},\ \alpha_3=-\alpha_{7},\ \alpha_5\neq0,\ \alpha_1-\alpha_5=0,$ then choosing $x=\frac{\alpha _5}{\alpha _{10}},$ 
        $y=0,$ 
        $r=\frac{\alpha _5^2}{\alpha _7 \alpha _{10}},$ we have the representative $\left\langle \nabla_1-\nabla_2-\nabla_3+\nabla_5+\nabla_7+ \nabla_{10}\right\rangle;$

        \item $\alpha_2=-\alpha_{10},\ \alpha_3=-\alpha_{7},\ \alpha_5\neq0,\ \alpha_1\neq\alpha_5,$ then by choosing $x=\frac{\alpha _5}{\alpha _{10}},$ 
        $y=\frac{ \sqrt{\alpha _5(\alpha _5-\alpha _1)}-\alpha_5}{\alpha _{10}},$ 
        $ r=\frac{\alpha_5(\alpha _5-\alpha _1)}{\alpha _7 \alpha _{10}},$ we have the representative $\left\langle -\nabla_2 -\nabla_3+\nabla_5+\nabla_7+ \nabla_{10}\right\rangle;$
        
        \item $\alpha_2=-\alpha_{10},\ \alpha_3\neq-\alpha_{7},\ \alpha_5=0,$ then choosing $x=1,$ $ y=0,$ 
        $ v=\frac{u(\alpha_3+\alpha_7)-\alpha _1}{\alpha_3+\alpha_7},$ $  r=\frac{\alpha_{10}}{\alpha_7},$ we have the family of representatives $\left\langle -\nabla_2 +\alpha\nabla_3+\nabla_7+ \nabla_{10}\right\rangle_{\alpha\neq -1};$
        
        \item $\alpha_2=-\alpha_{10},\ \alpha_3\neq-\alpha_{7},\ \alpha_5\neq0,$ then choosing $x=\frac{\alpha _5}{\alpha _{10}},$ $y=0, $ $ u=0,$ $  v=\frac{-\alpha _1\alpha_5} {(\alpha_3+\alpha_7)\alpha_{10}},$ $   r=\frac{\alpha _5^2}{\alpha_7\alpha_{10}},$ we have the family of representatives $\left\langle -\nabla_2 +\alpha\nabla_3+\nabla_5+\nabla_7+ \nabla_{10}\right\rangle_{\alpha\neq -1};$
      
        \item $\alpha_2\neq-\alpha_{10},$ then choosing $y=\frac{x \alpha _2}{\alpha _{10}},$ we have $\alpha_2^*=0.$ If $\alpha_3=-\alpha_7,$ then choosing $v=\frac{\alpha_1}{\alpha_7},$  if $\alpha_3\neq-\alpha_7,$ then choosing $u=\frac{x \alpha _1+v \alpha _3}{\alpha _3+\alpha _7},$ we have $\alpha_1^*=0.$ Thus, we always can assume $\alpha_1^*=0,\ u=v=0.$  
        \begin{enumerate}
            \item $\alpha_5=0,$ then choosing $x=1,\   r=\frac{\alpha_{10}}{\alpha_7},$ we have the representative $\left\langle\alpha\nabla_3+\nabla_7+ \nabla_{10}\right\rangle;$
            \item $\alpha_5\neq0,$ then choosing $x=\frac{\alpha _5}{\alpha _{10}},\ r=\frac{\alpha _5^2}{\alpha _7 \alpha _{10}},$ we have the representative \\ $\left\langle\alpha\nabla_3+\nabla_5+\nabla_7+ \nabla_{10}\right\rangle.$
        \end{enumerate}
    \end{enumerate}
    \item $\alpha_9=\alpha_7=\alpha_6=0,$ then $\alpha_3\neq0,$ and choosing $u=\frac{\left(x^2 \alpha _1+v x \alpha _3+y (x+y) \alpha _5\right) \alpha _{10}-x y \alpha _2 \alpha _5}{x \alpha _3 \alpha _{10}}$, we have $\alpha_1^{*}=0$. 
    \begin{enumerate}
        \item $\alpha_2=-\alpha_{10},\ \alpha_5=0,$ then choosing $x=1,\ y=0,\ r=\frac{\alpha_{10}}{\alpha_3},$ we have the representative $\left\langle-\nabla_2+\nabla_3+\nabla_{10} \right\rangle;$
        \item $\alpha_2=-\alpha_{10},\ \alpha_5\neq0,$ then choosing $x=\frac{\alpha_{10}}{\alpha_3},\ y=0,\ r=\frac{\alpha _5^2}{\alpha _3 \alpha _{10}},$ we have the representative $\left\langle-\nabla_2+\nabla_3+\nabla_5+\nabla_{10} \right\rangle;$
        \item $\alpha_2\neq-\alpha_{10},\ \alpha_5=0,$ then choosing $x=1,\ y=\frac{ \alpha _2}{\alpha _{10}},\ r=\frac{\left(\alpha _2+\alpha _{10}\right){}^2}{\alpha _3 \alpha _{10}},$ we have the representative $\left\langle\nabla_3+\nabla_{10} \right\rangle;$
        \item $\alpha_2\neq-\alpha_{10},\ \alpha_5\neq0,$ then choosing $x=\frac{\alpha _5}{\alpha _{10}},\ y=\frac{ \alpha_2\alpha_5}{\alpha_{10}^2},\ r=\frac{\alpha _5^2(\alpha _2+\alpha _{10})^2}{\alpha _3 \alpha _{10}^3},$ we have the representative $\left\langle\nabla_3+\nabla_5+\nabla_{10} \right\rangle.$
    \end{enumerate}
  \end{enumerate}
\end{enumerate}

Summarizing all cases, we have the following distinct orbits 

\begin{center} 
$\langle \nabla_2+\nabla_9\rangle,$ 
$\langle\nabla_2+ \nabla_6+\nabla_9\rangle,$ 
$\langle\nabla_2+ \nabla_4+\nabla_9\rangle,$
$\langle \nabla_2+\nabla_4+\nabla_6+\nabla_9\rangle,$ 
$\langle \nabla_2+\nabla_4+\nabla_5+\alpha\nabla_6+\nabla_9\rangle,$  $\langle\nabla_2+\nabla_5+\alpha\nabla_6+\nabla_9\rangle,$
$\langle \nabla_2+\nabla_5+\nabla_7+\nabla_8\rangle,$ 
$\langle \nabla_2+\nabla_5-\nabla_6+\nabla_7+\nabla_8\rangle,$  $\langle\nabla_2+\alpha\nabla_6+\nabla_7+\nabla_8\rangle,$
$\langle \nabla_2+\nabla_4+\nabla_8\rangle,$ 
$\langle \nabla_2+\nabla_5-\nabla_6+\nabla_8\rangle,$  $\langle\nabla_2+\alpha\nabla_6+\nabla_8\rangle,$
$\langle \nabla_2+\nabla_7\rangle,$ 
$\langle\nabla_2+\nabla_5+\nabla_7\rangle,$ 
$\langle\nabla_2+\nabla_6+\nabla_7\rangle,$
$\langle \nabla_2+\nabla_6\rangle,$ 
$\langle -\nabla_2+\nabla_9+\nabla_{10}\rangle,$ 
$\langle -\nabla_2+\nabla_3+\nabla_9+\nabla_{10}\rangle,$ 
$\langle\nabla_1-\nabla_2+\alpha\nabla_3+\nabla_9+\nabla_{10}\rangle^{O(\alpha)\simeq O(-\alpha)},$  
$\langle\nabla_1-\nabla_2+\nabla_5+\nabla_9+\nabla_{10}\rangle,$ 
$\langle-\nabla_2+\alpha\nabla_3+\nabla_5+\nabla_9+\nabla_{10}\rangle^{O(\alpha)\simeq O(-\alpha)},$ 
$\langle\nabla_1-\nabla_2+\nabla_3+\nabla_5+\nabla_9+\nabla_{10}\rangle,$ 
$\langle\nabla_9+\nabla_{10}\rangle,$ 
$\langle\nabla_1+\nabla_9+\nabla_{10}\rangle,$ 
$\langle\alpha\nabla_1+\nabla_3+\nabla_9+\nabla_{10}\rangle,$ 
$\langle\alpha\nabla_1+\beta\nabla_3+\nabla_5+\nabla_9+\nabla_{10}\rangle^{O(\alpha,\beta)\simeq O(\alpha,-\beta)},$ 
$\langle-\nabla_2+\nabla_3+\nabla_6+\nabla_{10}\rangle,$ 
$\langle\nabla_1-\nabla_2+\nabla_3+\nabla_6+\nabla_{10}\rangle,$  
$\langle-\nabla_2+\nabla_3+\nabla_6+\nabla_7+\nabla_{10}\rangle,$ 
$\langle\nabla_1-\nabla_2+\nabla_6+\nabla_7+\nabla_{10}\rangle,$ 
$\langle-\nabla_2+\nabla_6+\alpha\nabla_7+\nabla_{10}\rangle,$ $\langle\nabla_1+\alpha\nabla_3+\nabla_6-\alpha\nabla_7+\nabla_{10}\rangle,$  $\langle\alpha\nabla_3+\nabla_6+\beta\nabla_7+\nabla_{10}\rangle,$ 
$\langle\nabla_1-\nabla_2-\nabla_3+\nabla_7+\nabla_{10}\rangle,$ 
$\langle-\nabla_2+\alpha\nabla_3+\nabla_7+\nabla_{10}\rangle,$ 
$\langle\nabla_1-\nabla_2-\nabla_3+\nabla_5+\nabla_7+\nabla_{10}\rangle,$  
$\langle-\nabla_2+\alpha\nabla_3+\nabla_5+\nabla_7+\nabla_{10}\rangle,$  $\langle\alpha\nabla_3+\nabla_7+\nabla_{10}\rangle,$  $\langle\alpha\nabla_3+\nabla_5+\nabla_7+\nabla_{10}\rangle,$  
$\langle-\nabla_2+\nabla_3+\nabla_{10}\rangle,$  
$\langle-\nabla_2+\nabla_3+\nabla_5+\nabla_{10}\rangle,$  $\langle\nabla_3+\nabla_{10}\rangle,$ 
$\langle\nabla_3+\nabla_5+\nabla_{10}\rangle.$  
\end{center}

which gives the following new algebras (see section \ref{secteoA}):

\begin{center}

${\rm N}_{17},$ 
${\rm N}_{18},$ 
${\rm N}_{19},$ 
${\rm N}_{20},$
${\rm N}_{21}^{\alpha},$
${\rm N}_{22}^{\alpha},$
${\rm N}_{23},$
${\rm N}_{24},$
${\rm N}_{25}^{\alpha},$
${\rm N}_{26},$
${\rm N}_{27},$
${\rm N}_{28}^{\alpha},$
${\rm N}_{29},$
${\rm N}_{30},$
${\rm N}_{31},$
${\rm N}_{32},$
${\rm N}_{33},$
${\rm N}_{34},$
${\rm N}_{35}^{\alpha},$
${\rm N}_{36},$
${\rm N}_{37}^{\alpha},$
${\rm N}_{38},$
${\rm N}_{39},$
${\rm N}_{40},$
${\rm N}_{41}^{\alpha},$
${\rm N}_{42}^{\alpha,\beta},$
${\rm N}_{43,}$
${\rm N}_{44},$
${\rm N}_{45},$
${\rm N}_{46},$
${\rm N}_{47}^{\alpha},$
${\rm N}_{48}^{\alpha},$
${\rm N}_{49}^{\alpha,\beta},$
${\rm N}_{50},$
${\rm N}_{51},$ 
${\rm N}_{52},$ 
${\rm N}_{53}^{\alpha},$ 
${\rm N}_{54}^{\alpha},$ 
${\rm N}_{55}^{\alpha},$ 
${\rm N}_{56},$ 
${\rm N}_{57},$ 
${\rm N}_{58},$ 
${\rm N}_{59}.$ 
\end{center}

\subsubsection{Central extensions of ${\mathfrak N}_{07}$}
	Let us use the following notations:
	\begin{longtable}{lllllll} 
	$\nabla_1 = [\Delta_{11}],$ & $\nabla_2 = [\Delta_{22}],$ & 
 $\nabla_3 = [\Delta_{23}-\Delta_{13}],$ \\
	$\nabla_4 = [\Delta_{24}],$ & $\nabla_5 = [\Delta_{32}-\Delta_{13}],$ & $\nabla_6 = [\Delta_{41}-\Delta_{14}],$
	\end{longtable}	
	
Take $\theta=\sum\limits_{i=1}^{6}\alpha_i\nabla_i\in {\rm H^2}({\mathfrak N}_{07}).$
	The automorphism group of ${\mathfrak N}_{07}$ consists of invertible matrices of the form
		$$\phi=
	\begin{pmatrix}
	x &  0  & 0 & 0\\
	0 &  x  & 0 & 0\\
	z &  u  & x^2 & 0\\
    t &  v  & 0 & x^2
	\end{pmatrix}.
	$$
	Since
	$$
	\phi^T\begin{pmatrix}
	\alpha_1& 0   & -\alpha_3-\alpha_5 & -\alpha_6\\
	0 & \alpha_2 & \alpha_3 & \alpha_4\\
	0 &  \alpha_5  & 0 & 0\\
	\alpha_6 &  0  & 0 & 0
	\end{pmatrix} \phi=	\begin{pmatrix}
	\alpha_1^* & \alpha^*   & -\alpha_3^*-\alpha_5^* & -\alpha_6^*\\
	\alpha^{**}  & -\alpha^*+\alpha_2^* & \alpha_3^* & \alpha_4^*\\
	0 &  \alpha_5^*  & 0 & 0\\
	\alpha_6^* &  0  & 0 & 0
	\end{pmatrix},
	$$
	 we have that the action of ${\rm Aut} ({\mathfrak N}_{07})$ on the subspace
$\langle \sum\limits_{i=1}^{6}\alpha_i\nabla_i  \rangle$
is given by
$\langle \sum\limits_{i=1}^{6}\alpha_i^{*}\nabla_i\rangle,$
where
\begin{longtable}{lcllcllcl}
$\alpha^*_1$&$=$&$x (x\alpha_1-z(\alpha _3+\alpha_5)),$ &
$\alpha^*_3$&$=$&$x^3 \alpha _3,$ &
$\alpha_5^*$&$=$&$x^3 \alpha _5$\\

$\alpha^*_2$&$=$&$x \left(x \alpha _2+v \alpha _4+z \alpha _5-v \alpha _6\right),$ &
$\alpha_4^*$&$=$&$x^3 \alpha _4,$&
$\alpha_6^*$&$=$&$x^3 \alpha _6.$\\
\end{longtable}

 We are interested only in the cases with 
 $(\alpha_3,\alpha_5)\neq (0,0),$  $(\alpha_4,\alpha_6)\neq (0,0).$

\begin{enumerate}
    \item $\alpha_5\neq0,$ then choosing $z=-\frac{x \alpha _2+v \alpha _4-v \alpha _6}{\alpha _5},$ we have $\alpha_2^*=0.$ Now we consider following subcases:
    \begin{enumerate}
        \item $\alpha_4\neq0$.
        \begin{enumerate}
            \item if $\alpha_3=-\alpha_5,\ \alpha_1=0,$ then we have the family of representatives $\left\langle-\nabla_3+ \alpha\nabla_4+\nabla_5+\beta\nabla_6\right\rangle_{\alpha\neq0};$

            \item if $\alpha_3=-\alpha_5,\ \alpha_1\neq0,$ then choosing $x=\frac{\alpha_1}{\alpha_5},$ we have the family of  representatives 
            $\left\langle\nabla_1-\nabla_3+ \alpha\nabla_4+\nabla_5+\beta\nabla_6\right\rangle_{\alpha\neq0};$
            \item if $\alpha_3\neq-\alpha_5,$ then choosing $u=\frac{-v\alpha_3(\alpha _4+\alpha _6)+\alpha _5(x \alpha _1-v(\alpha_4+\alpha_6))}{2(\alpha _3+\alpha _5)^2},$ we have the family of representatives 
            $\left\langle\gamma\nabla_3+ \alpha\nabla_4+\nabla_5+ \beta\nabla_6\right\rangle_{\alpha\neq0,\gamma\neq-1}.$
        \end{enumerate}
        \item $\alpha_4=0, \alpha_6\neq0$. 
        \begin{enumerate}
            \item if $\alpha_3=-\alpha_5,\ \alpha_1=0,$ then we have the family of representatives $\left\langle-\nabla_3+ \nabla_5 +\beta\nabla_6\right\rangle_{\beta\neq0};$
            \item if $\alpha_3=-\alpha_5,\ \alpha_1\neq0,$ then choosing $x=\frac{\alpha_1}{\alpha_5},$ we have the family of  representatives 
            $\left\langle\nabla_1-\nabla_3+ \nabla_5+ \beta\nabla_6\right\rangle_{\beta\neq0};$
        
            \item if $\alpha_3\neq-\alpha_5,$ then choosing $v=\frac{\alpha _5(x \alpha _1-2 u \alpha_5)-2 u \alpha _3^2-4 u \alpha _3 \alpha _5}{(\alpha _3+\alpha_5) \alpha _6},$ we have the family of representatives              $\left\langle\gamma\nabla_3+\nabla_5+ \beta\nabla_6 \right\rangle_{\gamma\neq -1, \beta\neq0}.$
        \end{enumerate}
    \end{enumerate}
    \item $\alpha_5=0,\ \alpha_3\neq0,$ then choosing $z=\frac{x \alpha _1}{\alpha _3},$ we have $\alpha_1=0.$ 
    \begin{enumerate}
        \item if $\alpha_4=\alpha_6,\ \alpha_2=0,$ then we have the family of representatives  $\left\langle\nabla_3+\alpha\nabla_4+\alpha\nabla_6 \right\rangle_{\alpha\neq0};$
        \item if $\alpha_4=\alpha_6,\ \alpha_2\neq0,$ then choosing $x=\frac{\alpha_2}{\alpha_3},$ we have the family of representatives $\left\langle\nabla_2+\nabla_3+\alpha\nabla_4+\alpha\nabla_6 \right\rangle_{\alpha\neq0};$
        \item if $\alpha_4\neq \alpha_6,$ then choosing $x=1,\ v=-\frac{\alpha _2}{\alpha _4-\alpha _6},$ we have the family of representatives  $\left\langle\nabla_3+\alpha\nabla_4+\beta\nabla_6 \right\rangle_{\alpha\neq\beta,\ (\alpha,\beta)\neq(0,0)}.$
    \end{enumerate}
    
\end{enumerate}

Summarizing all cases, we have the following distinct orbits 

\begin{center} 
$\langle\nabla_1-\nabla_3+\alpha\nabla_4+\nabla_5+\beta\nabla_6\rangle_{(\alpha,\beta)\neq(0,0)},$ $\langle\gamma\nabla_3+\alpha\nabla_4+\nabla_5+\beta\nabla_6\rangle_{(\alpha,\beta)\neq(0,0)},$ 
$\langle\nabla_3+\alpha\nabla_4+\beta\nabla_6\rangle_{(\alpha,\beta)\neq(0,0)},$ 
$\langle\nabla_2+\nabla_3+\alpha\nabla_4+\alpha\nabla_6\rangle_{\alpha\neq0}$
\end{center}
which gives the following new algebras (see section \ref{secteoA}):

\begin{center}
${\rm N}_{60}^{(\alpha,\beta)\neq(0,0)},$
${\rm N}_{61}^{(\alpha,\beta)\neq(0,0)},$
${\rm N}_{62}^{(\alpha,\beta)\neq(0,0)},$
${\rm N}_{63}^{\alpha\neq0}.$
\end{center}

\subsubsection{Central extensions of ${\mathfrak N}_{08}^{\alpha\neq1}$}
	Let us use the following notations:
	\begin{longtable}{lllllll} 
	$\nabla_1 = [\Delta_{12}],$ & $\nabla_2 = [\Delta_{21}],$ & $\nabla_3 = [\Delta_{13}-\alpha\Delta_{23}],$ \\
	$\nabla_4 = [\Delta_{14}-\Delta_{24}],$ & $\nabla_5 = [\Delta_{31}-\Delta_{13}],$ & $\nabla_6 = [\Delta_{42}-\Delta_{24}].$
	\end{longtable}	
	
Take $\theta=\sum\limits_{i=1}^{6}\alpha_i\nabla_i\in {\rm H^2}({\mathfrak N}_{08}^{\alpha\neq1}).$
	The automorphism group of ${\mathfrak N}_{08}^{\alpha\neq1}$ consists of invertible matrices of the form
	$$\phi_1=
	\begin{pmatrix}
	x &  0  & 0 & 0\\
	0 &  x  & 0 & 0\\
	t &  v  & x^2 & 0\\
    u &  w  & 0 & x^2
	\end{pmatrix}, \quad 
	\phi_2(\alpha \neq 0)=
	\begin{pmatrix}
	0 &  \alpha x  & 0 & 0\\
	x &  0  & 0 & 0\\
	t &  v  & 0 & -\alpha^2x^2\\
    u &  w  & -x^2 & 0
	\end{pmatrix}.
	$$
	Since
	$$
	\phi_1^T\begin{pmatrix}
	0 & \alpha_1   & \alpha_3-\alpha_5 & \alpha_4\\
	\alpha_2 & 0 & -\alpha\alpha_3 & -\alpha_4-\alpha_6\\
	\alpha_5 &  0  & 0 & 0\\
	0 &  \alpha_6  & 0 & 0
	\end{pmatrix} \phi_1=	\begin{pmatrix}
\alpha^* & \alpha_1^*+\alpha^{**}   & \alpha_3^*-\alpha_5^* & \alpha_4^*\\
\alpha_2^*-\alpha\alpha^* & -\alpha^{**} & -\alpha\alpha_3^* & -\alpha_4^*-\alpha_6^*\\
\alpha_5^* &  0  & 0 & 0\\
0 &  \alpha_6^*  & 0 & 0
	\end{pmatrix},
	$$
	 we have that the action of ${\rm Aut} ({\mathfrak N}_{08}^{\alpha\neq1})$ on the subspace
$\langle \sum\limits_{i=1}^{6}\alpha_i\nabla_i  \rangle$
is given by
$\langle \sum\limits_{i=1}^{6}\alpha_i^{*}\nabla_i\rangle,$
where
\begin{longtable}{lcllcllcl}
$\alpha^*_1$&$=$&$x(x \alpha _1+v(1-\alpha ) \alpha _3-v \alpha _5+u \alpha_6),$ &
$\alpha^*_3$&$=$&$x^3 \alpha _3,$ &
$\alpha_5^*$&$=$&$x^3 \alpha _5$\\

$\alpha^*_2$&$=$&$x \left(x \alpha _2-u(1-\alpha ) \alpha _4+v \alpha _5-u \alpha _6\right),$ &
$\alpha_4^*$&$=$&$x^3 \alpha _4,$ &
$\alpha_6^*$&$=$&$x^3 \alpha _6.$\\
\end{longtable}

 We are interested only in the cases with 
 $(\alpha_3,\alpha_5)\neq (0,0),$  $(\alpha_4,\alpha_6)\neq (0,0).$

\begin{enumerate}
    \item $\alpha_5 = \alpha_6= 0,$ then $\alpha_3\alpha_4\neq 0,$ and choosing $u=\frac{x \alpha _2}{(1-\alpha ) \alpha _4},$ $v=-\frac{x \alpha _1}{(1-\alpha ) \alpha _3},$ we have $\alpha_1^*=\alpha_2^*=0$ and  obtain the representative  $\left\langle\nabla_3+ \beta \nabla_4 \right\rangle_{\beta \neq0}.$
\item $(\alpha_5, \alpha_6) \neq (0,0),$ $\alpha\neq 0,$ then with an action of a suitable $\phi_2$, we can suppose $\alpha_5\neq0$ and choosing $v=\frac{u \left((1-\alpha ) \alpha _4+\alpha _6\right)-x \alpha _2}{\alpha _5},$ we can suppose $\alpha_2^*=0.$ Now we consider following
subcases:

 \begin{enumerate}
 
    \item $\alpha _4 \alpha _5+(\alpha -1) \alpha _3\alpha _4\neq \alpha _3\alpha _6,$ then choosing $u=\frac{x \alpha _1 \alpha _5}{(1-\alpha ) \left(\alpha _4 \alpha _5-\alpha _3 \left((1-\alpha ) \alpha _4+\alpha _6\right)\right)},$
    we have the family of representatives $\langle\beta\nabla_3+\gamma\nabla_4+\nabla_5+\delta\nabla_6 \rangle_{\gamma+(\alpha-1) \beta\gamma\neq\beta\delta},$

    \item $\alpha _4 \alpha _5+ (\alpha -1) \alpha _3\alpha _4=\alpha _3\alpha _6,$ $\alpha_1=0,$ then we have the family of representatives $\langle\beta\nabla_3+\gamma\nabla_4+\nabla_5+\delta\nabla_6 \rangle_{\gamma+(\alpha-1) \gamma=\beta\delta};$
   
    \item $\alpha _4 \alpha _5+ (\alpha -1) \alpha _3\alpha _4=\alpha _3\alpha _6,$ $\alpha_1\neq 0,$ 
    $\alpha_5\neq(1-\alpha)\alpha_3,$ 
    we have the family of representatives $\langle\nabla_1+\beta\nabla_3+\frac{\beta\delta}{(\alpha-1)\beta+1}\nabla_4+\nabla_5+\delta\nabla_6 \rangle_{\beta\neq \frac{1}{1-\alpha}};$

    \item $\alpha _4 \alpha _5+ (\alpha -1) \alpha _3\alpha _4=\alpha _3\alpha _6,$ $\alpha_1\neq 0,$ 
    $\alpha_5=(1-\alpha)\alpha_3,$
    then $\alpha_3\neq0$ and $\alpha_6=0.$
    Hence, we have the family of representatives $\langle\nabla_1+\frac{1}{1-\alpha}\nabla_3+\beta\nabla_4+\nabla_5\rangle.$

\end{enumerate}

\item $(\alpha_5, \alpha_6) \neq (0,0),$ $\alpha= 0.$ If $\alpha_5\neq0,$ then we obtain the previous cases. Thus, we consider the case of $\alpha_5=0.$ Then $\alpha_3\alpha_6 \neq 0$ and choosing $u=-\frac{x \alpha _1+v \alpha _3}{\alpha _6},$ we can suppose $\alpha_1^*=0.$ Now we consider following
subcases:

    \begin{enumerate}
    \item $\alpha_4\neq -\alpha_6,$ then choosing $v=-\frac{x \alpha _2 \alpha _6}{\alpha _3 \left(\alpha _4+\alpha _6\right)},$ we obtain $\alpha_2^*=0$ and obtain the family of representatives $\langle\beta\nabla_3+\gamma\nabla_4+\nabla_6 \rangle_{\alpha=0, \gamma \neq -1},$
    
    \item $\alpha_4=-\alpha_6,$ $\alpha_2=0,$ then  we have the family of representatives $\langle\beta\nabla_3-\nabla_4+\nabla_6 \rangle_{\alpha=0},$
    
    \item $\alpha_4=-\alpha_6,$ $\alpha_2\neq0,$ then  we have the family of  representatives $\langle\nabla_2+ \beta\nabla_3-\nabla_4+\nabla_6 \rangle_{\alpha=0}.$

    \end{enumerate}

\end{enumerate}

Summarizing all cases for  the algebra ${\mathfrak N}_{08}^{\alpha\neq1}$, we have the following distinct orbits:
\begin{longtable}{lll}
$\langle\nabla_3+\beta\nabla_4\rangle_{\alpha\neq 1, \beta\neq0},$  & $\langle\beta\nabla_3+\gamma\nabla_4+\nabla_5+\delta\nabla_6\rangle_{\alpha\neq 1},$\\

$\langle\nabla_1+\beta\nabla_3+\frac{\beta\delta}{(\alpha-1)\beta+1}\nabla_4+\nabla_5+\delta\nabla_6 \rangle_{\beta\neq \frac{1}{1-\alpha}, \ \alpha\neq 1},$ &
$\langle\nabla_1+\frac{1}{1-\alpha}\nabla_3+\beta\nabla_4+\nabla_5\rangle_{\alpha\neq 1},$ \\

$\langle\beta\nabla_3+\gamma\nabla_4+\nabla_6\rangle_{\alpha=0},$ & $\langle\nabla_2+\beta\nabla_3-\nabla_4+\nabla_6\rangle_{\alpha=0},$
 
\end{longtable}
which gives the following new algebras (see section \ref{secteoA}):

\begin{center}
${\rm N}_{64}^{\alpha\neq1, \beta\neq 0},$
${\rm N}_{65}^{\alpha\neq1,\beta,\gamma,\delta},$
${\rm N}_{66}^{\alpha\neq 1,\beta\neq \frac{1}{1-\alpha},\delta},$
${\rm N}_{67}^{\alpha\neq1,\beta},$
${\rm N}_{68}^{\beta,\gamma},$
${\rm N}_{69}^{\beta}.$
\end{center}

\subsubsection{Central extensions of ${\mathfrak N}_{08}^{1}$}
	Let us use the following notations:
	\begin{longtable}{llll}
 $\nabla_1 = [\Delta_{12}],$ & $\nabla_2 = [\Delta_{21}],$ & $\nabla_3 = [\Delta_{13}-\Delta_{23}],$ &
$\nabla_4 = [\Delta_{14}-\Delta_{24}],$ \\ $\nabla_5 = [\Delta_{31}-\Delta_{13}],$ & $\nabla_6 = [\Delta_{42}-\Delta_{24}],$ & 
\multicolumn{2}{l}{$\nabla_7 = [\Delta_{32}+\Delta_{41}-\Delta_{23}-\Delta_{14}].$}
	\end{longtable}	
	
Take $\theta=\sum\limits_{i=1}^{7}\alpha_i\nabla_i\in {\rm H^2}({\mathfrak N}_{08}^{1}).$
	The automorphism group of ${\mathfrak N}_{08}^{1}$ consists of invertible matrices of the form
	$$\phi=
	\begin{pmatrix}
	x &  y  & 0 & 0\\
	x+y-z &  z  & 0 & 0\\
	t &  v  & x(z-y) & y(z-y)\\
    u &  w  & (x+y-z)(z-y) & z(z-y)
	\end{pmatrix}.
	$$
	Since
	$$
	\phi^T\begin{pmatrix}
	0 &  \alpha_1   & \alpha_3-\alpha_5 & \alpha_4-\alpha_7\\
	\alpha_2 & 0 & -\alpha_3-\alpha_7 & -\alpha_4- \alpha_6\\
	\alpha_5 &  \alpha_7  & 0 & 0\\
	\alpha_7 &  \alpha_6  & 0 & 0
	\end{pmatrix} \phi=	\begin{pmatrix}
\alpha^* & \alpha_1^*+\alpha^{**} & \alpha_3^*-\alpha_5^* & \alpha_4^*-\alpha_7^*\\
\alpha_2^*-\alpha^* & -\alpha^{**} & -\alpha_3^*-\alpha_7^* & -\alpha_4^*-\alpha_6^*\\
\alpha_5^* &  \alpha_7^*  & 0 & 0\\
\alpha_7^* &  \alpha_6^*  & 0 & 0
	\end{pmatrix},
	$$
	 we have that the action of ${\rm Aut} ({\mathfrak N}_{08}^{1})$ on the subspace
$\langle \sum\limits_{i=1}^{7}\alpha_i\nabla_i  \rangle$
is given by
$\langle \sum\limits_{i=1}^{7}\alpha_i^{*}\nabla_i\rangle,$
where
\begin{longtable}{lcl}
$\alpha^*_1$&$=$&$(x+y) z \alpha _1+y (x+y) \alpha _2-(vx-ty)\alpha_5-(w(x+y)-z(u+w))\alpha _6$\\
&&\multicolumn{1}{r}{$-(x(v+w)-y(u-v)-z(t+v))\alpha _7,$} \\
$\alpha^*_2$&$=$&$(x+y) (x+y-z)\alpha _1+x (x+y)\alpha_2+(vx-ty)\alpha_5+(w(x+y)-z(u+w))\alpha_6 $\\
&& \multicolumn{1}{r}{$+(x(v+w)-y(u-v)-z(t+v))\alpha _7,$}\\
$\alpha^*_3$&$=$&$(z-y)^2(x \alpha _3+(x+y-z) \alpha _4),$ \\
$\alpha_4^*$&$=$&$(z-y)^2(y \alpha _3+z \alpha_4),$\\
$\alpha_5^*$&$=$&$(z-y)(x^2 \alpha _5+(x+y-z)((x+y-z) \alpha _6+2 x \alpha _7))$\\
$\alpha_6^*$&$=$&$(z-y)(y^2 \alpha _5+z(z \alpha _6+2 y \alpha _7)),$\\
$\alpha_7^*$&$=$&$(z-y)((x+y-z)(z \alpha _6+y \alpha _7)+x(y \alpha _5+z \alpha _7)).$\\
\end{longtable}

 We are interested only in the cases with 
 \begin{center}
$(\alpha_3,\alpha_5,\alpha_7)\neq (0,0,0),$  $(\alpha_4,\alpha_6,\alpha_7)\neq (0,0,0).$ 
 \end{center} 

\begin{enumerate}
    \item $(\alpha_5,\alpha_6,\alpha_7)=(0,0,0),$ then $\alpha_3\neq0,\ \alpha_4\neq0.$ If $\alpha_4\neq-\alpha_3,$ then choosing $z=-\frac{y \alpha _3}{\alpha _4},$ we obtain that $\alpha_4^*=0$ which implies $(\alpha_4^*,\alpha_6^*,\alpha_7^*) = (0,0,0).$ Thus, we have that $\alpha_3=-\alpha_4.$
    \begin{enumerate}
        \item $(\alpha_1,\alpha_2)=(0,0),$ then we have the representative $\left\langle\nabla_3-\nabla_4\right\rangle;$
        \item $(\alpha_1,\alpha_2)\neq(0,0),$ without loss of generality, we can suppose $\alpha_1\neq0.$
        \begin{enumerate}
            \item $\alpha_1=-\alpha_2,$ then choosing $x=\frac{ \alpha_3}{\alpha_1},$ $y=0,$ $z=1,$ we have the representative $\left\langle\nabla_1-\nabla_2+\nabla_3-\nabla_4\right\rangle;$
 
            \item $\alpha_1\neq-\alpha_2,$ then choosing $x=0,\ y=-\frac{\alpha _1^3}{\left(\alpha _1+\alpha _2\right){}^2 \alpha _3},\ z= \frac{\alpha _1^3\alpha _2}{\left(\alpha _1+\alpha _2\right)^2 \alpha _1\alpha _3},$ we have the representative $\left\langle\nabla_2+\nabla_3-\nabla_4\right\rangle.$
        \end{enumerate}
    \end{enumerate}
    \item $(\alpha_5,\alpha_6,\alpha_7)\neq(0,0,0),$ then without loss of generality we can assume $\alpha_5\neq0$ and  consider following subcases:
    \begin{enumerate}
        \item $\alpha_6\alpha_5 =\alpha_7^2, \alpha_7=-\alpha_5,\ \alpha_4=-\alpha_3,$ $ \alpha_1=-\alpha_2,$ then taking
        $v=u=w=0,$ $t=\frac{(x+y) \alpha _1}{\alpha _5},$ we have the family of representatives $\left\langle\beta\nabla_3 -\beta\nabla_4+\nabla_5+\nabla_6-\nabla_7 \right\rangle;$
        
       \item $\alpha_6\alpha_5 =\alpha_7^2,  \alpha_7=-\alpha_5,\  \alpha_4=-\alpha_3,\ \alpha_1\neq-\alpha_2,$ then  taking 
       \begin{center}$y=v=u=w=0,$ $z=x=\frac{(\alpha _1+\alpha _2)}{\alpha _5},$ $t=\frac{\alpha_1(\alpha _1+\alpha _2)}{\alpha _5^2},$ 
       \end{center} we have the family of representatives $\left\langle\nabla_2 + \beta\nabla_3 -\beta\nabla_4+\nabla_5+\nabla_6-\nabla_7 \right\rangle;$

        \item $\alpha_6\alpha_5 =\alpha_7^2, \  \alpha_7=-\alpha_5,\ \alpha_4\neq-\alpha_3,$ then we
        can choose $y$ and $z$ such that $\alpha_4^*=0,$ $\alpha_3^*\neq0.$ Thus we can suppose $\alpha_4=0,$ 
         moreover taking $v=u=w=0,$ $t=\frac{x \alpha _1}{\alpha _5},$ we can also suppose  $\alpha_1=0.$
         \begin{enumerate}
        \item if $\alpha_2=0,$ then choosing $x=1,$ $z=\frac{\alpha_3}{\alpha_5},$ we have the representative $\left\langle \nabla_3+\nabla_5+\nabla_6-\nabla_7 \right\rangle;$
        \item if $\alpha_2\neq0,$ then choosing $x=\frac{\alpha_2\alpha _5^2}{\alpha_3^3},$ $z=\frac{\alpha_2 \alpha_5}{\alpha_3^2},$ we have the representative 
        $\left\langle \nabla_2+ \nabla_3+\nabla_5+\nabla_6-\nabla_7 \right\rangle.$ 
         \end{enumerate}
        \item $\alpha_6\alpha_5 =\alpha_7^2, \alpha_5\neq-\alpha_7,$ then choosing $y=-\frac {z\alpha_7}{\alpha_5},$ $w=0,$ $v =\frac{z(\alpha_1\alpha_5 - \alpha_2\alpha_7)}{\alpha_5(\alpha_5+\alpha_7)},$ we can suppose $\alpha_1=\alpha_6=\alpha_7=0.$ Since $(\alpha_4,\alpha_6,\alpha_7)\neq (0,0,0),$ we have that $\alpha_4\neq 0.$
          \begin{enumerate}
        \item $\alpha_2 = 0,$ then choosing $x=z\sqrt{\frac{\alpha_4}{\alpha_5}},$  we have the family of representatives $\left\langle \beta \nabla_3+ \nabla_4+\nabla_5\right\rangle.$
        \item $\alpha_2 \neq 0,$ then choosing $x=\frac{\alpha_2}{\alpha_5}\sqrt{\frac{\alpha_4}{\alpha_5}},$ $z=\frac{\alpha_2}{\alpha_5},$ we have the representative $\left\langle \nabla_2+\beta\nabla_3+ \nabla_4+\nabla_5\right\rangle.$
        
         \end{enumerate}
         
         \item $\alpha_6\alpha_5  \neq \alpha_7^2 ,$ then choosing suitable value of $z$ and $y$ such that $y-z\neq 0,$
         and we can suppose $\alpha_6=0$ and $\alpha_7 \neq 0.$
       \begin{enumerate}
        \item $\alpha_5 =-2\alpha_7,$ $\alpha_4= 0,$
        $\alpha_1=-\alpha_2,$ then choosing $y=u=v=w=0,$ $t= -\frac{x \alpha _1}{\alpha _7},$
        we have the family of  representatives $\left\langle\beta\nabla_3-2\nabla_5+\nabla_7\right\rangle;$
        \item $\alpha_5 =-2\alpha_7,$ $\alpha_4= 0,$
        $\alpha_1\neq-\alpha_2,$ then choosing $y=u=v=w=0,$ $t= -\frac{x \alpha _1}{\alpha _7},$ $z=\frac{x(\alpha_1+\alpha_2)}{\alpha_7},$
        we have the family of representatives $\left\langle\nabla_2+\beta\nabla_3-2\nabla_5+\nabla_7\right\rangle;$

        \item $\alpha_5 =-2\alpha_7,$ $\alpha_4\neq 0,$
        $\alpha_1=-\alpha_2,$ then choosing $y=u=v=w=0,$ $t= -\frac{x \alpha _1}{\alpha _7},$ $z=\frac{x\alpha_7}{\alpha_4},$
        we have the family of representatives $\left\langle\beta\nabla_3+\nabla_4-2\nabla_5+\nabla_7\right\rangle;$
       
        \item $\alpha_5 =-2\alpha_7,$ $\alpha_4\neq 0,$
        $\alpha_1\neq-\alpha_2,$ then choosing 
        \begin{center}$y=u=v=w=0,$ $t= -\frac{x \alpha _1}{\alpha _7},$
        $x=\frac{\left(\alpha _1+\alpha _2\right) \alpha _4^2}{\alpha _7^3},$
                $z=\frac{\left(\alpha _1+\alpha _2\right) \alpha _4}{\alpha _7^2},$
       \end{center} we have the family of representatives $\left\langle\nabla_2+\beta\nabla_3+\nabla_4-2\nabla_5+\nabla_7\right\rangle;$
       \item $\alpha_5 \neq-2\alpha_7 ,$ $\alpha_1=-\alpha_2,$ then choosing $y=u=v=w=0,$ $z=\frac{x(2\alpha _7+\alpha _5)}{\alpha_7},$ $t= -\frac{x \alpha _1}{\alpha _7},$ 
        we have the family of  representatives $\left\langle\beta\nabla_3+\gamma\nabla_4+\nabla_7\right\rangle;$

        \item $\alpha_5 \neq-2\alpha_7 ,$ $\alpha_1\neq-\alpha_2,$ then choosing 
        \begin{center}
            $y=u=v=w=0,$ $x=\frac{4 \left(\alpha _1+\alpha _2\right) \alpha _7}{\left(\alpha _5+2 \alpha _7\right){}^2},$ $z=\frac{2 \left(\alpha _1+\alpha _2\right)}{\alpha _5+2 \alpha _7},$ $t= -\frac{x \alpha _1}{\alpha _7},$
        \end{center} 
        we have the family of  representatives $\left\langle\nabla_2+\beta\nabla_3+\gamma\nabla_4+\nabla_7\right\rangle.$
        
       \end{enumerate}
    \end{enumerate}
\end{enumerate}

Summarizing all cases for the algebra ${\mathfrak N}_{08}^{\alpha}$, we have the following distinct orbits  

\begin{center}

$\left\langle\nabla_3-\nabla_4\right\rangle_{\alpha=1},$ 
$\left\langle\nabla_1-\nabla_2+\nabla_3-\nabla_4\right\rangle_{\alpha=1},$ 
$\left\langle\nabla_2+\nabla_3-\nabla_4\right\rangle_{\alpha=1},$ 
$\left\langle\beta\nabla_3 -\beta\nabla_4+\nabla_5+\nabla_6-\nabla_7 \right\rangle_{\alpha=1},$ 
$\left\langle\nabla_2 + \beta\nabla_3 -\beta\nabla_4+\nabla_5+\nabla_6-\nabla_7 \right\rangle_{\alpha=1},$ 
$\left\langle \nabla_3+\nabla_5+\nabla_6-\nabla_7 \right\rangle_{\alpha=1},$ 
$\left\langle \nabla_2+ \nabla_3+\nabla_5+\nabla_6-\nabla_7 \right\rangle_{\alpha=1},$ 
$\left\langle \beta \nabla_3+ \nabla_4+\nabla_5\right\rangle_{\alpha=1},$ 
$\left\langle \nabla_2+\beta\nabla_3+ \nabla_4+\nabla_5\right\rangle_{\alpha=1},$ 
$\left\langle\beta\nabla_3-2\nabla_5+\nabla_7\right\rangle_{\alpha=1},$ 
$\left\langle\nabla_2+\beta\nabla_3-2\nabla_5+\nabla_7\right\rangle_{\alpha=1},$ 
$\left\langle\beta\nabla_3+\nabla_4-2\nabla_5+\nabla_7\right\rangle_{\alpha=1},$ 
$\left\langle\nabla_2+\beta\nabla_3+\nabla_4-2\nabla_5+\nabla_7\right\rangle_{\alpha=1},$ 
$\left\langle\beta\nabla_3+\gamma\nabla_4+\nabla_7\right\rangle_{\alpha=1},$ 
$\left\langle\nabla_2+\beta\nabla_3+\gamma\nabla_4+\nabla_7\right\rangle_{\alpha=1},$
\end{center}

which gives the following new algebras (see section \ref{secteoA}):

\begin{center}
${\rm N}_{70},$
${\rm N}_{71},$
${\rm N}_{72},$
${\rm N}_{73}^{\beta},$
${\rm N}_{74}^{\beta},$
${\rm N}_{75},$
${\rm N}_{76},$
${\rm N}_{77}^{\beta},$
${\rm N}_{78}^{\beta},$
${\rm N}_{79}^{\beta},$
${\rm N}_{80}^{\beta},$
${\rm N}_{81}^{\beta},$
${\rm N}_{82}^{\beta},$
${\rm N}_{83}^{\beta,\gamma},$
${\rm N}_{84}^{\beta,\gamma}.$
\end{center}

\subsubsection{Central extensions of ${\mathfrak N}_{12}$}
	Let us use the following notations:
	\begin{longtable}{lllllll} 
	$\nabla_1 = [\Delta_{11}],$ & $\nabla_2 = [\Delta_{13}],$ & 
	$\nabla_3 = [\Delta_{14}-\Delta_{41}],$ \\
	$\nabla_4 = [\Delta_{22}],$ & 
	$\nabla_5 = [\Delta_{23}-\Delta_{32}],$ & 
	$\nabla_6 = [\Delta_{24}].$ 
	\end{longtable}	
	
Take $\theta=\sum\limits_{i=1}^{6}\alpha_i\nabla_i\in {\rm H^2}({\mathfrak N}_{12}).$
	The automorphism group of ${\mathfrak N}_{12}$ consists of invertible matrices of the form
		$$\phi_1=
	\begin{pmatrix}
	x &  0  & 0 & 0\\
	0 &  y  & 0 & 0\\
	z &  v  & xy & 0\\
    u &  t  & 0 & xy
	\end{pmatrix}, \quad 
	\phi_2=
	\begin{pmatrix}
	0 &  x  & 0 & 0\\
	y &  0  & 0 & 0\\
	z &  v  & 0 & x y\\
    u &  t  & x y & 0
	\end{pmatrix}.
	$$
	Since
	$$
	\phi^T_1\begin{pmatrix}
	\alpha_1 & 0 & \alpha_2 & \alpha_3\\
	0 & \alpha_4 & \alpha_5 & \alpha_6\\
	0 &  -\alpha_5  & 0 & 0\\
	-\alpha_3 &  0  & 0 & 0
	\end{pmatrix} \phi_1=\begin{pmatrix}
    \alpha_1^* & \alpha^* & \alpha_2^* & \alpha_3^*\\
	\alpha^{**} & \alpha_4^* & \alpha_5^* & \alpha_6^*\\
	0 &  -\alpha_5^*  & 0 & 0\\
	-\alpha_3^* &  0  & 0 & 0
	\end{pmatrix},
	$$
	 we have that the action of ${\rm Aut} ({\mathfrak N}_{12})$ on the subspace
$\langle \sum\limits_{i=1}^{6}\alpha_i\nabla_i  \rangle$
is given by
$\langle \sum\limits_{i=1}^{6}\alpha_i^{*}\nabla_i\rangle,$
where
\begin{longtable}{lcllcllcl}
$\alpha^*_1$&$=$&$x(x\alpha_1+z\alpha_2)$ &
$\alpha^*_2$&$=$&$x^2 y \alpha _2,$&
$\alpha^*_3$&$=$&$x^2 y \alpha _3,$ \\

$\alpha_4^*$&$=$&$y(y \alpha _4+v \alpha _6),$&
$\alpha_5^*$&$=$&$x y^2\alpha_5,$ &
$\alpha_6^*$&$=$&$x y^2\alpha_6.$\\
\end{longtable}

 We are interested only in the cases with 
 \begin{center}
$(\alpha_2,\alpha_5)\neq (0,0),$  $(\alpha_3,\alpha_6)\neq (0,0).$ 
 \end{center} 
\begin{enumerate}
    \item $(\alpha_2,\alpha_6)=(0,0),$ then $\alpha_3\neq0,\ \alpha_5\neq0$. 
    \begin{enumerate}
        \item $\alpha_4=0,\ \alpha_1=0,$ then choosing $x=1,\ y=\frac{\alpha_3}{\alpha_5},$ we have the representative $\left\langle\nabla_3+\nabla_5\right\rangle;$
        \item $\alpha_4=0,\ \alpha_1\neq0,$ then choosing $x=\frac{\alpha_1\alpha_5}{\alpha_3^2},\ y=\frac{\alpha_1}{\alpha_3},$ we have the representative $\left\langle\nabla_1+ \nabla_3+\nabla_5\right\rangle;$
        \item $\alpha_4\neq0,$ then choosing $x=\frac{\alpha_4}{\alpha_5},\ y=\frac{\alpha_3\alpha_4}{\alpha_5^2},$ we have the family of  representatives $\left\langle\alpha\nabla_1+ \nabla_3+\nabla_4+\nabla_5\right\rangle;$
    \end{enumerate}
    \item $(\alpha_2,\alpha_6)\neq(0,0),$ then without loss of generality, we can suppose $\alpha_6\neq0$ and choosing $v=-\frac{y \alpha _4}{\alpha _6},$ we have $\alpha_4^*=0.$
    \begin{enumerate}
        \item $\alpha_2\neq0,$ then choosing $x=1,\ y=\frac{\alpha_2}{\alpha_4},\ z=-\frac{\alpha_1}{\alpha _2},$ we have the family of representatives   $\left\langle\nabla_2+\alpha\nabla_3+\beta\nabla_5+\nabla_6\right\rangle;$
        \item $\alpha_2=0,\ \alpha_3=\alpha_1=0,$ then we have the family of representatives  $\left\langle\alpha\nabla_5+\nabla_6\right\rangle_{\alpha\neq0};$
        \item $\alpha_2=0,\ \alpha_3=0,\ \alpha_1\neq0,$ then choosing $x=\frac{\alpha_6}{\alpha_1},\ y=1,$ we have the  family of representatives $\left\langle\nabla_1+\alpha\nabla_5+\nabla_6\right\rangle_{\alpha\neq0};$
        \item $\alpha_2=0,\ \alpha_3\neq0,\ \alpha_1=0,$ then  choosing $x=\frac{\alpha_6}{\alpha_3},\ y=1,$ we have the  family of representatives $\left\langle\nabla_3+\alpha\nabla_5+\nabla_6\right\rangle_{\alpha\neq0};$
        \item $\alpha_2=0,\ \alpha_3\neq0,\ \alpha_1\neq0,$ then  choosing $x=\frac{\alpha_1\alpha_6}{\alpha_3^2},\ y=\frac{\alpha_1}{\alpha_3},$ we have the  family of representatives $\left\langle\nabla_1+\nabla_3+\alpha\nabla_5+\nabla_6\right\rangle_{\alpha\neq0};$
    \end{enumerate}
\end{enumerate}

Summarizing all cases, we have the following distinct orbits 

\begin{center} 
$\langle\nabla_3+\nabla_5\rangle,$  
$\langle\nabla_1+\nabla_3+\nabla_5\rangle,$
$\langle\alpha\nabla_1+\nabla_3+\nabla_4+\nabla_5\rangle^{O(\alpha)\simeq O({\alpha}^{-1})},$ $\langle\nabla_2+\alpha\nabla_3+\beta\nabla_5+\nabla_6\rangle^{O(\alpha,\beta)\simeq O(\beta,\alpha)},$
$\langle\alpha\nabla_5+\nabla_6\rangle_{\alpha\neq0}$,  $\langle\nabla_1+\alpha\nabla_5+\nabla_6\rangle_{\alpha\neq0},$ $\langle\nabla_3+\alpha\nabla_5+\nabla_6\rangle_{\alpha\neq0},$  $\langle\nabla_1+\nabla_3+\alpha\nabla_5+\nabla_6\rangle_{\alpha\neq0},$
\end{center}

which gives the following new algebras (see section \ref{secteoA}):

\begin{center}
${\rm N}_{85},$
${\rm N}_{86},$
${\rm N}_{87}^{\alpha},$
${\rm N}_{88}^{\alpha,\beta},$
${\rm N}_{89}^{\alpha\neq0},$
${\rm N}_{90}^{\alpha\neq0},$
${\rm N}_{91}^{\alpha\neq0},$
${\rm N}_{92}^{\alpha\neq0}.$

\end{center}

\subsubsection{Central extensions of ${\mathfrak N}_{13}$}
	Let us use the following notations:
	\begin{longtable}{lllllll} 
	$\nabla_1 = [\Delta_{21}],$ & $\nabla_2 = [\Delta_{22}],$ & 
	$\nabla_3 = [\Delta_{14}+\Delta_{23}],$ \\
	$\nabla_4 = [\Delta_{24}-\Delta_{13}+2\Delta_{14}],$ & 
	$\nabla_5 = [\Delta_{42}-2\Delta_{13}+\Delta_{31}-2\Delta_{32}],$ & 
	$\nabla_6 = [\Delta_{41}-2\Delta_{14}-\Delta_{32}].$ 
	\end{longtable}	
	
Take $\theta=\sum\limits_{i=1}^{6}\alpha_i\nabla_i\in {\rm H^2}({\mathfrak N}_{13}).$
	The automorphism group of ${\mathfrak N}_{13}$ consists of invertible matrices of the form
	$$\phi_1=
	\begin{pmatrix}
	x &  0  & 0 & 0\\
	0 & x  & 0 & 0\\
	z &  u  & x^2 & 0\\
    t &  v  & 0 & x^2
	\end{pmatrix}, \quad 
	\phi_2=
	\begin{pmatrix}
	0 &  x  & 0 & 0\\
	x &  0  & 0 & 0\\
	z &  u  & -x^2 & 2x^2\\
    t &  v  & 0 & x^2
	\end{pmatrix}.
	$$
	Since
	$$
	\phi^T_1\begin{pmatrix}
	0 & 0 & -\alpha_4-2\alpha_5 & \alpha_3+2\alpha_4-2\alpha_6\\
	\alpha_1 & \alpha_2 & \alpha_3 & \alpha_4\\
	\alpha_5 &  -2\alpha_5-\alpha_6  & 0 & 0\\
	\alpha_6 &  \alpha_5  & 0 & 0
	\end{pmatrix} \phi_1=$$
	$$=\begin{pmatrix}
    \alpha^* & \alpha^{**} & -\alpha_4^*-2\alpha_5^* & \alpha_3^*+2\alpha_4^*-2\alpha_6^*\\
	\alpha_1^*-\alpha^{**} & \alpha_2^*+\alpha^*+2\alpha^{**} & \alpha_3^* & \alpha_4^*\\
	\alpha_5^* &  -2\alpha_5^*-\alpha_6^*  & 0 & 0\\
	\alpha_6^* &  \alpha_5^*  & 0 & 0
	\end{pmatrix},
	$$
	 we have that the action of ${\rm Aut} ({\mathfrak N}_{13})$ on the subspace
$\langle \sum\limits_{i=1}^{6}\alpha_i\nabla_i  \rangle$
is given by
$\langle \sum\limits_{i=1}^{6}\alpha_i^{*}\nabla_i\rangle,$
where
for $\phi_1$: 
\begin{longtable}{lcl}
$\alpha^*_1$&$=$&$x(x\alpha_1+(v+z)\alpha_3+(t-u+2v)\alpha _4+(t-u-2z)\alpha _5-(v+z) \alpha_6),$ \\
$\alpha^*_2$&$=$&$x(x\alpha_2-(t-u+2v)\alpha_3-(2t-2u+3v-z)\alpha_4-$\\
&&\multicolumn{1}{r}{$(2t-2u-v-5z)\alpha_5 +(t-u+4v+2z)\alpha_6),$}\\
$\alpha^*_3$&$=$&$x^3 \alpha _3,$ \\
$\alpha_4^*$&$=$&$x^3 \alpha _4,$\\
$\alpha_5^*$&$=$&$x^3 \alpha _5,$\\
$\alpha_6^*$&$=$&$x^3 \alpha _6.$\\
\end{longtable}

for $\phi_2$: \begin{longtable}{lcl}
$\alpha^*_1$&$=$&$x(x\alpha_1+(t+u)\alpha_3+(2t+v-z)\alpha_4-(2u-v+z)\alpha_5- (t+u) \alpha_6),$ \\
$\alpha^*_2$&$=$&$-x(2x\alpha_1+x\alpha _2+(2u-v+z)\alpha_3+(t+u)\alpha _4+(t+u)\alpha_5 +(2t+v-z)\alpha _6),$\\
$\alpha^*_3$&$=$&$x^3(\alpha_4+2\alpha_5),$ \\
$\alpha_4^*$&$=$&$x^3(\alpha_3-2(2 \alpha _5+\alpha_6)),$\\
$\alpha_5^*$&$=$&$x^3(2\alpha_5+\alpha_6),$\\
$\alpha_6^*$&$=$&$-x^3(3\alpha_5+2\alpha_6).$\\
\end{longtable}

 We are interested only in the cases with 
$(\alpha_3,\alpha_4,\alpha_5,\alpha_6)\neq (0,0,0,0).$ 
\begin{enumerate}
    \item $(\alpha_5,\alpha_6)=(0,0).$ Let us consider the following subcases for automorphisms of type $\phi_1$:
    \begin{enumerate}
        \item $\alpha_3=0,$ then $\alpha_4\neq0$, then choosing $u=v=0,$  $z= -\frac{x (2 \alpha _1+\alpha _2)}{\alpha _4},$ $t=-\frac{x \alpha _1}{\alpha _4},$
        we have the representative  $\left\langle\nabla_4\right\rangle;$
        \item $\alpha_3\neq0,\ \alpha_3=-\alpha_4,\ \alpha_1=-\alpha_2,$ then choosing $u=v=z=0,$ $t=\frac{x\alpha_1}{\alpha_3,}$ we have the representative  $\left\langle\nabla_3-\nabla_4\right\rangle;$
        \item $\alpha_3\neq0,\ \alpha_3=-\alpha_4,\ \alpha_1\neq-\alpha_2,$ then choosing $u=v=z=0,$ $t=-\frac{x\alpha_2}{\alpha_3,},$ $x=\frac{\alpha_1+\alpha_2}{\alpha_3},$ we have the representative  $\left\langle\nabla_1+\nabla_3-\nabla_4\right\rangle;$
    
        \item $\alpha_3\neq0,\ \alpha_4\neq-\alpha_3,$ then choosing 
        \begin{center}$u=v=0,\ z= -\frac{x \left(\alpha _2 \alpha _4+\alpha _1 \left(\alpha _3+2 \alpha _4\right)\right)}{\left(\alpha _3+\alpha _4\right){}^2},$ $t=\frac{x \left(\alpha _2 \alpha _3-\alpha _1 \alpha _4\right)}{\left(\alpha _3+\alpha _4\right){}^2},$ \end{center} we have the family of representatives  $\left\langle\nabla_3+\alpha\nabla_4\right\rangle_{\alpha \neq-1};$
    \end{enumerate}
   
    \item $(\alpha_5,\alpha_6)\neq (0,0),$ then without loss of generality (after an action of $\phi_2$), we can suppose $\alpha_5\neq0$ and consider
    the following subcases for automorphisms of type $\phi_1$:
    \begin{enumerate}
        \item $\alpha_4=-\alpha_5,\ \alpha_3=2\alpha_5+\alpha_6,\ \alpha_1=0,$ then choosing $u=v=z=0,\ t=\frac{x \alpha _2}{2 \alpha _5},$ we have the family of representatives  $\left\langle(2+\alpha)\nabla_3-\nabla_4+\nabla_5+ \alpha\nabla_6\right\rangle;$
        \item $\alpha_4=-\alpha_5,\ \alpha_3=2\alpha_5+\alpha_6,\ \alpha_1\neq0,$ then choosing $u=v=z=0,\ t=\frac{\alpha_1\alpha_2} {2\alpha_5^2},$ $x=\frac{\alpha_1}{\alpha_5},$ we have the family of representatives  $\left\langle\nabla_1+(2+\alpha)\nabla_3-\nabla_4+ \nabla_5+\alpha\nabla_6\right\rangle;$

\item $\alpha_4=-\alpha_5,\ \alpha_3\neq2\alpha_5+\alpha_6, \ \alpha_3\neq \alpha_6,$ then 
        choosing $t=u=0,$ 
\begin{center}
    $x=(\alpha_3-\alpha_6) (\alpha_3-2 \alpha_5-\alpha_6),$
$z=2 \alpha_1 (\alpha_5+\alpha_6)+\frac{1}{2}\alpha_2 (2 \alpha_5+\alpha_6-\alpha_3)-\alpha_1 \alpha_3$\\
 and $v=\frac{1}{2}\alpha_2 (\alpha_3-2 \alpha_5-\alpha_6)-\alpha_1 (2 \alpha_5+\alpha_6),$ 
\end{center}
we have the family of representatives 
$\left\langle  \alpha\nabla_3-\nabla_4+ \nabla_5+\gamma\nabla_6\right\rangle_{\alpha\notin \{\gamma,  \gamma+2\} };$

 \item $\alpha_4=-\alpha_5,\ \alpha_3\neq2\alpha_5+\alpha_6, \ \alpha_3= \alpha_6,$ then 
        choosing 
$z=\frac{x \alpha_1}{2 \alpha_5}$ and 
$v=0,$ 
we have two families of representatives 
$\left\langle   \alpha\nabla_3-\nabla_4+ \nabla_5+\alpha\nabla_6\right\rangle$ 
and 
$\left\langle  \nabla_2 +\alpha\nabla_3-\nabla_4+ \nabla_5+\alpha\nabla_6\right\rangle,$  
depending on $\alpha_2 \alpha_5=-\alpha_1 (\alpha_3+2 \alpha_5)$ or not.
The last family for $\alpha\neq -2$ under an action of $\phi_2$
       gives a case with $\alpha_4\neq-\alpha_5,$ which will be considered below.
       
        \item $\alpha_4\neq-\alpha_5,$
        $\alpha_3^2 + \alpha_4^2 + 2 \alpha_4 \alpha_5 + 
 2 \alpha_3 \alpha_4  + (\alpha_5 + \alpha_6)^2\neq 2 \alpha_3\alpha_6,$ then by choosing 
\begin{center} 
$x=\alpha_3^2+\alpha_4^2+2 \alpha_4 \alpha_5+2 \alpha_3 (\alpha_4-\alpha_6)+(\alpha_5+\alpha_6)^2,$ 
 $u=\alpha_2 (\alpha_6-\alpha_3-2 \alpha_4)+\alpha_1 (\alpha_5+4 \alpha_6-2 \alpha_3-3 \alpha_4),$ 
$v=-\alpha_2 (\alpha_4+\alpha_5)-\alpha_1 (\alpha_3+2 \alpha_4+2 \alpha_5-\alpha_6)$
 and $z=t=0,$      \end{center}        we have the family of representatives \begin{center}
     $\left\langle\alpha\nabla_3+\beta\nabla_4+ \nabla_5+\gamma\nabla_6\right\rangle_{\beta\neq-1, \ 
 \alpha^2+\beta^2+2\beta+2\alpha\beta+(1+\gamma)^2\neq 2\alpha\gamma};$ 

 \end{center}
   \item $\alpha_4\neq-\alpha_5,$
        $\alpha_3^2 + \alpha_4^2 + 2 \alpha_4 \alpha_5 + 
 2 \alpha_3 \alpha_4  + (\alpha_5 + \alpha_6)^2= 2 \alpha_3\alpha_6,$ then by choosing
 $u=\frac{x \alpha_1}{\alpha_4+\alpha_5}$ and $z=t=0,$
 we have two families of representatives 
\begin{center}
$\left\langle \nabla_2+ \Xi \nabla_3+\alpha\nabla_4+ \nabla_5 +\beta\nabla_6 \right\rangle_{\alpha\neq-1}$ 
and
$\left\langle   \Xi \nabla_3+\alpha\nabla_4+ \nabla_5 +\beta\nabla_6 \right\rangle_{\alpha\neq-1},$ 
\end{center}
where $\Xi= \beta-\alpha\pm \sqrt{-2\alpha\beta-2\alpha-2\beta-1},$
depending on 
\begin{center}
$\alpha_2 (\alpha_4+\alpha_5)+\alpha_1 (\alpha_4+2 \alpha_5) \neq \alpha_1\sqrt{-2 \alpha_4 (\alpha_5+\alpha_6)-\alpha_5 (\alpha_5+2 \alpha_6)})$ or not.

\end{center}

        \end{enumerate}

    \end{enumerate}

Summarizing all cases, we have the following distinct orbits 
\begin{center}
$\langle\nabla_4\rangle,$  
$\langle\nabla_3+\alpha\nabla_4\rangle^{O(\alpha)\simeq O(\alpha^{-1})},$ 
$\langle\nabla_1+\nabla_3-\nabla_4\rangle,$ 

$\langle \nabla_1+(2+\alpha) \nabla_3 -\nabla_4+\nabla_5 +\alpha  \nabla_6\rangle^{O(\alpha)\simeq O((2+\alpha)^{-1})},$ 

${\langle \nabla_2-2 \nabla_3 -\nabla_4+\nabla_5 -2  \nabla_6\rangle,}$ 

$\langle \alpha \nabla_3+ \beta\nabla_4+\nabla_5 +\gamma \nabla_6\rangle^{O(\alpha,\beta,\gamma)\simeq 
O(\frac{2+\beta}{2+\gamma},\frac{\alpha-2(2+\gamma)}{2+\gamma},-\frac{3+2\gamma}{2+\gamma})},$ 

$\langle \nabla_2+(\beta-\alpha+ \sqrt{-2\alpha\beta-2\alpha-2\beta-1})\nabla_3+ \alpha\nabla_4+\nabla_5 +\beta\nabla_6\rangle_{\alpha\neq-1},$ 

$\langle \nabla_2+(\beta-\alpha- \sqrt{-2\alpha\beta-2\alpha-2\beta-1})\nabla_3+ \alpha\nabla_4+\nabla_5 +\beta\nabla_6\rangle_{\alpha\notin \{-1,-\frac{1 + 2 \beta}{2 + 2 \beta} \}},$ 
 
\end{center}
which gives the following new algebras (see section \ref{secteoA}):

\begin{center}
${\rm N}_{93},$
${\rm N}_{94}^{\alpha},$
${\rm N}_{95},$
${\rm N}_{96}^{\alpha},$
${\rm N}_{97},$
${\rm N}_{98}^{\alpha, \beta, \gamma},$
${\rm N}_{99}^{\alpha, \beta},$
${\rm N}_{100}^{\alpha, \beta}.$
\end{center}

\subsubsection{Central extensions of ${\mathfrak N}_{14}^0$}
	Let us use the following notations:
	\begin{longtable}{lllllll} 
	$\nabla_1 = [\Delta_{11}],$ & $\nabla_2 = [\Delta_{21}],$ &
	$\nabla_3 = [\Delta_{23}],$ &
	$\nabla_4 = [\Delta_{13}+\Delta_{24}],$ \\
	$\nabla_5 = [\Delta_{32}],$ &
	$\nabla_6 = [\Delta_{42}-\Delta_{24}],$ & 
	$\nabla_7 = [\Delta_{14}].$
	\end{longtable}	
	
Take $\theta=\sum\limits_{i=1}^{7}\alpha_i\nabla_i\in {\rm H^2}({\mathfrak N}_{14}^0).$
	The automorphism group of ${\mathfrak N}_{14}^0$ consists of invertible matrices of the form
	$$\phi=
	\begin{pmatrix}
	x &  z  & 0 & 0\\
	0 &  y  & 0 & 0\\
	w &  u  & y^2 & 0\\
    t &  v  & yz & xy
	\end{pmatrix}.
	$$
	Since
	$$
	\phi^T\begin{pmatrix}
	\alpha_1 & 0 & \alpha_4 & \alpha_7\\
	\alpha_2 & 0 & \alpha_3 & \alpha_4-\alpha_6\\
	0 &  \alpha_5  & 0 & 0\\
	0 &  \alpha_6  & 0 & 0
	\end{pmatrix} \phi=\begin{pmatrix}
    \alpha_1^* & \alpha^* & \alpha_4^* & \alpha_7^*\\
	\alpha_2^* & \alpha^{**} & \alpha_3^* & \alpha_4^*-\alpha_6^*\\
	0 &  \alpha_5^*  & 0 & 0\\
	0 &  \alpha_6^*  & 0 & 0
	\end{pmatrix},
	$$
	 we have that the action of ${\rm Aut} ({\mathfrak N}_{14}^0)$ on the subspace
$\langle \sum\limits_{i=1}^{7}\alpha_i\nabla_i  \rangle$
is given by
$\langle \sum\limits_{i=1}^{7}\alpha_i^{*}\nabla_i\rangle,$
where
\begin{longtable}{lcl}
$\alpha^*_1$&$=$&$x(x \alpha _1+w \alpha _4+t \alpha_7),$ \\
$\alpha^*_2$&$=$&$x z \alpha _1+x y \alpha _2+w y \alpha _3+(ty+wz)\alpha _4-t y\alpha_6+t z \alpha_7,$\\
$\alpha^*_3$&$=$&$y(y^2 \alpha _3+z(2 y \alpha _4-y \alpha _6+z \alpha _7)),$ \\
$\alpha_4^*$&$=$&$x y \left(y \alpha _4+z \alpha _7\right),$\\
$\alpha_5^*$&$=$&$y^2 (y \alpha _5+z \alpha _6),$\\
$\alpha_6^*$&$=$&$x y^2 \alpha _6,$\\
$\alpha_7^*$&$=$&$x^2 y \alpha _7.$\\
\end{longtable}

We are interested only in the cases with
 \begin{center}
$(\alpha_3,\alpha_4,\alpha_5)\neq (0,0,0),$ $(\alpha_4,\alpha_6,\alpha_7)\neq (0,0,0).$
 \end{center}
\begin{enumerate}
    \item $\alpha_7\neq0,$ then choosing $z=-\frac{y \alpha _4}{\alpha _7},$ $t=-\frac{x \alpha _1+w \alpha _4}{\alpha _7},$
    we have $\alpha_1^*=\alpha_4^*=0.$ Thus, we can suppose $\alpha_1=\alpha_4=0$ and consider following subcases:
\begin{enumerate}
    \item $\alpha_3\neq0,$ then choosing $x=1, \ y = \sqrt{{\alpha _7}{\alpha _3^{-1}}}, \ w=-{\alpha_2}{\alpha_3^{-1}},$
        we have the family of representatives  $\left\langle \nabla_3+\beta\nabla_4+\gamma\nabla_6+\nabla_7\right\rangle;$

        \item $\alpha_3=0,$ then $\alpha_5\neq0.$
            \begin{enumerate}
            \item $\alpha_2=0,$ then choosing $x=1, \ y = \sqrt{\alpha _7\alpha _5^{-1}},$ we have the family of representatives  $\left\langle\nabla_5+\beta\nabla_6+\nabla_7\right\rangle;$
            
            \item $\alpha_2\neq0,$ then choosing $x=\frac{\alpha_2}{\alpha_7}, \ y=\frac{\alpha _2}{\sqrt{\alpha _5\alpha _7}},$ we have the family of representatives $\left\langle\nabla_2+\nabla_5+\beta\nabla_6+\nabla_7\right\rangle.$
            \end{enumerate}
    \end{enumerate}
\item $\alpha_7 = 0,$ $\alpha_6 \neq 0,$  then choosing $z=-\frac{y \alpha _5}{\alpha _6},$ we have $\alpha_5^*=0.$ Thus, we can suppose $\alpha_5=0$ and consider following subcases:
            \begin{enumerate}
            \item $\alpha_4 = 0,$ then $\alpha_3 \neq 0$ and choosing $t= \frac{x \alpha _2}{\alpha _6}, \ w=0,$ we can suppose $\alpha_2=0.$  Consider following subcases:
                \begin{enumerate}
                \item  $\alpha_1 = 0,$ then choosing $x=\frac{\alpha_3}{\alpha_6}, \ y=1,$ we have the representative $\left\langle\nabla_3+\nabla_6\right\rangle;$
                \item  $\alpha_1 \neq 0,$ then choosing $x=\frac{\alpha _1 \alpha _3^2}{\alpha _6^3}, \ y=\frac{\alpha _1 \alpha _3}{\alpha _6^2},$ we have the representative $\left\langle\nabla_1+\nabla_3+\nabla_6\right\rangle.$
                
                \end{enumerate}
            \item $\alpha_4 = \alpha_6,$ then  choosing $w= -\frac{x \alpha _1}{\alpha _6},$ we can suppose $\alpha_1=0$ and consider following subcases:
                \begin{enumerate}
                \item  $\alpha_2 = 0,$ $\alpha_3 = 0,$ then we have the representative $\left\langle\nabla_4+\nabla_6\right\rangle;$
                \item  $\alpha_2 = 0,$ $\alpha_3 \neq 0,$ then choosing $x=\frac{\alpha _3}{\alpha _6}, \ y=1,$ we have the representative $\left\langle\nabla_3+\nabla_4+\nabla_6\right\rangle;$
                \item  $\alpha_2 \neq 0,$ $\alpha_3 = 0,$ then choosing $x=1, \ y=\frac{\alpha _2}{\alpha _6},$ we have the representative $\left\langle\nabla_2+\nabla_4+\nabla_6\right\rangle;$
                 \item  $\alpha_2 \neq 0,$ $\alpha_3 \neq 0,$ then choosing $x=\frac{\alpha _2 \alpha _3}{\alpha _6^2}, \ y=\frac{\alpha _2}{\alpha _6},$ we have the representative $\left\langle\nabla_2+\nabla_3+\nabla_4+\nabla_6\right\rangle.$
                \end{enumerate}
            \item $\alpha_4 \neq \alpha_6,$ $\alpha_4 \neq 0,$ then  choosing $t= \frac{x (\alpha _1 \alpha _3- \alpha _2 \alpha _4)}{\alpha _4(\alpha _4  -\alpha _6)}, \ w= -\frac{x \alpha _1}{\alpha _4},$ we can suppose $\alpha_1=\alpha_2=0$ and consider following subcases:
            \begin{enumerate}
                \item  $\alpha_3 = 0,$ then we have the family of representatives $\left\langle \beta \nabla_4+\nabla_6\right\rangle_{\beta\neq 0;1};$
                \item  $\alpha_3 \neq 0,$ then choosing $x=\frac{\alpha _3}{\alpha _6}, \ y=1,$ we have the family of representatives $\left\langle \nabla_3+\beta \nabla_4+\nabla_6\right\rangle_{\beta\neq 0;1};$
            \end{enumerate}
\end{enumerate}
\item $\alpha_7 = 0,$ $\alpha_6 = 0,$  then $\alpha_4 \neq 0$ and choosing $z=-\frac{y \alpha _3}{2 \alpha _4},\ t= \frac{x \left(\alpha _1 \alpha _3-\alpha _2 \alpha _4\right)}{\alpha _4^2}, \ w=-\frac{x \alpha _1}{\alpha _4},$ we have $\alpha_1^*=\alpha_2^*=\alpha_3^*=0$ and consider following subcases:
 \begin{enumerate}
                \item  $\alpha_5 = 0,$ then we have the representative $\left\langle \nabla_4\right\rangle;$
                \item  $\alpha_5 \neq 0,$ then choosing $x=\frac{\alpha _5}{\alpha _4}, \ y=1,$ we have the representative $\left\langle \nabla_4+\nabla_5\right\rangle.$
            \end{enumerate}
\end{enumerate}

Summarizing all cases, we have the following distinct orbits 

\begin{center} 
$\left\langle \nabla_3+\beta\nabla_4+\gamma\nabla_6+\nabla_7\right\rangle^{O(\beta, \gamma) \simeq O(-\beta, -\gamma)},$ 
$\left\langle\nabla_2+\nabla_5+\beta\nabla_6+\nabla_7\right\rangle^{O(\beta) \simeq O(-\beta)},$
$\left\langle\nabla_5+\beta\nabla_6+\nabla_7\right\rangle^{O(\beta) \simeq O(-\beta)},$
$\left\langle\nabla_1+\nabla_3+\nabla_6\right\rangle,$ 
$\left\langle\nabla_2+\nabla_3+\nabla_4+\nabla_6\right\rangle,$  $\left\langle\nabla_2+\nabla_4+\nabla_6\right\rangle,$
$\left\langle \nabla_3+\beta \nabla_4+\nabla_6\right\rangle,$   
$\left\langle \beta \nabla_4+\nabla_6\right\rangle_{\beta\neq 0},$
$\left\langle \nabla_4+\nabla_5\right\rangle,$  $\left\langle \nabla_4\right\rangle.$

\end{center}

\subsubsection{Central extensions of ${\mathfrak N}_{14}^1$}
	Let us use the following notations:
	\begin{longtable}{lllllll} 
	$\nabla_1 = [\Delta_{11}],$ & $\nabla_2 = [\Delta_{21}],$ &
	$\nabla_3 = [\Delta_{23}],$ \\
	$\nabla_4 = [\Delta_{13}+\Delta_{24}],$&
	$\nabla_5 = [\Delta_{32}],$ &
	$\nabla_6 = [\Delta_{31}+\Delta_{42}].$ 
	\end{longtable}	
	
Take $\theta=\sum\limits_{i=1}^{6}\alpha_i\nabla_i\in {\rm H^2}({\mathfrak N}_{14}^1).$
	The automorphism group of ${\mathfrak N}_{14}^1$ consists of invertible matrices of the form
	$$\phi=
	\begin{pmatrix}
	x &  z  & 0 & 0\\
	0 &  y  & 0 & 0\\
	w &  u  & y^2 & 0\\
    t &  v  & 2yz & xy
	\end{pmatrix}.
	$$
	Since
	$$
	\phi^T\begin{pmatrix}
	\alpha_1 & 0 & \alpha_4 & 0\\
	\alpha_2 & 0 & \alpha_3 & \alpha_4\\
	\alpha_6 &  \alpha_5  & 0 & 0\\
	0 &  \alpha_4+\alpha_6  & 0 & 0
	\end{pmatrix} \phi=\begin{pmatrix}
    \alpha_1^* & \alpha^* & \alpha_4^* & 0\\
	\alpha_2^*+\alpha^* & \alpha^{**} & \alpha_3^* & \alpha_4^*\\
	\alpha_6^* &  \alpha_5^*  & 0 & 0\\
	0 &  \alpha_6^*  & 0 & 0
	\end{pmatrix},
	$$
	 we have that the action of ${\rm Aut} ({\mathfrak N}_{14}^1)$ on the subspace
$\langle \sum\limits_{i=1}^{6}\alpha_i\nabla_i  \rangle$
is given by
$\langle \sum\limits_{i=1}^{6}\alpha_i^{*}\nabla_i\rangle,$
where
\begin{longtable}{lcl}
$\alpha^*_1$&$=$&$x \left(x \alpha _1+w \left(\alpha _4+\alpha _6\right)\right),$ \\
$\alpha_2^*$&$=$&$xy\alpha_2+wy\alpha_3-(ux-ty-wz)\alpha_4-wy\alpha_5 +(ux-ty-wz)\alpha_6,$\\
$\alpha^*_3$&$=$&$y^2(y \alpha _3+3 z \alpha _4),$ \\
$\alpha^*_4$&$=$&$x y^2 \alpha _4,$\\
$\alpha_5^*$&$=$&$y^2 \left(y \alpha _5+3 z \alpha _6\right),$\\
$\alpha_6^*$&$=$&$x y^2 \alpha _6.$\\
\end{longtable}

We are interested only in the cases with 
 \begin{center}
$(\alpha_4,\alpha_6)\neq (0,0),$ $(\alpha_2, \alpha_3-\alpha_5,\alpha_4-\alpha_6)\neq (0,0,0).$ 
 \end{center} 
\begin{enumerate}
    \item $\alpha_6\neq0,$ then choosing $z=-\frac{y \alpha _5}{3 \alpha _6},$ we have $\alpha_5^*=0.$ Thus, we can suppose $\alpha_5=0$ and  consider following subcases:
    \begin{enumerate}
        
        \item $\alpha_4\neq\alpha_6,$ then choosing $t=-\frac{x \alpha _2+w \alpha _3}{\alpha _4-\alpha _6},$ we  can suppose $\alpha_2=0$ and  consider following subcases:
        \begin{enumerate}
            \item $\alpha_4= -\alpha_6,$ $\alpha_1=0,$ $\alpha_3=0,$ then  we have the representative $\left\langle -\nabla_4+\nabla_6\right\rangle;$
            \item $\alpha_4= -\alpha_6,$ $\alpha_1=0,$ $\alpha_3\neq 0,$ then choosing $x=1, \ y=\frac{\alpha _6}{\alpha _3},$ we have the representative $\left\langle \nabla_3 -  \nabla_4+\nabla_6\right\rangle;$
            \item $\alpha_4= -\alpha_6,$ $\alpha_1\neq 0,$ $\alpha_3= 0,$ then choosing $x=\frac{\alpha _6}{\alpha _1},\ y=1,$ we have the representative $\left\langle \nabla_1 -  \nabla_4+\nabla_6\right\rangle;$
            \item $\alpha_4= -\alpha_6,$ $\alpha_1\neq 0,$ $\alpha_3\neq 0,$ then choosing $x=\frac{\alpha _1 \alpha _3^2}{\alpha _6^3},\ y=\frac{\alpha _1 \alpha _3}{\alpha _6^2},$ we have the representative $\left\langle \nabla_1 +\nabla_3 -  \nabla_4+\nabla_6\right\rangle;$
            \item $\alpha_4\neq -\alpha_6,$ $\alpha_3= 0,$ then choosing $x=1,\ w=-\frac{\alpha _1}{\alpha_4+\alpha_6},$ we have the family of representatives $\left\langle \beta  \nabla_4+\nabla_6\right\rangle_{\beta\notin \{-1,1\}};$
            \item $\alpha_4\neq -\alpha_6,$ $\alpha_3\neq 0,$ then choosing $x=1,\ y=\frac{\alpha _6}{\alpha _3}, \ w=-\frac{\alpha _1}{\alpha_4+\alpha_6},$ we have the family of representatives $\left\langle \nabla_3+\beta  \nabla_4+\nabla_6\right\rangle_{\beta\notin \{-1,1 \}}.$
            \end{enumerate}
        \item $\alpha_4=\alpha_6,$ then $\alpha_3\neq 0$ and choosing 
        $w=-\frac{x \alpha _2}{\alpha _3},$ we  can suppose $\alpha_2=0$ and  consider following subcases:
            \begin{enumerate}
            \item $\alpha_1=0,$ then  choosing $x=1, \ y=\frac{\alpha _6}{\alpha _3},$ we have the representative $\left\langle \nabla_3+\nabla_4+\nabla_6\right\rangle;$
            \item $\alpha_1\neq 0,$ then  choosing $x=\frac{\alpha _1 \alpha _3^2}{\alpha _6^3}, \ y=\frac{\alpha _1 \alpha _3}{\alpha _6^2},$ we have the representative $\left\langle \nabla_1+\nabla_3+\nabla_4+\nabla_6\right\rangle.$
            \end{enumerate}
    \end{enumerate}
    \item $\alpha_6=0,$ then $\alpha_4\neq0,$ and choosing $z=-\frac{y \alpha _3}{3 \alpha _4},$ $t=\frac{x \left(3 \alpha _4 \left(u \alpha _4-y \alpha _2\right)+y \alpha _1 \left(2 \alpha _3-3 \alpha _5\right)\right)}{3 y \alpha _4^2},$ $w=-\frac{x \alpha _1}{\alpha _4},$ we have $\alpha_1^*=\alpha_2^*=\alpha_3^*=0$ and  consider following subcases:
        \begin{enumerate}
        \item $\alpha_5=0,$ then we have the representative  $\left\langle\nabla_4\right\rangle;$
        \item  $\alpha_5\neq 0,$ then choosing $x=1, \ y = \frac{\alpha_4}{\alpha_5},$  have the representative  $\left\langle\nabla_4+\nabla_5\right\rangle.$
        \end{enumerate}
\end{enumerate}
Summarizing all cases, we have the following distinct orbits 

\begin{center} 
$\langle\nabla_1+\nabla_3+\nabla_4+\nabla_6\rangle,$  
$\langle\nabla_1+\nabla_3-\nabla_4+\nabla_6\rangle,$ 
$\langle\nabla_1-\nabla_4+\nabla_6\rangle,$
$\langle\nabla_3+\beta\nabla_4+\nabla_6\rangle,$ $\langle\beta\nabla_4+\nabla_6\rangle_{\beta\neq1},$ 
$\langle\nabla_4+\nabla_5\rangle,$ $\langle\nabla_4\rangle.$ 
\end{center}

\subsubsection{Central extensions of ${\mathfrak N}_{14}^{\alpha\notin \{0,1\}}$}
	Let us use the following notations:
	\begin{longtable}{lllllll} 
	$\nabla_1 = [\Delta_{11}],$ & $\nabla_2 = [\Delta_{21}],$ & 
	$\nabla_3 = [\Delta_{23}],$ \\
	$\nabla_4 = [\Delta_{13}+\Delta_{24}],$ & 
	$\nabla_5 = [\Delta_{32}],$ &
$\nabla_6=[(\alpha-1)\Delta_{24}+\alpha\Delta_{31}+\Delta_{42}].$
	\end{longtable}	
	
Take $\theta=\sum\limits_{i=1}^{6}\alpha_i\nabla_i\in {\rm H^2}({\mathfrak N}_{14}^{\alpha\notin \{0,1\}}).$
	The automorphism group of ${\mathfrak N}_{14}^{\alpha\notin \{0,1\}}$ consists of invertible matrices of the form
	$$\phi=
	\begin{pmatrix}
	x &  z  & 0 & 0\\
	0 &  y  & 0 & 0\\
	w &  u  & y^2 & 0\\
    t &  v  & (1+\alpha)yz & xy
	\end{pmatrix}.
	$$
	Since
	$$
	\phi^T\begin{pmatrix}
	\alpha_1 & 0 & \alpha_4 & 0\\
	\alpha_2 & 0 & \alpha_3 & \alpha_4+(\alpha-1)\alpha_6\\
	\alpha\alpha_6 &  \alpha_5  & 0 & 0\\
	0 &  \alpha_6  & 0 & 0
	\end{pmatrix} \phi=\begin{pmatrix}
    \alpha_1^* & \alpha^* & \alpha_4^* & 0\\
	\alpha_2^*+\alpha\alpha^* & \alpha^{**} & \alpha_3^* & \alpha_4^*+(\alpha-1)\alpha_6^*\\
	\alpha\alpha_6^* &  \alpha_5^*  & 0 & 0\\
	0 &  \alpha_6^*  & 0 & 0
	\end{pmatrix},
	$$
	 we have that the action of ${\rm Aut} ({\mathfrak N}_{14}^{\alpha\notin \{0,1\}})$ on the subspace
$\langle \sum\limits_{i=1}^{6}\alpha_i\nabla_i  \rangle$
is given by
$\langle \sum\limits_{i=1}^{6}\alpha_i^{*}\nabla_i\rangle,$
where
\begin{longtable}{lcl}
$\alpha^*_1$&$=$&$x(x\alpha _1+w(\alpha _4+\alpha \alpha _6)),$ \\
$\alpha^*_2$&$=$&$x z (1-\alpha ) \alpha _1+x y \alpha _2+w y \alpha _3+(ty+wz-ux\alpha)\alpha_4-wy\alpha\alpha_5-(ty-ux\alpha+wz\alpha^2)\alpha _6,$\\
$\alpha^*_3$&$=$&$y^2(y\alpha_3+z((2+\alpha)\alpha_4-(1-\alpha^2) \alpha _6)),$ \\
$\alpha_4^*$&$=$&$x y^2 \alpha _4,$\\
$\alpha_5^*$&$=$&$y^2(y\alpha _5+z(1+2\alpha)\alpha_6),$\\
$\alpha_6^*$&$=$&$x y^2 \alpha _6.$\\
\end{longtable}

We are interested only in the cases with 
$(\alpha_4,\alpha_6)\neq (0,0).$
\begin{enumerate}
    
    \item $\alpha_6\neq0,$  then consider  following subcases:
    \begin{enumerate}
        \item $\alpha_4\neq-\alpha\alpha_6,\ \alpha_4\neq \alpha_6,$ then by choosing $u=0,$
        $w=-\frac{x\alpha_1}{\alpha_4+\alpha\alpha_6}$ and 
        $t=\frac{x(\alpha_1 (y\alpha_3+\alpha(z\alpha_4-y\alpha_5-z\alpha_6))-y\alpha_2(\alpha _4+\alpha \alpha _6))} {y(\alpha _4-\alpha _6)(\alpha _4+\alpha\alpha_6)},$  we have $\alpha_1^*=\alpha_2^*=0.$
        \begin{enumerate}
            \item $\alpha =-\frac{1}{2},\ \alpha_5=0,$ then choosing $z=-\frac{4y\alpha_3}{3(2\alpha_4-\alpha_6)},$ we have the family of  representatives $\left\langle\beta\nabla_4+\nabla_6\right\rangle_{ \beta\notin \{ 1,-\alpha\},\ \alpha=-\frac{1}{2}};$
        
            \item $\alpha =-\frac{1}{2},\ \alpha_5\neq0,$ then choosing $x=\frac{\alpha_5}{\alpha_6},$ $ y=1,$ $  z=-\frac{4\alpha_3}{3(2\alpha_4-\alpha_6)}$ we have the family of  representatives $\left\langle\beta\nabla_4+\nabla_5+\nabla_6\right\rangle_{\beta\notin \{1,-\alpha\},\ \alpha=-\frac{1}{2}};$
            \item $\alpha \neq-\frac{1}{2},$ then choosing $z=-\frac{y\alpha_5}{(1+2\alpha)\alpha_6}$ we have $\alpha_5^*=0.$ 
            \begin{enumerate}
                \item $\alpha_3=0,$ then we have the family of  representatives $\left\langle\beta\nabla_4+\nabla_6\right\rangle_{\beta\notin \{1,-\alpha\},\ \alpha\neq-\frac{1}{2}};$
                \item $\alpha_3\neq0,$ then choosing $x=\frac{\alpha_3}{\alpha_6},\  y=1,$ we have the family of  representatives $\left\langle\nabla_3+\beta\nabla_4+\nabla_6\right\rangle_{\beta\notin \{1,-\alpha\},\ \alpha\neq-\frac{1}{2}}.$
            \end{enumerate}
        \end{enumerate}
        \item $\alpha_4\neq-\alpha\alpha_6,\ \alpha_4=\alpha_6,$ then choosing $w=-\frac{x \alpha _1}{(1+\alpha ) \alpha _6},$ we have $\alpha_1^*=0.$
        \begin{enumerate}
            \item $\alpha =-\frac{1}{2},\ \alpha_5=0,\ \alpha_2=0,$ then choosing $y=1, \ z=-\frac{4\alpha_3}{3\alpha_6},$ we have the representative $\left\langle\nabla_4+\nabla_6\right\rangle_{\alpha=-\frac{1}{2}};$
            \item $\alpha =-\frac{1}{2},\ \alpha_5=0,\ \alpha_2\neq0,$ then choosing $x=1,\ y=\frac{\alpha_2}{\alpha_6},\ z=-\frac{4\alpha_2\alpha_3}{3\alpha_6^2},$ we have the representative $\left\langle\nabla_2+\nabla_4+ \nabla_6\right\rangle_{\alpha=-\frac{1}{2}};$
            \item $\alpha =-\frac{1}{2},\ \alpha_5\neq0,\ \alpha_2=0,$ then choosing $x=\frac{\alpha_5}{\alpha_6},\  y=1,\  z=-\frac{4\alpha_3}{3\alpha_6}$ we have the representative $\left\langle\nabla_4+\nabla_5+\nabla_6 \right\rangle_{\alpha=-\frac{1}{2}};$
            \item $\alpha =-\frac{1}{2},\ \alpha_5\neq0,\ \alpha_2\neq0,$ then choosing $x=\frac{\alpha_2\alpha_5}{\alpha_6^2},\  y=\frac{\alpha_2}{\alpha_6},\  z=-\frac{4\alpha_2\alpha_3}{3\alpha_6^2}$ we have the representative $\left\langle\nabla_2+\nabla_4+\nabla_5+\nabla_6 \right \rangle_{\alpha=-\frac{1}{2}};$
            \item $\alpha \neq-\frac{1}{2},$ then choosing $z=-\frac{y\alpha_5}{(1+2\alpha)\alpha_6}$ we have $\alpha_5^*=0.$ 
            \begin{enumerate}
                \item $\alpha_3=\alpha_2=0,$ then we have the representative $\left\langle\nabla_4+\nabla_6\right\rangle_{\alpha\notin \{-1,-\frac{1}{2}\}};$
                \item $\alpha_3=0,\ \alpha_2\neq0,$ then choosing $y=\frac{\alpha_2}{\alpha_6},$ we have the representative $\left\langle\nabla_2+\nabla_4+\nabla_6\right\rangle_{\alpha\notin \{-1,-\frac{1}{2}\}};$
                \item $\alpha_3\neq0,\ \alpha_2=0,$ then choosing $x=\frac{\alpha_3}{\alpha_6},\ y=1,$ we have the representative $\left\langle\nabla_3+\nabla_4+\nabla_6\right\rangle_{\alpha\notin \{-1,-\frac{1}{2}\}};$
                \item $\alpha_3\neq0,\ \alpha_2\neq0,$ then choosing $x=\frac{\alpha_2\alpha_3}{\alpha_6^2},\ y=\frac{\alpha_2}{\alpha_6},$ we have the representative $\left\langle\nabla_2+\nabla_3+\nabla_4+ \nabla_6\right\rangle_{\alpha\notin \{-1,-\frac{1}{2}\}}.$
            \end{enumerate}
            \end{enumerate}
           \item $\alpha_4=-\alpha\alpha_6,$ $\alpha\neq-1,$ then choosing $u=0,\ w=0,\ t=\frac{x \left(z(1-\alpha ) \alpha _1+y \alpha _2\right)}{y (1+\alpha ) \alpha _6},$ we have $\alpha_2^*=0.$  Hence, 
            \begin{enumerate}
            \item $\alpha =-\frac{1}{2},\ \alpha_5=\alpha_1=\alpha_3=0,$ then we have the representative $\left\langle\frac{1}{2}\nabla_4+\nabla_6\right\rangle_{\alpha=-\frac{1}{2}};$ 
            \item $\alpha =-\frac{1}{2},\ \alpha_5=\alpha_1=0,\ \alpha_3\neq0,$ then choosing $x=\frac{\alpha_3}{\alpha_6},\ y=1,$ we have the representative $\left\langle\nabla_3+\frac{1}{2}\nabla_4+\nabla_6\right\rangle_{\alpha=-\frac{1}{2}};$
            \item $\alpha =-\frac{1}{2},\ \alpha_5=0,\ \alpha_1\neq0,\ \alpha_3=0$ then choosing $x=\frac{\alpha_6}{\alpha_1},\ y=1,$ we have the representative $\left\langle\nabla_1+\frac{1}{2}\nabla_4+\nabla_6\right\rangle_{\alpha=-\frac{1}{2}};$
            \item $\alpha =-\frac{1}{2},\ \alpha_5=0, \alpha_1\neq0, \ \alpha_3\neq0,$ then choosing $x=\frac{\alpha_1\alpha_3^2}{\alpha_6^3},\ y=\frac{\alpha_1\alpha_3}{\alpha_6^2},$ we have the representative $\left\langle\nabla_1+\nabla_3+\frac{1}{2}\nabla_4+\nabla_6\right\rangle_{\alpha=-\frac{1}{2}};$
            \item $\alpha =-\frac{1}{2},\ \alpha_5\neq0,\ \alpha_1=0,$ then choosing $x=\frac{\alpha_5}{\alpha_6},\ y=1,$ we have the family of  representatives $\left\langle\beta\nabla_3+\frac{1}{2}\nabla_4+ \nabla_5+\nabla_6\right\rangle_{\alpha=-\frac{1}{2}};$
            \item $\alpha =-\frac{1}{2},\ \alpha_5\neq0,\ \alpha_1\neq0,$ then choosing $x=\frac{\alpha_1\alpha_5^2}{\alpha_6^3},\ y=\frac{\alpha_1\alpha_5}{\alpha_6^2},$ we have the family of representatives $\left\langle\nabla_1+\beta\nabla_3+\frac{1}{2}\nabla_4+\nabla_5+\nabla_6\right \rangle_{\alpha=-\frac{1}{2}};$
            \item $\alpha \neq-\frac{1}{2},$ then choosing $z=-\frac{y\alpha_5} {(1+2\alpha)\alpha_6},$ we have $\alpha_5^*=0.$ 
           
           \begin{enumerate}
                \item $\alpha_1=0,\ \alpha_3=0,$ then we have the family of representatives $\left\langle-\alpha\nabla_4+\nabla_6\right\rangle_{\alpha \notin\{-1,-\frac 1 2\}};$
           
                \item $\alpha_1=0,\  \alpha_3\neq0,$ then choosing $x=\frac{\alpha_3}{\alpha_6},\ y=1,$ we have the family of representatives $\left\langle\nabla_3-\alpha\nabla_4+\nabla_6\right\rangle_{\alpha \notin\{-1,-\frac 1 2\}};$
                \item $\alpha_1\neq0,\  \alpha_3=0,$ then choosing $x=\frac{\alpha_6}{\alpha_1},\ y=1,$ we have the family of representatives $\left\langle\nabla_1-\alpha\nabla_4+\nabla_6\right\rangle_{\alpha \notin\{-1,-\frac 1 2\}};$
                \item $\alpha_1\neq0,\  \alpha_3\neq0,$ then choosing $x=\frac{\alpha_1\alpha_3^2}{\alpha_6^3},\ y=\frac{\alpha_1\alpha_3}{\alpha_6^2},$ we have the representative $\left\langle\nabla_1+\nabla_3- \alpha\nabla_4+\nabla_6\right\rangle_{\alpha \neq-1,-\frac 1 2}.$
            \end{enumerate}
            
        \end{enumerate}
        \item $\alpha _4=-\alpha\alpha_6,\ \alpha=-1,$ then choosing $z=\frac{y \alpha _5}{\alpha _6},$ we have $\alpha_5^*=0.$
        \begin{enumerate}
            \item $\alpha_3=\alpha_1=\alpha_2=0,$ then we have the representative $\left\langle\nabla_4+\nabla_6\right \rangle_{\alpha =-1};$
            \item $\alpha_3=\alpha_1=0,\ \alpha_2\neq0,$ then choosing $y=\frac{\alpha_2}{\alpha_6},$ we have the representative $\left\langle\nabla_2+\nabla_4 +\nabla_6\right\rangle_{\alpha=-1};$
            \item $\alpha_3=0,\ \alpha_1\neq0,\ \alpha_2=0,$ then choosing $x=\frac{\alpha_6}{\alpha_1},\ y=1,$ we have the representative $\left\langle\nabla_1+\nabla_4 +\nabla_6\right\rangle_{\alpha=-1};$
            \item $\alpha_3=0,\ \alpha_1\neq0,\ \alpha_2\neq0,$ then choosing $x=\frac{\alpha_2^2}{\alpha_1\alpha_6},\ y=\frac{\alpha_2}{\alpha_6},$ we have the representative $\left\langle\nabla_1+\nabla_2+\nabla_4 +\nabla_6\right\rangle_{\alpha=-1};$
            \item $\alpha_3\neq0,\ \alpha_1=0,$ then choosing $x=1,\ y=\frac{\alpha_6} {\alpha_3}, \ w=-\frac{\alpha_2}{\alpha_3},$ we have the representative $\left\langle\nabla_3+\nabla_4 +\nabla_6\right\rangle_{\alpha=-1};$
            \item $\alpha_3\neq0,\ \alpha_1\neq0,$ then choosing $x=\frac{\alpha_1\alpha_3^2}{\alpha_6^3},\ y=\frac{\alpha_1\alpha_3}{\alpha_6^2},\  w=-\frac{\alpha_1\alpha_2\alpha_3}{\alpha_6^3},$ we have the representative $\left\langle\nabla_1+\nabla_3+ \nabla_4+\nabla_6\right\rangle_{\alpha=-1}.$
        \end{enumerate} 
        
    \end{enumerate}
    \item $\alpha_6=0,$ then $\alpha_4\neq0$ and choosing $w=-\frac{x \alpha _1}{\alpha _4},$ 
    $ t=\frac{x (\alpha _4(u \alpha \alpha_4-y \alpha _2)+\alpha_1(y \alpha _3+z \alpha  \alpha_4-y \alpha \alpha _5))}{y \alpha _4^2},$ we have $\alpha_1^*=\alpha_2^*=0.$
    Thus, we can suppose $\alpha_1=\alpha_2=0$ and consider following subcases:
    \begin{enumerate}
        \item $\alpha=-2,\ \alpha_3=\alpha_5=0,$ then we have the representative  $\left\langle\nabla_4\right\rangle;$
        \item $\alpha=-2,\ \alpha_3=0,\ \alpha_5\neq0,$ then choosing $x=\frac{\alpha_5}{\alpha_4},\ y=1,$ we have the representative  $\left\langle\nabla_4+\nabla_5\right\rangle;$
        \item $\alpha=-2,\ \alpha_3\neq0,$ then choosing $x=\frac{\alpha_3}{\alpha_4},\ y=1,$ we have the family of representatives $\left\langle\nabla_3+\nabla_4+\beta\nabla_5\right\rangle;$
        \item $\alpha\neq-2,\ \alpha_5=0,$ then choosing $y=1,\ z=-\frac{\alpha _3}{(2+\alpha ) \alpha _4},$ we have the representative $\left\langle\nabla_4\right\rangle_{\alpha\neq-2};$
        \item $\alpha\neq-2,\ \alpha_5\neq0,$ then choosing $x=\frac{\alpha_5}{\alpha_4},\ y=1,\ z=-\frac{\alpha _3}{(2+\alpha)\alpha_4},$ we have the representative $\left\langle\nabla_4+\nabla_5\right\rangle_{\alpha\neq-2}.$
    \end{enumerate}
    
\end{enumerate}

Summarizing all cases for the family of algebras ${\mathfrak N}_{14}^{\alpha},$ we have the following distinct orbits:

\begin{center} 
$\left\langle \nabla_3+\beta\nabla_4+\gamma\nabla_6+\nabla_7\right\rangle^{O(\beta, \gamma) \simeq O(-\beta, -\gamma)}_{\alpha=0},$ 
$\left\langle\nabla_2+\nabla_5+\beta\nabla_6+\nabla_7\right\rangle^{O(\beta) \simeq O(-\beta)}_{\alpha=0},$
$\left\langle\nabla_5+\beta\nabla_6+\nabla_7\right\rangle^{O(\beta) \simeq O(-\beta)}_{\alpha=0},$
$\langle\nabla_1+\nabla_3+\nabla_4+\nabla_6\rangle_{\alpha=1},$  
$\langle\nabla_1+\nabla_2+\nabla_4+\nabla_6\rangle_{\alpha= -1},$ 
$\langle\nabla_3+\nabla_4+\beta\nabla_5\rangle_{\alpha=-2},$
$\langle\nabla_1+\beta\nabla_3+\frac12\nabla_4+\nabla_5+\nabla_6\rangle_{\alpha=-\frac12},$ 
$\langle\beta\nabla_3+\frac12\nabla_4+\nabla_5+\nabla_6\rangle_{\alpha=-\frac12, \beta\neq0},$
$\langle\nabla_2+\nabla_4+\nabla_5+\nabla_6\rangle_{\alpha=-\frac12},$  $\langle\beta\nabla_4+\nabla_5+\nabla_6\rangle_{\alpha=-\frac12},$
$\langle\nabla_2+\nabla_3+\nabla_4+\nabla_6\rangle_{\alpha \neq -\frac 1 2},$ $\langle\nabla_2+\nabla_4+\nabla_6\rangle,$ 
$\langle\nabla_1+\nabla_3-\alpha\nabla_4+\nabla_6\rangle,$
$\langle\nabla_1-\alpha\nabla_4+\nabla_6\rangle_{\alpha\neq 0},$
$\langle\nabla_3+\beta\nabla_4+\nabla_6\rangle,$  $\langle\beta\nabla_4+\nabla_6\rangle,$ 
$\langle\nabla_4+\nabla_5\rangle,$  $\langle\nabla_4\rangle,$
\end{center}
which gives the following new algebras (see section \ref{secteoA}):

\begin{center}
${\rm N}_{101}^{\beta, \gamma},$
${\rm N}_{102}^{\beta},$
${\rm N}_{103}^{\beta},$
${\rm N}_{104},$
${\rm N}_{105},$
${\rm N}_{106}^{\beta},$
${\rm N}_{107}^{\beta},$
${\rm N}_{108}^{\beta},$
${\rm N}_{109},$
${\rm N}_{110}^{\beta},$
${\rm N}_{111}^{\alpha (\alpha \neq -\frac 1 2)},$

${\rm N}_{112}^{\alpha},$
${\rm N}_{113}^{\alpha},$
${\rm N}_{114}^{\alpha(\alpha \neq 0)},$
${\rm N}_{115}^{\alpha, \beta},$
${\rm N}_{116}^{\alpha, \beta},$
${\rm N}_{117}^{\alpha},$
${\rm N}_{118}^{\alpha}.$
\end{center}

\subsection{$1$-dimensional central extensions of  two-generated
$4$-dimensional $3$-step nilpotent  Novikov algebras}

\subsubsection{The description of second cohomology space}
In the following table, we give the description of the second cohomology space of  two-generated $4$-dimensional $3$-step
  nilpotent Novikov algebras.

\begin{longtable}{llllllll}

\hline

${\mathcal N}^4_{01}$ $:$ \quad $e_1 e_1 = e_2$  \quad $e_2 e_1=e_3$  \\

\multicolumn{8}{l}{
${\rm H}^2({\mathcal N}^4_{01})=
\Big\langle   [\Delta_{12}], [\Delta_{13}-\Delta_{31}], [\Delta_{14}], [\Delta_{41}],[\Delta_{44}] \Big\rangle $}\\
\hline
${\mathcal N}^4_{02}(\lambda)$ $:$\quad $e_1 e_1 = e_2$ \quad $e_1 e_2=e_3$ \quad $e_2 e_1=\lambda e_3$   \\

\multicolumn{8}{l}{
${\rm H}^2({\mathcal N}^4_{02}(\lambda\neq1))=
\Big\langle [\Delta_{14}], [\Delta_{21}], [(2-\lambda)\Delta_{13}+\lambda\Delta_{22}+\lambda\Delta_{31}], [\Delta_{41}],[\Delta_{44}] 
\Big\rangle $}\\

\multicolumn{8}{l}{
${\rm H}^2_{Comm}({\mathcal N}^4_{02}(1))=
\Big\langle [\Delta_{13}+\Delta_{22}+\Delta_{31}], [\Delta_{14}+\Delta_{41}],[\Delta_{44}] 
\Big\rangle $}\\

\multicolumn{8}{l}{
${\rm H}^2({\mathcal N}^4_{02}(1))={\rm H}^2_{Comm}({\mathcal N}^4_{02}(1))\oplus \Big\langle [\Delta_{21}], [\Delta_{41}] 
\Big\rangle $}\\

\hline

${\mathcal N}^4_{04}(\alpha)$ $:$ \quad
$e_1 e_1 = e_2$ \quad   $e_1e_2=e_4$ \quad $e_2e_1=\alpha e_4$ \quad $e_3e_3=e_4$  \\

\multicolumn{8}{l}{
${\rm H}^2({\mathcal N}^4_{04})=
\Big\langle  [\Delta_{13}], [\Delta_{21}], [\Delta_{31}], [\Delta_{33}] 
\Big\rangle $}\\
\hline

${\mathcal N}^4_{05}$ $:$\quad
$e_1 e_1 = e_2$ \quad   $e_1e_2=e_4$ \quad $e_1e_3= e_4$ \quad  $e_2e_1= e_4$ \quad $e_3e_3=e_4$ \\

\multicolumn{8}{l}{
${\rm H}^2({\mathcal N}^4_{05})=
\Big\langle [\Delta_{13}], [\Delta_{21}], [\Delta_{31}], [\Delta_{33}]  
\Big\rangle $}\\
\hline

${\mathcal N}^4_{06}(\alpha)_{\alpha\neq0}$ $:$\quad
$e_1 e_1 = e_2$ \quad $e_1e_2=e_4$ \quad $e_1e_3=e_4$ \quad $e_2e_1=\alpha e_4$  \\

\multicolumn{8}{l}{
${\rm H}^2({\mathcal N}^4_{06}(\alpha))=
\Big\langle [\Delta_{13}], [\Delta_{21}], [\Delta_{31}], [\frac{2-\alpha}{\alpha}\Delta_{14}+\Delta_{22}+\Delta_{23}-\Delta_{32}+\Delta_{41}], [\Delta_{33}] 
\Big\rangle $}\\

\hline

${\mathcal N}^4_{07}$ \quad $:$ \quad
$e_1 e_1 = e_2$ \quad    $e_2e_1= e_4$ \quad $e_3e_3=e_4$ \\

\multicolumn{8}{l}{
${\rm H}^2({\mathcal N}^4_{07})=
\Big\langle  [\Delta_{12}], [\Delta_{13}], [\Delta_{31}], [\Delta_{33}]
\Big\rangle $}\\
\hline

${\mathcal N}^4_{08}$ $:$ \quad
$e_1 e_1 = e_2$ \quad $e_1e_3=e_4$ \quad    $e_2e_1= e_4$  \\

\multicolumn{8}{l}{
${\rm H}^2({\mathcal N}^4_{08})=
\Big\langle [\Delta_{12}], [\Delta_{21}], [\Delta_{31}], [\Delta_{33}], [-\Delta_{14}+\Delta_{23}-\Delta_{32}+\Delta_{41}] 
\Big\rangle $}\\
\hline

${\mathcal N}^4_{09}$ $:$ \quad
$e_1 e_1 = e_2$ \quad $e_1e_2=e_4$ \quad    $e_3e_1= e_4$  \\

\multicolumn{8}{l}{
${\rm H}^2({\mathcal N}^4_{09})=
\Big\langle [\Delta_{13}], [\Delta_{21}], [\Delta_{31}], [\Delta_{33}], [\Delta_{14}+\Delta_{32}] 
\Big\rangle $}\\
\hline

${\mathcal N}^{4}_{10}$ $:$ \quad $e_1 e_2=e_3$ \quad $e_1 e_3=e_4$\\

\multicolumn{8}{l}{
${\rm H}^2({\mathcal N}^4_{10})=
\Big\langle [\Delta_{11}], [\Delta_{14}], [\Delta_{21}], [\Delta_{22}], [\Delta_{23}-\Delta_{32}]  
\Big\rangle $}\\
\hline

${\mathcal N}^{4}_{11}$ $:$ \quad   $e_1e_2=e_3$ \quad $e_1 e_3=e_4$  \quad $e_2 e_1=e_4$ \\

\multicolumn{8}{l}{
${\rm H}^2({\mathcal N}^4_{11})=
\Big\langle  [\Delta_{11}], [\Delta_{21}], [\Delta_{22}], [\Delta_{23}-\Delta_{32}]  
\Big\rangle $}\\
\hline

${\mathcal N}^{4}_{12}$ $:$ \quad  $e_1e_2=e_3$ \quad $e_2 e_3=e_4$  \quad $e_3 e_2=-e_4$ \\

\multicolumn{8}{l}{
${\rm H}^2({\mathcal N}^4_{12})=
\Big\langle [\Delta_{11}], [\Delta_{13}], [\Delta_{21}], [\Delta_{22}] 
\Big\rangle $}\\
\hline

${\mathcal N}^{4}_{13}$ $:$\quad  $e_1e_2=e_3$ \quad $e_1 e_1=e_4$ \quad $e_2 e_3=e_4$  \quad $e_3 e_2=-e_4$ \\

\multicolumn{8}{l}{
${\rm H}^2({\mathcal N}^4_{13})=
\Big\langle   [\Delta_{11}], [\Delta_{13}], [\Delta_{21}], [\Delta_{22}]  
\Big\rangle $}\\
\hline

${\mathcal N}^{4}_{14}$ $:$\quad  $e_1e_2=e_3$ \quad $e_1 e_3=e_4$ \quad $e_2 e_3=e_4$  \quad $e_3 e_2=-e_4$ \\

\multicolumn{8}{l}{
${\rm H}^2({\mathcal N}^4_{14})=
\Big\langle  [\Delta_{11}], [\Delta_{13}], [\Delta_{21}], [\Delta_{22}] 
\Big\rangle $}\\
\hline

${\mathcal N}^{4}_{15}$ $:$ \quad  $e_1e_2=e_3$ \quad $e_1 e_1=e_4$ \quad  $e_1 e_3=e_4$ \quad $e_2 e_3=e_4$ \quad $e_3 e_2=-e_4$\\

\multicolumn{8}{l}{
${\rm H}^2({\mathcal N}^4_{15})=
\Big\langle  [\Delta_{11}], [\Delta_{13}], [\Delta_{21}], [\Delta_{22}] 
\Big\rangle $}\\
\hline

${\mathcal N}^{4}_{16}$ $:$ \quad  $e_1e_2=e_3$ \quad  $e_1 e_3=e_4$ \quad $e_2 e_2=e_4$  \\

\multicolumn{8}{l}{
${\rm H}^2({\mathcal N}^4_{16})=
\Big\langle  [\Delta_{11}], [\Delta_{21}], [\Delta_{22}], [\Delta_{14}+\Delta_{23}], [-\Delta_{23}+\Delta_{32}] 
\Big\rangle $}\\
\hline

${\mathcal N}^{4}_{17}$  $:$ \quad  $e_1e_2=e_3$ \quad  $e_1 e_3=e_4$ \quad $e_2 e_1=e_4$ \quad $e_2 e_2=e_4$  \\

\multicolumn{8}{l}{
${\rm H}^2({\mathcal N}^4_{17})=
\Big\langle   [\Delta_{11}], [\Delta_{21}], [\Delta_{22}], [\Delta_{23}-\Delta_{32}],
\Big\rangle $}\\
\hline

${\mathcal N}^{4}_{18}$ $:$\quad  $e_1e_2=e_3$ \quad  $e_2 e_2=e_4$ \quad $e_2 e_3=e_4$  \quad $e_3 e_2=-e_4$  \\

\multicolumn{8}{l}{
${\rm H}^2({\mathcal N}^4_{18})=
\Big\langle  [\Delta_{11}], [\Delta_{13}], [\Delta_{21}], [\Delta_{22}]
\Big\rangle $}\\
\hline

${\mathcal N}^{4}_{19}$ $:$ \quad  $e_1e_2=e_3$ \quad $e_1 e_1=e_4$ \quad  $e_2 e_2=e_4$ \quad $e_2 e_3=e_4$  \quad $e_3 e_2=-e_4$  \\

\multicolumn{8}{l}{
${\rm H}^2({\mathcal N}^4_{19})=
\Big\langle [\Delta_{11}], [\Delta_{13}], [\Delta_{21}], [\Delta_{22}] 
\Big\rangle $}\\
\hline

${\mathcal N}^{4}_{20}(\alpha)$  $:$\quad  $e_1e_2=e_3$ \quad $e_1 e_1=\alpha e_4$ \quad  $e_1 e_3=e_4$ \quad $e_2 e_2=e_4$ \quad $e_2e_3=e_4$  \quad $e_3 e_2=-e_4$ \\
 
\multicolumn{8}{l}{
${\rm H}^2({\mathcal N}^4_{20})=
\Big\langle  [\Delta_{11}], [\Delta_{13}], [\Delta_{21}], [\Delta_{22}] 
\Big\rangle $}\\
\hline
\end{longtable}
\subsubsection{Central extensions of ${\mathcal N}^4_{01}$}
	Let us use the following notations:
	\begin{longtable}{lllllll} 
	$\nabla_1 = [\Delta_{12}],$ & $\nabla_2 = [\Delta_{13}-\Delta_{31}],$ &
	$\nabla_3 = [\Delta_{14}],$ &
	$\nabla_4 = [\Delta_{41}],$&
	$\nabla_5 = [\Delta_{44}].$ 
	\end{longtable}	
	
Take $\theta=\sum\limits_{i=1}^{5}\alpha_i\nabla_i\in {\rm H^2}({\mathcal N}^4_{01}).$
	The automorphism group of ${\mathcal N}^4_{01}$ consists of invertible matrices of the form
	$$\phi=
	\begin{pmatrix}
	x &  0  & 0 & 0\\
	y &  x^2  & 0 & 0\\
	z &  xy  & x^3 & t\\
    u &  0  & 0 & r
	\end{pmatrix}.
	$$
	Since
	$$
	\phi^T\begin{pmatrix}
	0 & \alpha_1 & \alpha_2 & \alpha_3\\
	0 &	0 & 0 & 0\\
	-\alpha_2 &  0  & 0 & 0\\
	\alpha_4 & 0  & 0 & \alpha_5
	\end{pmatrix} \phi=\begin{pmatrix}
	\alpha^* & \alpha_1^* & \alpha_2^* & \alpha_3^*\\
	\alpha^{**} &	0 & 0 & 0\\
	-\alpha_2^* &  0  & 0 & 0\\
	\alpha_4^* & 0  & 0 & \alpha_5^*
	\end{pmatrix},
	$$
	 we have that the action of ${\rm Aut} ({\mathcal N}^4_{01})$ on the subspace
$\langle \sum\limits_{i=1}^{5}\alpha_i\nabla_i  \rangle$
is given by
$\langle \sum\limits_{i=1}^{5}\alpha_i^{*}\nabla_i\rangle,$
where
\begin{longtable}{lcllcllcl}
$\alpha^*_1$&$=$&$x^2(x \alpha _1+y \alpha _2),$ &
$\alpha_2^*$&$=$&$x^4 \alpha _2,$ &
$\alpha^*_3$&$=$&$t x \alpha _2+r x \alpha _3+r u \alpha _5,$ \\
$\alpha^*_4$&$=$&$-t x \alpha _2+r x \alpha _4+r u \alpha _5,$ &
$\alpha_5^*$&$=$&$r^2 \alpha _5.$\\
\end{longtable}

We are interested only in the cases with 
 \begin{center}
$(\alpha_3,\alpha_4,\alpha_5)\neq (0,0,0),$  $\alpha_2\neq 0.$ 
 \end{center} 
Since $\alpha_2\neq0,$ then choosing $y=-\frac{x \alpha _1}{\alpha _2},\ t=-\frac{r \left(x \alpha _3+u \alpha _5\right)}{x \alpha _2},$ we have $\alpha_1^*=\alpha_3^*=0.$
\begin{enumerate}
    \item If $\alpha_5\neq0,$ then choosing $x=1,\ u=-\frac{\alpha _4}{2 \alpha _5},\ r=\sqrt{\frac{\alpha_2}{\alpha_5}},$ we have the representative  $\left\langle\nabla_2+\nabla_5\right\rangle.$
    \item If $\alpha_5=0,$ then  $\alpha_4\neq0$ and choosing $x=1,\ r=-\frac{\alpha_2}{\alpha_4},$ we have the representative  $\left\langle\nabla_2+\nabla_4\right\rangle.$  
\end{enumerate}

Therefore, we have the following distinct orbits 
\begin{longtable} {llll}
$\langle\nabla_2+\nabla_5\rangle,$ & $\langle\nabla_2+\nabla_4\rangle,$ \\
\end{longtable}
which gives the following new algebras (see section \ref{secteoA}):

\begin{center}
${\rm N}_{119},$
${\rm N}_{120},$
\end{center}

\subsubsection{Central extensions of ${\mathcal N}^4_{02}(\lambda\neq1)$}
	Let us use the following notations:
	\begin{longtable}{lllllll} 
	$\nabla_1 = [\Delta_{14}],$ & $\nabla_2 = [\Delta_{21}],$ &
	$\nabla_3 = [(2-\lambda)\Delta_{13}+\lambda\Delta_{22}+\lambda\Delta_{31}],$ &
	$\nabla_4 = [\Delta_{41}],$&
	$\nabla_5 = [\Delta_{44}].$ 
	\end{longtable}	
	
Take $\theta=\sum\limits_{i=1}^{5}\alpha_i\nabla_i\in {\rm H^2}({\mathcal N}^4_{02}(\lambda\neq1)).$
	The automorphism group of ${\mathcal N}^4_{02}(\lambda\neq1)$ consists of invertible matrices of the form
	$$\phi=
	\begin{pmatrix}
	x &  0  & 0 & 0\\
	y &  x^2  & 0 & 0\\
	z &  (1+\lambda)xy  & x^3 & t\\
    u &  0  & 0 & r
	\end{pmatrix}.
	$$
	Since
	$$
	\phi^T\begin{pmatrix}
	0 & 0 & (2-\lambda)\alpha_3 & \alpha_1\\
	\alpha_2 &  \lambda\alpha_3  & 0 & 0\\
	\lambda\alpha_3 &	0 & 0 & 0\\
	\alpha_4 & 0  & 0 & \alpha_5
	\end{pmatrix} \phi=\begin{pmatrix}
	\alpha^* & \alpha^{**} & (2-\lambda)\alpha_3^* & \alpha_1^*\\
	\alpha_2^{*}+\lambda \alpha^{**} &	\lambda\alpha_3^* & 0 & 0\\
	\lambda\alpha_3^* &  0  & 0 & 0\\
	\alpha_4^* & 0  & 0 & \alpha_5^*
	\end{pmatrix},
	$$
	 we have that the action of ${\rm Aut} ({\mathcal N}^4_{02}(\lambda\neq1))$ on the subspace
$\langle \sum\limits_{i=1}^{5}\alpha_i\nabla_i  \rangle$
is given by
$\langle \sum\limits_{i=1}^{5}\alpha_i^{*}\nabla_i\rangle,$
where
\begin{longtable}{lcllcllcl}
$\alpha^*_1$&$=$&$r x \alpha _1+t x (2-\lambda ) \alpha _3+r u \alpha _5,$ &
$\alpha_2^*$&$=$&$x^2 \left(x \alpha _2-y (1-\lambda ) \lambda ^2 \alpha _3\right),$&
$\alpha^*_3$&$=$&$x^4 \alpha _3,$ \\
$\alpha^*_4$&$=$&$t x \lambda  \alpha _3+r x \alpha _4+r u \alpha _5,$ &
$\alpha_5^*$&$=$&$r^2 \alpha _5.$\\
\end{longtable}

We are interested only in the cases with 
 \begin{center}
$\alpha_3\neq 0, \ $   $(\alpha_1,\alpha_4,\alpha_5)\neq (0,0,0).$
 \end{center} 
\begin{enumerate}
\item $\lambda=0,$ then choosing $t=-\frac{r \left(x \alpha _1+u \alpha _5\right)}{2 x \alpha _3},$ we have $\alpha_1^*=0.$
\begin{enumerate}
    \item $\alpha_5=\alpha_2=0,\ \alpha_4\neq0,$ then choosing $x=1,\ r=\frac{\alpha_3}{\alpha_4},$ we have the representative  $\left\langle\nabla_3+\nabla_4\right\rangle;$
    \item $\alpha_5=0,\ \alpha_2\neq0,\ \alpha_4=0,$ then choosing $x=\frac{\alpha_2}{\alpha_3},$ we have the representative  $\left\langle\nabla_2+\nabla_3\right\rangle;$
    \item $\alpha_5=0,\ \alpha_2\neq0,\ \alpha_4\neq0,$ then choosing $x=\frac{\alpha_2}{\alpha_3},\ r=\frac{\alpha_2^3}{\alpha_3^2\alpha_4},$ we have the representative  $\left\langle\nabla_2+\nabla_3+\nabla_4\right\rangle;$
    \item $\alpha_5\neq0,\ \alpha_2=0,$ then choosing $x=1,\ r=\sqrt{\frac{\alpha_3}{\alpha_4}},\ u=-\frac{\alpha_4}{\alpha_5},$ we have the representative  $\left\langle\nabla_3+\nabla_5\right\rangle;$
    \item $\alpha_5\neq0,\ \alpha_2\neq0,$ then choosing $x=\frac{\alpha_2}{\alpha_3},\ r=\frac{\alpha_2^2}{\alpha_3^2}\sqrt{\frac{\alpha_3}{\alpha_4}},\ u=-\frac{\alpha_2\alpha_4}{\alpha_3\alpha_5},$ we have the representative  $\left\langle\nabla_2+\nabla_3+\nabla_5\right\rangle.$
\end{enumerate}
\item $\lambda\neq0,$ then choosing $y=\frac{x \alpha _2}{(1-\lambda ) \lambda ^2 \alpha _3},\  t=-\frac{r \left(x \alpha _4+u \alpha _5\right)}{x \lambda  \alpha _3},$ we have $\alpha_2^*=\alpha_4^*=0.$ 
\begin{enumerate}
    \item $\alpha_5\neq0,$ then choosing $u=\frac{\lambda  \alpha _1}{2(1-\lambda)\alpha _5},\ x=1,\ r=\sqrt{\frac{\alpha_3}{\alpha_5}},$ we have the representative  $\left\langle\nabla_3+\nabla_5\right\rangle_{\lambda\neq0},$
    \item $\alpha_5=0,\ \alpha_1\neq0,$ then choosing $x=1,\ r=\frac{\alpha_3}{\alpha_1},$ we have the representative $\left\langle\nabla_1+\nabla_3\right\rangle_{\lambda\neq0}.$
\end{enumerate}
\end{enumerate}

Summarizing all cases, we have the following distinct orbits:
\begin{center} 
$\langle\nabla_3+\nabla_4\rangle_{\lambda=0},$  
$\langle\nabla_2+\nabla_3\rangle_{\lambda=0},$  $\langle\nabla_2+\nabla_3+\nabla_4\rangle_{\lambda=0},$

$\langle\nabla_2+\nabla_3+\nabla_5\rangle_{\lambda=0},$ 
$\langle\nabla_3+\nabla_5\rangle_{\lambda\neq1},$  $\langle\nabla_1+\nabla_3\rangle_{\lambda\neq 0; 1},$
\end{center}
which gives the following new algebras (see section \ref{secteoA}):

\begin{center}
${\rm N}_{121},$
${\rm N}_{122},$
${\rm N}_{123},$
${\rm N}_{124},$
${\rm N}_{125}^{\lambda\neq1},$
${\rm N}_{126}^{\lambda\neq0; 1}.$
\end{center}

\subsubsection{Central extensions of ${\mathcal N}^4_{02}( 1)$}

Let us use the following notations:
	\begin{longtable}{lllllll} 
	$\nabla_1 = [\Delta_{14}+\Delta_{41}],$ & $\nabla_2=[\Delta_{21}],$ &
	$\nabla_3 = [\Delta_{13}+\Delta_{22}+\Delta_{31}],$ &
	$\nabla_4 = [\Delta_{41}],$&
	$\nabla_5 = [\Delta_{44}].$ 
	\end{longtable}	
	
Take $\theta=\sum\limits_{i=1}^{5}\alpha_i\nabla_i\in {\rm H^2}({\mathcal N}^4_{02}(1)).$
	The automorphism group of ${\mathcal N}^4_{02}(1)$ consists of invertible matrices of the form
$$\phi=
	\begin{pmatrix}
	x &  0  & 0 & 0\\
	y &  x^2  & 0 & 0\\
	z &  xy  & x^3 & t\\
    u &  0  & 0 & r
	\end{pmatrix}.
	$$
	Since
	$$
	\phi^T\begin{pmatrix}
	0 & 0 & \alpha_3 & \alpha_1\\
	\alpha_2 &  \alpha_3  & 0 & 0\\
	\alpha_3 &	0 & 0 & 0\\
	\alpha_1+\alpha_4 & 0  & 0 & \alpha_5
	\end{pmatrix} \phi=\begin{pmatrix}
	\alpha^* & \alpha^{**} & \alpha_3^* & \alpha_1^*\\
	\alpha_2^{*}+\alpha^{**} &	\alpha_3^* & 0 & 0\\
	\alpha_3^* &  0  & 0 & 0\\
	\alpha_1^*+\alpha_4^* & 0  & 0 & \alpha_5^*
	\end{pmatrix},
$$
we have that the action of ${\rm Aut} ({\mathcal N}^4_{02}( 1))$ on the subspace
$\langle \sum\limits_{i=1}^{5}\alpha_i\nabla_i  \rangle$
is given by
$\langle \sum\limits_{i=1}^{5}\alpha_i^{*}\nabla_i\rangle,$
where
\begin{longtable}{lcllcllcl}
$\alpha^*_1$&$=$&$r x \alpha _1+t x \alpha _3+r u \alpha _5,$ &
$\alpha_2^*$&$=$&$x^3 \alpha _2,$\ &
$\alpha^*_3$&$=$&$x^4 \alpha _3,$ \\
$\alpha^*_4$&$=$&$rx\alpha_4,$ &
$\alpha_5^*$&$=$&$r^2 \alpha _5.$\\
\end{longtable}

We are interested only in the cases with 
 \begin{center}
$\alpha_3\neq 0, \ $   $(\alpha_1,\alpha_4,\alpha_5)\neq (0,0,0), \ $ $(\alpha_2,\alpha_4)\neq (0,0).$ 
 \end{center} 
$\alpha_3\neq0,$ then choosing $t=-\frac{r \left(x \alpha _1+u \alpha _5\right)}{x \alpha _3},$ we have $\alpha_1^*=0.$
\begin{enumerate}
    \item $\alpha_4=0,$ then $\alpha_5\neq0,\ \alpha_2\neq0.$ Choosing $x=\frac{\alpha _2}{\alpha _3},\ r=\frac{\alpha _2^2}{\alpha_3\sqrt{\alpha_3\alpha _5}},$ we have the representative  $\left\langle\nabla_2+\nabla_3+\nabla_5\right\rangle;$
    \item $\alpha_4\neq0,\ \alpha_2=0,\ \alpha_5=0,$ then choosing $x=1,\ r=\frac{\alpha_3}{\alpha_4},$ we have the representative  $\left\langle\nabla_3+\nabla_4\right\rangle;$
    \item $\alpha_4\neq0,\ \alpha_2=0,\ \alpha_5\neq0$ then choosing $x=\frac{\alpha _4}{\sqrt{\alpha_3\alpha_5}},\ r=\frac{\alpha_4^2}{\sqrt{\alpha_3\alpha_5^3}},$ we have the representative  $\left\langle\nabla_3+\nabla_4+\nabla_5\right\rangle;$
    \item $\alpha_4\neq0,\ \alpha_2\neq0,$ then choosing $x=\frac{\alpha_2}{\alpha_3},\ r=\frac{\alpha_2^3}{\alpha_3^2\alpha_4},$ we have the representative  $\left\langle\nabla_2+\nabla_3+\nabla_4+\alpha\nabla_5\right\rangle.$
\end{enumerate}

Summarizing all cases, we have the following distinct orbits
\begin{longtable} {llll}
$\langle\nabla_2+\nabla_3+\nabla_4\rangle,$ & $\langle\nabla_3+\nabla_4\rangle,$ & $\langle\nabla_3+\nabla_4+\nabla_5\rangle,$ & $\langle\nabla_2+\nabla_3+\nabla_4+\alpha\nabla_5\rangle,$ \\
\end{longtable}
which gives the following new algebras (see section \ref{secteoA}):

\begin{center}
${\rm N}_{127},$
${\rm N}_{128},$
${\rm N}_{129},$
${\rm N}_{130}^{\alpha}.$
\end{center}

\subsubsection{Central extensions of ${\mathcal N}^4_{06}(\alpha\neq0)$}
Let us use the following notations:
	\begin{longtable}{lllllll} 
$\nabla_1 = [\Delta_{13}],$&
$\nabla_2 = [\Delta_{21}],$ & $\nabla_3=[\Delta_{31}],$ &
$\nabla_4=[\frac{2-\alpha}{\alpha}\Delta_{14}+\Delta_{22}+\Delta_{23}-\Delta_{32}+\Delta_{41}],$ &
$\nabla_5 = [\Delta_{33}].$ 
	\end{longtable}	
	
Take $\theta=\sum\limits_{i=1}^{5}\alpha_i\nabla_i\in {\rm H^2}({\mathcal N}^4_{06}(\alpha \neq 0)).$
	The automorphism group of ${\mathcal N}^4_{06}(\alpha \neq 0)$ consists of invertible matrices of the form
$$\phi=
	\begin{pmatrix}
	x &  0  & 0 & 0\\
	y &  x^2  & 0 & 0\\
	z &  0  & x^2 & 0\\
    u &  x((1+\alpha)y+z)  & v & x^3
	\end{pmatrix}.
	$$
	Since
	$$
	\phi^T\begin{pmatrix}
	0 & 0 & \alpha_1 & \frac{2-\alpha}{\alpha}\alpha_4\\
	\alpha_2 &  \alpha_4  & \alpha_4 & 0\\
	\alpha_3 &	-\alpha_4 & \alpha_5 & 0\\
	\alpha_4 & 0  & 0 & 0
	\end{pmatrix} \phi=\begin{pmatrix}
	\alpha^* & \alpha^{**} & \alpha_1^*+\alpha^{**} & \frac{2-\alpha}{\alpha}\alpha_4^*\\
	\alpha_2^*+\alpha \alpha^{**} &  \alpha_4^*  & \alpha_4^* & 0\\
	\alpha_3^* &	-\alpha_4^* & \alpha_5^* & 0\\
	\alpha_4^* & 0  & 0 & 0
	\end{pmatrix},
$$
we have that the action of ${\rm Aut} ({\mathcal N}^4_{06}(\alpha \neq 0))$ on the subspace
$\langle \sum\limits_{i=1}^{5}\alpha_i\nabla_i  \rangle$
is given by
$\langle \sum\limits_{i=1}^{5}\alpha_i^{*}\nabla_i\rangle,$
where
\begin{longtable}{lcl}
$\alpha^*_1$&$=$&${x\left(x^2   \alpha _1-\left(v (\alpha-2 )+x \left(2 z (1-\alpha )+y \left(2+\alpha -\alpha ^2\right)\right)\right) {\alpha^{-1} } \alpha _4+x z   \alpha _5\right)},$ \\
$\alpha_2^*$&$=$&$x^2 \left(x \alpha _2+(2 z-y (1-\alpha )) \alpha  \alpha _4\right),$\\
$\alpha^*_3$&$=$&$x \left(x^2 \alpha _3+(v-x y) \alpha _4+x z \alpha _5\right),$ \\
$\alpha^*_4$&$=$&$x^4 \alpha _4,$\\
$\alpha_5^*$&$=$&$x^4 \alpha_5.$\\
\end{longtable}
We are interested only in the cases with $\alpha_4\neq 0.$ Choosing 
\begin{center}
$z=\frac{1}{2}((1-\alpha)y -\frac{x \alpha _2}{\alpha\alpha_4})$ and 
$v=-\frac{x(2x\alpha\alpha_3 \alpha_4-2y\alpha \alpha_4^2-x \alpha _2 \alpha_5-y(\alpha-1) \alpha  \alpha_4 \alpha_5)}{2\alpha\alpha_4^2},$
\end{center} we have $\alpha_2^*=\alpha_3^*=0.$
\begin{enumerate}
    \item $\alpha=1,$ then choosing $y=\frac{x\alpha_1}{\alpha_4}$ we have the family of representatives  $\left\langle\nabla_4+\beta\nabla_5\right\rangle;$
    \item $\alpha\neq1,$  
   $(\alpha-1)^2\alpha_5=-\alpha_4,\ \alpha_1=0,$ then we have the representative  $\left\langle\nabla_4+\frac{1}{(\alpha-1)^2}\nabla_5\right\rangle;$
    \item $\alpha\neq1,$ $(\alpha-1)^2\alpha_5=-\alpha_4,\ \alpha_1\neq0,$ then choosing $x=\frac{\alpha_1}{\alpha_5},$ we have the representative \\ $\left\langle\nabla_1+\nabla_4+\frac{1}{(\alpha-1)^2}\nabla_5\right\rangle;$
    \item $\alpha\neq1,$ $(\alpha-1)^2\alpha_5\neq-\alpha_4$ then choosing $y=\frac{x \alpha\alpha_1}{\alpha_4+(\alpha-1)^2 \alpha _5},$ we have the family of representatives $\left\langle\nabla_4+\beta\nabla_5\right\rangle_{\beta\neq \frac{1}{(\alpha-1)^2}}.$
\end{enumerate}

Summarizing all cases, we have the following distinct orbits
\begin{center} 
$\langle\nabla_4+\beta\nabla_5\rangle,$  
$\langle\nabla_1+\nabla_4+\frac{1}{(\alpha-1)^2}\nabla_5\rangle_{\alpha\neq1},$
\end{center}
which gives the following new algebras (see section \ref{secteoA}):

\begin{center}
${\rm N}_{131}^{\alpha\neq0}, \beta,$
${\rm N}_{132}^{\alpha\neq0,1}.$
\end{center}

\subsubsection{Central extensions of ${\mathcal N}^4_{08}$}
Let us use the following notations:
	\begin{longtable}{lllllll} 
$\nabla_1 = [\Delta_{12}],$& $\nabla_2 = [\Delta_{21}],$ & $\nabla_3=[\Delta_{31}],$ &
$\nabla_4=[\Delta_{33}],$ &
$\nabla_5=[\Delta_{23}-\Delta_{32}-\Delta_{14}+\Delta_{41}].$
\end{longtable}	
	
Take $\theta=\sum\limits_{i=1}^{5}\alpha_i\nabla_i\in {\rm H^2}({\mathcal N}^4_{08}).$
The automorphism group of ${\mathcal N}^4_{08}$ consists of invertible matrices of the form
$$\phi=
	\begin{pmatrix}
	x &  0  & 0 & 0\\
	y &  x^2  & 0 & 0\\
	z &  0  & x^2 & 0\\
    u &  x(y+z)  & v & x^3
	\end{pmatrix}.
	$$
	Since
	$$
	\phi^T\begin{pmatrix}
	0 & \alpha_1 & 0  & -\alpha_5\\
	\alpha_2 & 0  & \alpha_5 & 0\\
	\alpha_3 &	-\alpha_5 & \alpha_4 & 0\\
	\alpha_5 & 0  & 0 & 0
	\end{pmatrix} \phi=\begin{pmatrix}
	\alpha^* & \alpha_1^* & \alpha^{**}  & -\alpha_5^*\\
	\alpha_2^*+\alpha^{**} & 0  & \alpha_5^* & 0\\
	\alpha_3^* &	-\alpha_5^* & \alpha_4^* & 0\\
	\alpha_5^* & 0  & 0 & 0
	\end{pmatrix},
$$
we have that the action of ${\rm Aut} ({\mathcal N}^4_{08})$ on the subspace
$\langle \sum\limits_{i=1}^{5}\alpha_i\nabla_i  \rangle$
is given by
$\langle \sum\limits_{i=1}^{5}\alpha_i^{*}\nabla_i\rangle,$
where
\begin{longtable}{lcllcllcl}
$\alpha^*_1$&$=$&$x^2(x\alpha_1-(y+2z)\alpha_5),$ &
$\alpha^*_3$&$=$&$ x(x^2 \alpha _3+x z \alpha _4+(v-x y) \alpha _5),$ &
$\alpha_5^*$&$=$&$x^4 \alpha_5.$\\
$\alpha_2^*$&$=$&$x(x^2 \alpha _2-x z \alpha _4+(v+2 x z) \alpha _5),$&                                               
$\alpha^*_4$&$=$&$x^4 \alpha _4,$\\
\end{longtable}
We are interested only in the cases with $\alpha_5\neq 0.$ Choosing 
\begin{center}
    $v=\frac{x \left(x \alpha _1 \left(\alpha _4-2 \alpha _5\right)+\alpha _5 \left(2 y \alpha _5-2 x \alpha _2-y \alpha _4\right)\right)}{2 \alpha _5^2},$ and
$z=\frac{x \alpha _1--y\alpha _5}{2\alpha _5},$

\end{center}we have $\alpha_1^*=\alpha_2^*=0.$
\begin{enumerate}
    \item $\alpha_4=\alpha_3=0,$ then we have the representative  $\left\langle\nabla_5\right\rangle;$
    \item $\alpha_4=0,\ \alpha_3\neq0,$ then choosing $x=\frac{\alpha_3}{\alpha_5},$ we have the representative  $\left\langle\nabla_3+\nabla_5\right\rangle;$
    \item $\alpha_4\neq0,$ then choosing $y=\frac{x \alpha _3}{\alpha _4},$ we have the family of representatives $\left\langle\alpha\nabla_4+\nabla_5\right\rangle_{\alpha\neq0}.$
\end{enumerate}

Summarizing all cases, we have the following distinct orbits
\begin{center} 
$\langle\nabla_3+\nabla_5\rangle,$   $\langle\alpha\nabla_4+\nabla_5\rangle,$\\
\end{center}
which gives the following new algebras (see section \ref{secteoA}):

\begin{center}
${\rm N}_{133}^{\alpha\neq0},$
${\rm N}_{134}^{\alpha\neq0,1}.$
\end{center}

\subsubsection{Central extensions of ${\mathcal N}^4_{09}$}
Let us use the following notations:
	\begin{longtable}{lllllll} 
$\nabla_1 = [\Delta_{13}],$& $\nabla_2 = [\Delta_{21}],$ & $\nabla_3=[\Delta_{31}],$ &
$\nabla_4=[\Delta_{33}],$ &
$\nabla_5=[\Delta_{14}+\Delta_{32}].$
\end{longtable}	
	
Take $\theta=\sum\limits_{i=1}^{5}\alpha_i\nabla_i\in {\rm H^2}({\mathcal N}^4_{09}).$
The automorphism group of ${\mathcal N}^4_{09}$ consists of invertible matrices of the form
$$\phi=
	\begin{pmatrix}
	x &  0  & 0 & 0\\
	y &  x^2  & 0 & 0\\
	z &  0  & x^2 & 0\\
    u &  x(y+z)  & v & x^3
	\end{pmatrix}.
	$$
	Since
	$$
	\phi^T\begin{pmatrix}
	0 & 0 & \alpha_1  & \alpha_5\\
	\alpha_2 & 0  & 0 & 0\\
	\alpha_3 &	\alpha_5 & \alpha_4 & 0\\
	0 & 0  & 0 & 0
	\end{pmatrix} \phi=\begin{pmatrix}
	\alpha^* &  \alpha^{**}  & \alpha_1^* & \alpha_5^*\\
	\alpha_2^* & 0  & 0 & 0\\
	\alpha_3^*+\alpha^{**} & \alpha_5^* & \alpha_4^* & 0\\
	0 & 0 & 0 & 0
	\end{pmatrix},
$$
we have that the action of ${\rm Aut} ({\mathcal N}^4_{09})$ on the subspace
$\langle \sum\limits_{i=1}^{5}\alpha_i\nabla_i  \rangle$
is given by
$\langle \sum\limits_{i=1}^{5}\alpha_i^{*}\nabla_i\rangle,$
where
\begin{longtable}{lcllcllcl}
$\alpha^*_1$&$=$&$x(x^2\alpha_1+xz\alpha_4+v\alpha_5),$ &
$\alpha_2^*$&$=$&$x^3 \alpha _2,$\\
$\alpha^*_3$&$=$&$x^2(x\alpha_3+z(\alpha _4-2 \alpha _5)),$ &
$\alpha^*_4$&$=$&$x^4 \alpha _4,$&
$\alpha_5^*$&$=$&$x^4 \alpha_5.$\\
\end{longtable}
We are interested only in the cases with $\alpha_5\neq 0.$ Choosing $v=-\frac{x \left(x \alpha _1+z \alpha _4\right)}{\alpha _5},$ we have $\alpha_1^*=0.$
\begin{enumerate}
    \item $\alpha_4=2\alpha_5,\ \alpha_3=\alpha_2=0,$ then we have the representative  $\left\langle2\nabla_4+\nabla_5\right\rangle;$
    \item $\alpha_4=2\alpha_5,\ \alpha_3=0,\ \alpha_2\neq0,$ then choosing $x=\frac{\alpha_2}{\alpha_5},$ we have the representative  $\left\langle\nabla_2+2\nabla_4+\nabla_5\right\rangle;$
    \item $\alpha_4=2\alpha_5,\ \alpha_3\neq0,$ then choosing $x=\frac{\alpha_3}{\alpha_5},$ we have the family of representatives $\left\langle\alpha\nabla_2+\nabla_3+ 2\nabla_4+ \nabla_5\right\rangle;$
    \item $\alpha_4\neq2\alpha_5,\ \alpha_2=0,$ then choosing $x=1, \ z=-\frac{\alpha _3}{\alpha _4-2 \alpha _5},$ we have the representative $\left\langle\alpha\nabla_4+ \nabla_5\right\rangle_{\alpha\neq2};$
    \item $\alpha_4\neq2\alpha_5,\ \alpha_2\neq0,$ then choosing $x=\frac{\alpha_2}{\alpha_5}, \ z=-\frac{\alpha_2\alpha _3}{(\alpha_4-2\alpha_5) \alpha_5},$ we have the family of  representatives $\left\langle\nabla_2+\alpha\nabla_4+ \nabla_5\right\rangle_{\alpha\neq2}.$
\end{enumerate}

Summarizing all cases, we have the following distinct orbits
\begin{longtable} {llll}
$\langle\alpha\nabla_2+\nabla_3+2\nabla_4+\nabla_5\rangle,$ & $\langle\alpha\nabla_4+\nabla_5\rangle,$ & $\langle\nabla_2+\alpha\nabla_4+\nabla_5\rangle,$\\
\end{longtable}
which gives the following new algebras (see section \ref{secteoA}):

\begin{center}
${\rm N}_{135}^{\alpha},$
${\rm N}_{136}^{\alpha},$
${\rm N}_{137}^{\alpha}.$
\end{center}

\subsubsection{Central extensions of ${\mathcal N}^4_{10}$}
Let us use the following notations:
	\begin{longtable}{lllllll} 
$\nabla_1 = [\Delta_{11}],$& $\nabla_2 = [\Delta_{14}],$ & $\nabla_3=[\Delta_{21}],$ &
$\nabla_4=[\Delta_{22}],$ &
$\nabla_5=[\Delta_{23}-\Delta_{32}].$
\end{longtable}	
	
Take $\theta=\sum\limits_{i=1}^{5}\alpha_i\nabla_i\in {\rm H^2}({\mathcal N}^4_{10}).$
The automorphism group of ${\mathcal N}^4_{10}$ consists of invertible matrices of the form
$$\phi=
	\begin{pmatrix}
	x &  0  & 0 & 0\\
	0 &  y  & 0 & 0\\
	0 &  z  & xy & 0\\
    u &  v  & xz & x^2y
	\end{pmatrix}.
	$$
	Since
	$$
	\phi^T\begin{pmatrix}
	\alpha_1 & 0 & 0  & \alpha_2\\
	\alpha_3 & \alpha_4  & \alpha_5 & 0\\
	0 &	-\alpha_5 & 0 & 0\\
	0 & 0  & 0 & 0
	\end{pmatrix} \phi=\begin{pmatrix}
    \alpha_1^* & \alpha^* & \alpha^{**}  & \alpha_2^*\\
	\alpha_3^* & \alpha_4^*  & \alpha_5^* & 0\\
	0 &	-\alpha_5^* & 0 & 0\\
	0 & 0  & 0 & 0
	\end{pmatrix},
$$
we have that the action of ${\rm Aut} ({\mathcal N}^4_{10})$ on the subspace
$\langle \sum\limits_{i=1}^{5}\alpha_i\nabla_i  \rangle$
is given by
$\langle \sum\limits_{i=1}^{5}\alpha_i^{*}\nabla_i\rangle,$
where
\begin{longtable}{lcllcllcl}
$\alpha^*_1$&$=$&$x \left(x \alpha _1+u \alpha _2\right),$ &
$\alpha_2^*$&$=$&$x^3y \alpha _2,$ &
$\alpha^*_3$&$=$&$x y \alpha _3,$ \\
$\alpha^*_4$&$=$&$y^2 \alpha _4,$&
$\alpha_5^*$&$=$&$xy^2 \alpha_5.$\\
\end{longtable}
We are interested only in the cases with $\alpha_2\neq 0.$ Choosing $v=-\frac{x \alpha _1}{\alpha _2},$ we have $\alpha_1^*=0.$
\begin{enumerate}
    \item $\alpha_5=\alpha_3=\alpha_4=0,$ then we have the representative  $\left\langle\nabla_2\right\rangle;$
    \item $\alpha_5=\alpha_3=0,\ \alpha_4\neq0,$ then choosing $x=1,\ y=\frac{\alpha_2}{\alpha_4},$ we have the representative  $\left\langle\nabla_2+ \nabla_4\right\rangle;$
    \item $\alpha_5=0,\ \alpha_3\neq0,\ \alpha_4=0,$ then choosing $x=\sqrt{\frac{\alpha_3}{\alpha_2}},$ we have the representative $\left\langle\nabla_2+\nabla_3\right\rangle;$
    \item $\alpha_5=0,\ \alpha_3\neq0,\ \alpha_4\neq0,$ then choosing $x=\sqrt{\frac{\alpha_3}{\alpha_2}},\ y=\frac{\alpha_3}{\alpha_4}\sqrt{\frac{\alpha_3}{\alpha_2}}$ we have the representative $\left\langle\nabla_2+\nabla_3+ \nabla_4\right\rangle;$
    \item $\alpha_5\neq0,\ \alpha_3=0,\ \alpha_4=0,$ then choosing $x=1,\ y=\frac{\alpha_2}{\alpha_5}$ we have the representative $\left\langle\nabla_2+\nabla_5\right\rangle;$
    \item $\alpha_5\neq0,\ \alpha_3=0,\ \alpha_4\neq0,$ then choosing $x=\frac{\alpha _4}{\alpha _5},\ y=\frac{\alpha_2 \alpha_4^2}{\alpha_5^3},$ we have the representative $\left\langle\nabla_2+\nabla_4+\nabla_5\right\rangle;$
    \item $\alpha_5\neq0,\ \alpha_3\neq0,$ then choosing $x=\sqrt{\frac{\alpha_3}{\alpha_2}},\ y=\frac{\alpha_3 } {\alpha_5},$ we have the family of representatives $\left\langle\nabla_2+\nabla_3+\alpha\nabla_4+\nabla_5\right\rangle.$
\end{enumerate}

Summarizing all cases, we have the following distinct orbits
\begin{center} 
$\langle\nabla_2\rangle,$  
$\langle\nabla_2+\nabla_4\rangle,$ 
$\langle\nabla_2+\nabla_3\rangle,$  
$\langle\nabla_2+\nabla_3+\nabla_4\rangle,$
$\langle\nabla_2+\nabla_5\rangle,$  
$\langle\nabla_2+\nabla_4+\nabla_5\rangle,$  $\langle\nabla_2+\nabla_3+\alpha\nabla_4+\nabla_5\rangle,$
\end{center}
which gives the following new algebras (see section \ref{secteoA}):

\begin{center}
${\rm N}_{138},$
${\rm N}_{139},$
${\rm N}_{140},$
${\rm N}_{141},$
${\rm N}_{142},$
${\rm N}_{143},$
${\rm N}_{144}^{\alpha}.$
\end{center}

\subsubsection{Central extensions of ${\mathcal N}^4_{16}$}
Let us use the following notations:
	\begin{longtable}{lllllll} 
$\nabla_1 = [\Delta_{11}],$& $\nabla_2 = [\Delta_{21}],$ & $\nabla_3=[\Delta_{22}],$ &
$\nabla_4=[\Delta_{14}+\Delta_{23}],$ &
$\nabla_5=[-\Delta_{23}+\Delta_{32}].$
\end{longtable}	
	
Take $\theta=\sum\limits_{i=1}^{5}\alpha_i\nabla_i\in {\rm H^2}({\mathcal N}^4_{16}).$
The automorphism group of ${\mathcal N}^4_{16}$ consists of invertible matrices of the form
$$\phi=
	\begin{pmatrix}
	x &  0  & 0 & 0\\
	0 &  x^2  & 0 & 0\\
	0 &  y  & x^3 & 0\\
    u &  v  & xy & x^4
	\end{pmatrix}.
	$$
	Since
	$$
	\phi^T\begin{pmatrix}
	\alpha_1 & 0 & 0  & \alpha_4\\
	\alpha_2 & \alpha_3  & \alpha_4-\alpha_5 & 0\\
	0 &	\alpha_5 & 0 & 0\\
	0 & 0  & 0 & 0
	\end{pmatrix} \phi=\begin{pmatrix}
    \alpha_1^* & \alpha^* & \alpha^{**}  & \alpha_4^*\\
	\alpha_2^* & \alpha_3^*+\alpha^{**}  & \alpha_4^*-\alpha_5^* & 0\\
	0 &	\alpha_5^* & 0 & 0\\
	0 & 0  & 0 & 0
	\end{pmatrix},
$$
we have that the action of ${\rm Aut} ({\mathcal N}^4_{16})$ on the subspace
$\langle \sum\limits_{i=1}^{5}\alpha_i\nabla_i  \rangle$
is given by
$\langle \sum\limits_{i=1}^{5}\alpha_i^{*}\nabla_i\rangle,$
where
\begin{longtable}{lcllcllcl}
$\alpha^*_1$&$=$&$x(x\alpha_1+u\alpha_4),$ &
$\alpha_2^*$&$=$&$x^3 \alpha_2,$ &
$\alpha^*_3$&$=$&$x^4\alpha_3,$ \\
$\alpha^*_4$&$=$&$x^5 \alpha _4,$&
$\alpha_5^*$&$=$&$x^5 \alpha_5.$\\
\end{longtable}
We are interested only in the cases with $\alpha_4\neq 0.$ Choosing $u=-\frac{x\alpha_1}{\alpha_4},$ we have $\alpha_1^*=0.$
\begin{enumerate}
    \item $\alpha_3=\alpha_2=0,$ then we have the family of representatives  $\left\langle\nabla_4+\alpha\nabla_5\right\rangle;$
    \item $\alpha_3=0,\ \alpha_2\neq0,$ then choosing $x=\sqrt{\frac{\alpha_2}{\alpha_4}},$ we have the representative  $\left\langle\nabla_2+ \nabla_4+ \alpha\nabla_5\right\rangle;$
    \item $\alpha_3\neq0,$ then choosing $x=\frac{\alpha_3}{\alpha_4},$ we have the representative $\left\langle\alpha\nabla_2+\nabla_3+\nabla_4+ \beta\nabla_5\right\rangle.$
\end{enumerate}

Summarizing all cases, we have the following distinct orbits
\begin{longtable} {llll}
$\langle\nabla_4+\alpha\nabla_5\rangle,$ & $\langle\nabla_2+\nabla_4+\alpha\nabla_5\rangle,$ & $\langle\alpha\nabla_2+\nabla_3+\nabla_4+\beta\nabla_5\rangle,$\\
\end{longtable}
which gives the following new algebras (see section \ref{secteoA}):

\begin{center}
${\rm N}_{145}^{\alpha},$
${\rm N}_{146}^{\alpha},$
${\rm N}_{147}^{\alpha,\beta}.$
\end{center}

\subsection{$2$-dimensional central extensions of two-generated $3$-dimensional  nilpotent Novikov algebras}

\subsubsection{The description of second cohomology spaces.}
In the following table, we give the description of the second cohomology space of two-generated  $3$-dimensional nilpotent Novikov algebras

\begin{longtable}{ll llllllllllll}
${\mathcal N}^{3*}_{01}$ &:&  $e_1 e_1 = e_2$\\

\multicolumn{8}{l}{${\rm H^2}({\mathcal N}^{3*}_{01}) =
\langle[\Delta_{12}+\Delta_{21}], [\Delta_{13}+\Delta_{31}], [\Delta_{21}], [\Delta_{31}], [\Delta_{33}] \rangle$}\\
\hline

${\mathcal N}^{3*}_{02}$ &:&  $e_1 e_1 = e_3$ &  $e_2 e_2=e_3$ \\
 
\multicolumn{8}{l}{${\rm H^2}({\mathcal N}^{3*}_{02}) =\langle [\Delta_{12}], [\Delta_{21}], [\Delta_{22}]\rangle$}\\
\hline
 
${\mathcal N}^{3*}_{03}$ &:&   $e_1 e_2=e_3$ & $e_2 e_1=-e_3$   \\

\multicolumn{8}{l}{${\rm H^2}({\mathcal N}^{3*}_{03}) =   \langle [\Delta_{11}], [\Delta_{21}], [\Delta_{22}] \rangle$} \\
\hline

${\mathcal N}^{3*}_{04}(\lambda\neq 0)$ &:&
$e_1 e_1 = \lambda e_3$  & $e_2 e_1=e_3$  & $e_2 e_2=e_3$   \\

\multicolumn{8}{l}{${\rm H^2}({\mathcal N}^{3*}_{04}(\lambda\neq 0)) =   \langle [\Delta_{12}], [\Delta_{21}], [\Delta_{22}] \rangle$} \\
\hline

${\mathcal N}^{3*}_{04}(0)$ &:& $e_1 e_2=e_3$  \\

\multicolumn{8}{l}{${\rm H^2}({\mathcal N}^{3*}_{04}( 0)) =\langle
[\Delta_{11}], [\Delta_{13}], [\Delta_{21}],  [\Delta_{22}],   [\Delta_{23}]-[\Delta_{32}]
\rangle$}  \\
 \hline

%${\mathcal N}^3_{01}$ &:&  $e_1 e_1 = e_2$  & $e_2 e_1=e_3$  \\

%\multicolumn{8}{l}{${\rm H^2}({\mathcal N}^{3}_{01}) = \langle [\Delta_{12}], [\Delta_{13}]-[\Delta_{31}] \rangle$} \\\hline
 
%${\mathcal N}^3_{02}(\lambda)$ &:& $e_1 e_1 = e_2$ & $e_1 e_2=e_3$ & $e_2 e_1=\lambda e_3$\\

%\multicolumn{8}{l}{${\rm H^2}({\mathcal N}^{3}_{02})=
%\langle
%  [\Delta_{21}],   (2-\lambda)[\Delta_{13}]+\lambda([\Delta_{22}]+[\Delta_{31}])
% \rangle$}\\ 

\end{longtable}

\subsubsection{Central extensions of ${\mathcal N}^{3*}_{01}$}

Let us use the following notations:
\[
\nabla_1=[\Delta_{12}+\Delta_{21}], \quad \nabla_2=[\Delta_{13}+\Delta_{31}], \quad \nabla_3=[\Delta_{21}], \quad \nabla_4=[\Delta_{31}], \quad  \nabla_5=[\Delta_{33}]. \]

The automorphism group of ${\mathcal N}^{3*}_{01}$ consists of invertible matrices of the form

\[\phi=\begin{pmatrix}
x & 0 & 0\\
u & x^2 & w\\
z & 0 & y
\end{pmatrix}. \]

Since

\[ \phi^T\begin{pmatrix}
0 & \alpha_1 & \alpha_2\\
\alpha_1+\alpha_3 & 0 & 0\\
\alpha_2+\alpha_4 & 0 & \alpha_5
\end{pmatrix}\phi =
\begin{pmatrix}
\alpha^* & \alpha^*_1 & \alpha^*_2 \\
\alpha^*_1+ \alpha^*_3& 0 & 0 \\
\alpha^*_2+\alpha^*_4 & 0 & \alpha^*_5
\end{pmatrix},
\]
the action of $\operatorname{Aut} (\mathcal{N}_{01}^{3*})$ on subspace
$\Big\langle \sum\limits_{i=1}^5 \alpha_i\nabla_i \Big\rangle$ is given by
$\Big\langle \sum\limits_{i=1}^5 \alpha_i^*\nabla_i \Big\rangle,$
where
\begin{longtable}{lcl}
$\alpha^*_1$&$=$&$x^3 \alpha_1,$\\
$\alpha^*_2$&$=$&$w x \alpha _1+x y \alpha _2+y z \alpha _5,$\\
$\alpha^*_3$&$=$&$x^3 \alpha_3,$\\
$\alpha^*_4$&$=$&$ x(w\alpha_3+y \alpha_4),$\\
$\alpha^*_5$&$=$&$y^2 \alpha_5.$
\end{longtable}

We are interested  only in $2$-dimensional central extensions and consider the vector space generated by the following two cocycles:
\begin{center}
$\theta_1=\alpha_1\nabla_1+\alpha_2\nabla_2+\alpha_3\nabla_3+\alpha_4\nabla_4+\alpha_5\nabla_5 \ \ \text{and} \ \  \theta_2=\beta_1\nabla_1+\beta_2\nabla_2+\beta_4\nabla_4+\beta_5\nabla_5.$
\end{center}
Our aim is to find only central extensions with $(\alpha_3,\alpha_4, \beta_3,\beta_4)\neq 0.$
Hence, we have the following cases.

\begin{enumerate}
    \item $\alpha_3\neq0,$ then we have
    \begin{longtable}{lcllcl}
    $\alpha^*_1$&$=$&$x^3 \alpha_1,$ &   $\beta^*_1$&$=$&$x^3\beta_1,$ \\
    $\alpha^*_2$&$=$&$w x \alpha _1+x y \alpha _2+y z \alpha _5,$ &   $\beta^*_2$&$=$&$w x \beta_1+xy \beta_2+y z \beta_5,$\\ 
    $\alpha^*_3$&$=$&$x^3 \alpha_3,$   & $\beta^*_3$&$=$&$0,$\\ 
    $\alpha_4^*$&$=$&$x(w\alpha_3+y \alpha_4),$   & $\beta_4^*$&$=$&$x y\beta_4,$ \\
    $\alpha_5^*$&$=$&$y^2\alpha_5,$ &   $\beta_5^*$&$=$&$y^2\beta_5.$ \\
    \end{longtable}
    \begin{enumerate}
        \item $\beta_5\neq0,$ then we can suppose $\alpha_5=0$ and choosing 
        $w=-\frac{y\alpha_4}{\alpha_3},$ 
        $z=-\frac{x(\alpha_4\beta_1-\alpha_3\beta_2)}{\alpha_3\beta_5}$, 
        we have $\alpha_4^*=\beta_2^*=0.$ Thus, we can assume $\alpha_4=\beta_2=0$ and consider following subcases:
        \begin{enumerate}
            \item $\alpha_2=\beta_4=\beta_1=0,$ then we have the family of  representatives $ \left\langle \alpha\nabla_1+ \nabla_3,\nabla_5 \right\rangle;$ 
 
 \item $\alpha_2=\beta_4=0,\ \beta_1\neq0,$ then choosing $x=\sqrt[3]{{\beta_5}{\beta_1^{-1}}},\ y=1,$ we have the family of representatives $ \left\langle  \alpha \nabla_1+\nabla_3,\nabla_1+\nabla_5 \right\rangle
 ;$ 
            \item $\alpha_2=0,\ \beta_4\neq0,\ \beta_1=0$ then choosing $x={\beta_5}{\beta_4^{-1}}, \ y=1,$ we have the family of representatives $ \left\langle  \alpha \nabla_1+\nabla_3,\nabla_4+\nabla_5 \right\rangle;$ 
            
            \item $\alpha_2=0,\ \beta_4\neq0,\ \beta_1\neq0,$ then choosing $x={\beta_4^2}{\beta_1^{-1}\beta_5^{-1}},\ y={\beta_4^3}{\beta_1^{-1}\beta_5^{-2}},$ we have the family of representatives $ \left\langle  \alpha \nabla_1+\nabla_3,\nabla_1+\nabla_4+\nabla_5 \right\rangle;$ 
            
            \item $\alpha_2\neq0,\ \beta_4=\beta_1=0,$ then choosing $x=1,\ y={\alpha_3}{\alpha_2^{-1}},$ we have the family of  representatives $ \left\langle  \alpha \nabla_1+ \nabla_2+\nabla_3,\nabla_5\right\rangle;$ 
            
            \item $\alpha_2\neq0,\ \beta_4=0,\ \beta_1\neq0,$ then choosing 
            $x={\alpha_2^2\beta_1}{\alpha_3^{-2}\beta_5^{-1}},$ 
            $y={\alpha_2^3\beta_1^2}{\alpha_3^{-3}\beta_5^{-2}},$ we have the family of representatives $\left\langle  \alpha \nabla_1+ \nabla_2 +\nabla_3,\nabla_1+\nabla_5\right\rangle;$ 

\item $\alpha_2\neq0,\ \beta_4\neq0,$ then choosing $x={\alpha_2\beta_4}{\alpha_3^{-1}\beta_5^{-1}},$ $y={\alpha_2\beta_4^2}{\alpha_3^{-1}\beta_5^{-2}},$ 
we have the family of  representatives $ \left\langle  \alpha \nabla_1+ \nabla_2+ \nabla_3,\beta \nabla_1+ \nabla_4+ \nabla_5\right\rangle;$ 
           \end{enumerate}

\item $\beta_5=0, \beta_4\neq0.$ 
        \begin{enumerate}
            \item $\alpha_5=\beta_1=0, $ $ \alpha_1\beta_4\neq\alpha_3\beta_2,$ then choosing $y=1,\ w=\frac{\alpha_4\beta_2-\alpha_2\beta_4}{\alpha_1\beta_4-\alpha_3\beta_2},$ we have
            the family of representatives $ \left\langle \alpha\nabla_1+ \nabla_3,\beta\nabla_2+\nabla_4 \right\rangle_{\alpha\neq \beta};$ 

\item $\alpha_5=\beta_1=0,$ $\alpha_1\beta_4=\alpha_3\beta_2,$ $\alpha_2\alpha_3=\alpha_1\alpha_4,$ 
then choosing $y=1,\ w=-{\alpha_4}{\alpha_3^{-1}},$ we have the family of representatives $\left\langle \alpha\nabla_1+ \nabla_3,\alpha\nabla_2+\nabla_4 \right\rangle;$ 

            \item $\alpha_5=\beta_1=0,$ $ \alpha_1\beta_4=\alpha_3\beta_2,$ $ \alpha_2\alpha_3\neq\alpha_1\alpha_4$ then choosing 
            $x=\alpha_4 \beta_2-\alpha_2 \beta_4,$
            $y=-\alpha_3 \beta_4 (\alpha_4 \beta_2-\alpha_2 \beta_4),$ and 
            $w=\alpha_4 \beta_4 (\alpha_4 \beta_2-\alpha_2 \beta_4),$ we have the family of  representatives 
            $\left\langle \alpha\nabla_1+\nabla_2+\nabla_3,\alpha\nabla_2+\nabla_4 \right\rangle;$ 
            
            \item $\alpha_5=0,$ $\beta_1\neq0,$ then choosing 
            \begin{center}
                $x=1, \ y=\frac{\alpha_3}{\beta_4},$ 
            $w=\frac{(\alpha_3\beta_2+ \alpha_4 \beta_1-\alpha_1\beta_4)- \sqrt{(\alpha_3\beta_2+ \alpha_4 \beta_1-\alpha_1\beta_4)^2- 4\alpha_3 \beta_1(\alpha_4\beta_2-\alpha_2 \beta_4)}} {\alpha_3\beta_4},$
            \end{center} we have the family of representatives 
            $ \left\langle \alpha\nabla_1+ \nabla_3,\nabla_1+ \beta \nabla_2 +\nabla_4 \right\rangle;$
            
            \item $\alpha_5\neq0,$ $\beta_1=0,$ then choosing 
      \begin{center}      $x=\alpha_5,$
            $y=-\sqrt{\alpha_3}\alpha_5,$
            $z=0$
            and
            $w=-\frac{\sqrt{\alpha_3}\alpha_5(\alpha_4\beta_2-\alpha_2\beta_4)}{\alpha_3\beta_2-\alpha_1\beta_4},$\end{center}  
            we have the family of representatives 
            $\left\langle \alpha\nabla_1+\nabla_3+\nabla_5,\beta \nabla_2 +\nabla_4 \right\rangle;$ 
            
            \item $\alpha_5\neq0,$ $\beta_1\neq0,$ then choosing 
        \begin{center}
        $x=\frac{\alpha_3\beta_4^2}{\alpha_5\beta_1^2}, $ 
            $y=\frac{\alpha_3^2\beta_4^3}{\alpha_5^2\beta_1^3},$ 
            $w=-\frac{\alpha_3^2 \beta_2 \beta_4^3}{\alpha_5^2 \beta_1^4}$
            and
            $z=\frac{\alpha_3 (\alpha_1 \beta_2-\alpha_2 \beta_1) \beta_4^2}{\alpha_5^2 \beta_1^3},$
            \end{center}we have the family of representatives   $\left\langle \alpha\nabla_1+ \nabla_3+ \nabla_5, \nabla_1+\nabla_4\right \rangle;$ 
        \end{enumerate}
        
        \item $\beta_5=0, \beta_4=0,\ \beta_1\neq0,$ then we can suppose $\alpha_1=0$ and consider following subcases:
        \begin{enumerate}
            \item $\alpha_5=0,$ then choosing $w=-\frac{y\beta_2}{\beta_1}$, we have $\beta_2^*=0.$ 
            \begin{enumerate}
                \item if $\alpha_2=\alpha_4=0,$ then we have a   split algebra;
                \item if $\alpha_2=0,$ $ \alpha_4\neq0,$ then choosing $x=1,$ $ y=\frac{\alpha_3}{\alpha_4},$ we have the representative $\left\langle \nabla_3+\nabla_4, \nabla_1\right \rangle;$

\item if $\alpha_2\neq0,$ then choosing $x=1,$ $ y=\frac{\alpha_3}{\alpha_2},$ we have the family of representatives  $\left\langle \nabla_2+\nabla_3+\alpha \nabla_4, \nabla_1\right \rangle;$
            \end{enumerate}

\item $\alpha_5\neq0, $ $ \beta_2=0,$ then choosing 
$x=\alpha_5 ,$
$y=\sqrt{\alpha_3}  \alpha_5,$
$z=-\alpha_2$ and
$w= 0,$
 we have the representative $\left\langle \nabla_3+ \nabla_5, \nabla_1\right \rangle;$

\item $\alpha_5\neq0, $ $ \beta_2\neq0,$ then choosing 
\begin{center}
    $x=\frac{\alpha_3 \beta_2^2}{\alpha_5 \beta_1^2},$
$y=\frac{\alpha_3^2 \beta_2^3}{\alpha_5^2 \beta_1^3},$
$z=\frac{\alpha_3 \beta_2^2 (\alpha_2 \beta_1-\alpha_1 \beta_2)}{\alpha_5^2 \beta_1^3}$ and
$w=0,$
\end{center}
 we have the representative $\left\langle \nabla_3+\nabla_5, \nabla_1+\nabla_2\right \rangle.$

        \end{enumerate}
        \item $\beta_5=\beta_4=\beta_1=0, \beta_2\neq0,$ then we can suppose $\alpha_2=0$ and choosing $w=-\frac{y\alpha_4}{\alpha_3},$ we have $\alpha_4^*=0.$  Thus, we have following subcases:
        \begin{enumerate}

\item if $\alpha_5=0,$ then we have the family of representatives 
$\left\langle \alpha \nabla_1+ \nabla_3, \nabla_2\right \rangle;$
 
 \item if $\alpha_5\neq0,$ then choosing $x=1,\ y=\sqrt{{\alpha_3}{\alpha_5^{-1}}},$ we have the family of representatives $\left\langle \alpha \nabla_1+ \nabla_3+\nabla_5, \nabla_2\right\rangle.$
        \end{enumerate}
    \end{enumerate}

\item $\alpha_3=0,$ $\alpha_4\neq0,$ then we can suppose $\beta_4=0.$
    \begin{longtable}{lcllcl}
    $\alpha^*_1$&$=$&$x^3 \alpha_1,$   & $\beta^*_1$&$=$&$x^3\beta_1,$ \\
    $\alpha^*_2$&$=$&$w x \alpha _1+x y \alpha _2+y z \alpha _5,$ &   $\beta^*_2$&$=$&$w x \beta_1+xy \beta_2+y z \beta_5,$\\ 
    $\alpha^*_3$&$=$&$0,$ &  $\beta^*_3$&$=$&$0,$\\ 
    $\alpha_4^*$&$=$&$x y \alpha_4,$   & $\beta_4^*$&$=$&$0,$ \\
    $\alpha_5^*$&$=$&$y^2\alpha_5,$ & $\beta_5^*$&$=$&$y^2\beta_5.$ \\
    \end{longtable}
    \begin{enumerate}
        \item $\beta_5\neq0,$ then we can suppose $\alpha_5=0$ and choosing $z=-\frac{x(w\beta_1+y\beta_2)}{y\beta_5},$ we have $\beta_2^*=0.$ Thus, we have following subcases:
        \begin{enumerate}
            \item if $\beta_1=0,$ then $\alpha_1\neq0$ and  choosing $x=1,\ y=\frac{\alpha_1}{\alpha_4},\ w=-\frac{\alpha_2}{\alpha_4}$ we have the representative $\left\langle \nabla_1+ \nabla_4, \nabla_5\right\rangle;$
            
            \item if $\beta_1\neq0,$ $ \alpha_1=0$ then choosing $x=1,\ y=\sqrt{\frac{\beta_1} {\beta_5}},$ we have the family of representatives              $\left\langle \alpha \nabla_2+\nabla_4, \nabla_1+\nabla_5\right\rangle;$

\item if $\beta_1\neq0,$ $ \alpha_1\neq0$ then choosing $x=\frac{\alpha_4^2\beta_1} {\alpha_1^2\beta_5},$ 
$ y=\frac{\alpha_2^3 \beta_1^2}{\alpha_1^3\beta_5^2}, $ 
$ w=-\frac{\alpha_2^4\beta_1^2}{\alpha_1^4\beta_5^2},$ we have the representative  
            $\left\langle \nabla_1+ \nabla_4, \nabla_1+\nabla_5\right\rangle.$
        \end{enumerate}
        \item $\beta_5=0,$ $\beta_1\neq0,$ then we can suppose $\alpha_1=0$ and choosing $w=-\frac{y\beta_2}{\beta_1},$ we have $\beta_2^*=0.$ Thus, we have following subcases:
        \begin{enumerate}
            \item if $\alpha_5=0,$ then we have the representative $\left\langle \alpha \nabla_2+ \nabla_4, \nabla_1\right\rangle;$
            \item if $\alpha_5\neq0,$ then choosing $x=1,\ y=\frac{\alpha_4}{\alpha_5},\ z=-\frac{\alpha_2}{\alpha_5}$ we have the representative 
            $\left\langle \nabla_4+\nabla_5, \nabla_1 \right\rangle.$
        \end{enumerate}
    \item $\beta_5=\beta_1=0,$ $ \beta_2\neq0,$ then we can suppose $\alpha_2=0$. Since in case of $\alpha_1 =0,$ we have a split extension, we can assume $\alpha_1  \neq 0,$  
        Thus, we have following subcases:
        \begin{enumerate}
            \item if $\alpha_5=0,$ then choosing $x=1,\ y=\frac{\alpha_1}{\alpha_4},$ we have the representative 
            $\left\langle \nabla_1+\nabla_4, \nabla_2 \right\rangle;$
            \item if $\alpha_5\neq0,$ then choosing $x=\frac{\alpha_4^2}{\alpha_1\alpha_5},\ y=\frac{\alpha_4^3}{\alpha_1\alpha_5},$ we have the representative 
            $\left\langle \nabla_1+\nabla_4+\nabla_5, \nabla_2 \right\rangle.$
        \end{enumerate}
    \end{enumerate}
\end{enumerate}

Now we have the following distinct orbits:
\begin{center}$ \left\langle \alpha\nabla_1+ \nabla_3,\nabla_5 \right\rangle,$ \ 
$ \left\langle  \alpha \nabla_1+\nabla_3,\nabla_1+\nabla_5 \right\rangle,$  \ 
$ \left\langle  \alpha \nabla_1+\nabla_3,\nabla_4+\nabla_5 \right\rangle,$ \
$ \left\langle  \alpha \nabla_1+\nabla_3,\nabla_1+\nabla_4+\nabla_5 \right\rangle,$ \ 
$ \left\langle  \alpha \nabla_1+ \nabla_2+\nabla_3,\nabla_5\right\rangle,$  \ 
$\left\langle  \alpha \nabla_1+ \nabla_2 +\nabla_3,\nabla_1+\nabla_5\right\rangle,$ \
$ \left\langle  \alpha \nabla_1+ \nabla_2+ \nabla_3,\beta \nabla_1+ \nabla_4+ \nabla_5\right\rangle,$ \ 
$ \left\langle \alpha\nabla_1+ \nabla_3,\beta\nabla_2+\nabla_4 \right\rangle,$ \ 
$\left\langle \alpha\nabla_1+\nabla_2+\nabla_3,\alpha\nabla_2+\nabla_4 \right\rangle,$  \ 
$ \left\langle \alpha\nabla_1+ \nabla_3,\nabla_1+ \beta \nabla_2 +\nabla_4 \right\rangle,$ \ 
$\left\langle \alpha\nabla_1+\nabla_3+\nabla_5,\beta \nabla_2 +\nabla_4 \right\rangle,$  \ 
$\left\langle \alpha\nabla_1+ \nabla_3+ \nabla_5, \nabla_1+\nabla_4\right \rangle,$  \ 
$\left\langle \nabla_1, \nabla_3+\nabla_4 \right \rangle,$  \ 
$\left\langle \nabla_1, \nabla_2+\nabla_3+\alpha \nabla_4\right \rangle,$  \ 
$\left\langle \nabla_1, \nabla_3+\nabla_5, \right \rangle,$ \ 
$\left\langle \nabla_1+\nabla_2, \nabla_3+\nabla_5, \right \rangle,$ \ 
$\left\langle \alpha \nabla_1+ \nabla_3, \nabla_2\right \rangle,$  \ 
$\left\langle \alpha \nabla_1+ \nabla_3+\nabla_5, \nabla_2\right\rangle,$ \ 
$\left\langle \nabla_1+ \nabla_4, \nabla_5\right\rangle,$ \  
$\left\langle \nabla_1+\nabla_5, \alpha \nabla_2+\nabla_4\right\rangle,$ \ 
$\left\langle \nabla_1+ \nabla_4, \nabla_1+\nabla_5\right\rangle,$ \ 
$\left\langle \nabla_1, \alpha \nabla_2+ \nabla_4\right\rangle,$ \ 
$\left\langle \nabla_1, \nabla_4+\nabla_5\right\rangle,$ \ 
$\left\langle \nabla_1+\nabla_4, \nabla_2 \right\rangle,$ \ 
$\left\langle \nabla_1+\nabla_4+\nabla_5, \nabla_2 \right\rangle.$
\end{center}

Hence, we have the following new $5$-dimensional nilpotent Novikov algebras (see section \ref{secteoA}):

\begin{center}
${\rm N}_{148}^{\alpha},$ 
${\rm N}_{149}^{\alpha}$
${\rm N}_{150}^{\alpha},$
${\rm N}_{151}^{\alpha},$
${\rm N}_{152}^{\alpha},$
${\rm N}_{153}^{\alpha},$
${\rm N}_{154}^{\alpha, \beta},$
${\rm N}_{155}^{\alpha, \beta},$
${\rm N}_{156}^{\alpha},$
${\rm N}_{157}^{\alpha, \beta},$
${\rm N}_{158}^{\alpha, \beta},$
${\rm N}_{159}^{\alpha}$
${\rm N}_{160},$

${\rm N}_{161}^{\alpha}$
${\rm N}_{162},$
${\rm N}_{163},$
${\rm N}_{164}^{\alpha},$
${\rm N}_{165}^{\alpha},$
${\rm N}_{166},$
${\rm N}_{167}^{\alpha},$
${\rm N}_{168},$
${\rm N}_{169}^{\alpha},$
${\rm N}_{170},$
${\rm N}_{171},$
${\rm N}_{172}.$
\end{center}

\subsubsection{Central extensions of ${\mathcal N}^{3*}_{04}(0)$}

Let us use the following notations:
\[\nabla_1=[\Delta_{11}], \quad  \nabla_2=[\Delta_{13}], \quad \nabla_3=[\Delta_{21}], \quad \nabla_4=[\Delta_{22}], \quad
\nabla_5=[\Delta_{23}-\Delta_{32}].\]

The automorphism group of ${\mathcal N}^{3*}_{04}(0)$ consists of invertible matrices of the form

\[\phi=\left(
                             \begin{array}{ccc}
                               x & 0 & 0   \\
                               0 & y & 0  \\
                               z & t & xy                               \end{array}\right)
                               .\]

Since
\[
\phi^T
                           \left(\begin{array}{ccc}
                                \alpha_1& 0 & \alpha_2  \\
                                 \alpha_3 & \alpha_4 & \alpha_5  \\
                                 0 & -\alpha_5 & 0 \\
                             \end{array}
                           \right)\phi
                           =\left(\begin{array}{ccc}
                                \alpha^*_1 & \alpha^* &  \alpha_2^*   \\
                                 \alpha_3^*   &  \alpha_4^* & \alpha_5^*  \\
                                 0 &  -\alpha_5^* & 0 \\
                             \end{array}\right),\]
the action of $\operatorname{Aut} ({\mathcal N}^{3*}_{04}(0))$ on the subspace
$\langle  \sum\limits_{i=1}^5\alpha_i \nabla_i \rangle$
is given by
$\langle  \sum\limits_{i=1}^5\alpha^*_i \nabla_i \rangle,$ where

\[
\begin{array}{rcl}
\alpha^*_1&=&x(x\alpha_1+z\alpha_2);\\
\alpha^*_2&=&x^2y\alpha_2;\\
\alpha^*_3&=&y(x\alpha_3+z\alpha_5);\\
\alpha^*_4&=&y^2\alpha_4;\\
\alpha^*_5&=&xy^2\alpha_5.
\end{array}\]

We  consider the vector space generated by the following two cocycles:
\begin{center}  $\theta_1=\alpha_1\nabla_1+\alpha_2\nabla_2+\alpha_3\nabla_3+\alpha_4\nabla_4+\alpha_5\nabla_5$  \ \ and \ \  $\theta_2=\beta_1\nabla_1+\beta_3\nabla_3+\beta_4\nabla_4+\beta_5\nabla_5.$
\end{center}
We are interested  only in $(\alpha_2,\alpha_5,\beta_2,\beta_5)\neq 0.$
Hence, we have the following cases.

\begin{enumerate}
\item $\alpha_2\neq0,$ then we can suppose $\beta_2=0$ and have
\begin{longtable}{lcllcl}
$\alpha^*_1$&$=$&$x(x\alpha_1+z\alpha_2),$ &   $\beta^*_1$&$=$&$x^2\beta_1,$ \\
$\alpha^*_2$&$=$&$x^2y\alpha_2,$ &   $\beta^*_2$&$=$&$0,$\\ 
$\alpha^*_3$&$=$&$y(x\alpha_3+z\alpha_5),$ &   $\beta^*_3$&$=$&$y(x\beta_3+z\beta_5),$\\ 
$\alpha_4^*$&$=$&$y^2\alpha_4,$ &   $\beta_4^*$&$=$&$y^2\beta_4,$ \\
$\alpha_5^*$&$=$&$xy^2\alpha_5,$ &   $\beta_5^*$&$=$&$xy^2\beta_5.$ \\
\end{longtable}

\begin{enumerate}
    \item $\beta_5\neq0,$ then we can suppose $\alpha_5^*=0$ and choosing $z=-\frac{x\beta_3}{\beta_5}$, we have $\beta_3^*=0$ and we will suppose that $\beta_3=0.$ 
    Thus, we have following subcases:
    \begin{enumerate}
        \item $\alpha_1=\alpha_3=\alpha_4=\beta_4=\beta_1=0,$ then we have the   representative $\left\langle \nabla_2,\nabla_5 \right\rangle;$
        
        \item $\alpha_1=\alpha_3=\alpha_4=\beta_4=0, \ \beta_1\neq0,$ then choosing 
        $x=\frac{\beta_5}{\beta_1}, \ y=1,$ we have the representative $\left\langle \nabla_2,\nabla_1+\nabla_5 \right\rangle;$
        \item $\alpha_1=\alpha_3=\alpha_4=0, \beta_4\neq0,$ $\beta_1=0,$ then choosing $x=\frac{\beta_4}{\beta_5},$ we have the representative $\left\langle \nabla_2,\nabla_4+\nabla_5 \right\rangle;$

\item $\alpha_1=\alpha_3=\alpha_4=0,\ \beta_4\neq0,\ \beta_1\neq0,$ then choosing $x=\frac{\beta_4}{\beta_5},\ y=\frac{\sqrt{\beta_1 \beta_4}}{\beta_5},$ 
we have the representative $\left\langle \nabla_2,\nabla_1+\nabla_4+\nabla_5 \right\rangle;$
        \item $\alpha_1=\alpha_3=0,\ \alpha_4\neq0,\ \beta_4=\beta_1=0,$ then choosing $x=1,\ y=\frac{\alpha_2}{\alpha_4},$ we have the representative $\left\langle \nabla_2+\nabla_4,\nabla_5 \right\rangle;$

\item $\alpha_1=\alpha_3=0,\ \alpha_4\neq0,\ \beta_4=0,\ \beta_1\neq0,$ then choosing 
        $x=\sqrt[3]{\frac{\alpha_4^2\beta_1}{\alpha_2^2\beta_5}},$
        $y=\sqrt[3]{\frac{\alpha_4\beta_1^2}{\alpha_2\beta_5^2}},$
         we have the representative $\left\langle \nabla_2+\nabla_4,\nabla_1+\nabla_5 \right\rangle;$
        
        \item $\alpha_1=\alpha_3=0,\ \alpha_4\neq0,\ \beta_4\neq0,$ then choosing 
        $x=\frac{\beta_4}{\beta_5},\ y=\frac{\alpha_2\beta_4^2}{\alpha_4\beta_5^2},$
        we have the family of representatives $\left\langle \nabla_2+\nabla_4,\alpha\nabla_1+ \nabla_4+\nabla_5 \right\rangle;$
        
        \item $\alpha_1=0,\ \alpha_3\neq0,\ \alpha_4=\beta_1=0,$ then choosing $x=\frac{\alpha_3}{\alpha_2},$ we have the family of representatives $\left\langle \nabla_2+\nabla_3,\alpha \nabla_4+\nabla_5 \right\rangle;$
        
        \item $\alpha_1=0,\ \alpha_3\neq0,\ \alpha_4=0,\ \beta_1\neq0,$ then choosing $x=\frac{\alpha_3}{\alpha_2},$ $ y=\sqrt{\frac{\alpha_3\beta_1}{\alpha_2\beta_5}},$ we have the family of representatives $\left\langle \nabla_2+\nabla_3,\nabla_1+\alpha \nabla_4+\nabla_5 \right\rangle;$

\item $\alpha_1=0,\ \alpha_3\neq0,\ \alpha_4\neq0,$ then choosing $x=\frac{\alpha_3}{\alpha_2},$ $ y=\frac{\alpha_3^2}{\alpha_2\alpha_4},$ we have the family of representatives $\left\langle \nabla_2+\nabla_3+\nabla_4,\alpha \nabla_1+\beta \nabla_4+\nabla_5 \right\rangle;$
        \item $\alpha_1\neq0,\ \alpha_3=\alpha_4=\beta_1=\beta_4=0,$ then choosing $y=\frac{\alpha_1}{\alpha_2},$ we have the representative $\left\langle \nabla_1+\nabla_2,\nabla_5 \right\rangle;$

\item $\alpha_1\neq0,\ \alpha_3=\alpha_4=\beta_1=0,\ \beta_4\neq0,$ then choosing $x=\frac{\beta_4}{\beta_5},$ $y=\frac{\alpha_1}{\alpha_2}$ we have the family of representatives $\left\langle \nabla_1+\nabla_2,\nabla_4+\nabla_5\right\rangle;$

        \item $\alpha_1\neq0,\ \alpha_3=\alpha_4=0,\ \beta_1\neq0,$ then choosing $x=\frac{\alpha_1^2\beta_5}{\alpha_2^2\beta_1},\ y=\frac{\alpha_1}{\alpha_2},$ we have the family of representatives $\left\langle \nabla_1+\nabla_2,\nabla_1+\alpha \nabla_4+\nabla_5 \right\rangle;$
        
        \item $\alpha_1\neq0,\ \alpha_3=0,\ \alpha_4\neq0,$ then choosing $x=\frac{\sqrt{\alpha_1\alpha_4}}{\alpha_2},\ y=\frac{\alpha_1}{\alpha_2},$ we have the family of representatives $\left\langle \nabla_1+\nabla_2+\nabla_4,\alpha \nabla_1+\beta \nabla_4+\nabla_5 \right\rangle;$
        \item $\alpha_1\neq0,\ \alpha_3\neq0,$ then choosing $x=\frac{\alpha_3}{\alpha_2},$ $y=\frac{\alpha_1}{\alpha_2},$ we have the family of representatives 
$\left\langle\nabla_1+\nabla_2+\nabla_3+\alpha\nabla_4,\beta\nabla_1+\gamma\nabla_4+\nabla_5\right\rangle.$
    \end{enumerate}

    \item $\beta_5=0, \  \beta_4\neq0,$ then choosing $z=\frac{x(\alpha_4\beta_1-\alpha_1\beta_4)}{\alpha_2\beta_4},$ we can suppose $\alpha_1^*=\alpha_4^*=0$ and  have following subcases: 
    \begin{enumerate}
        \item $\alpha_3=\alpha_5=\beta_3=\beta_1=0,$ then we have the representative $\left\langle \nabla_2,\nabla_4 \right\rangle;$
        \item $\alpha_3=\alpha_5=\beta_3=0,$ $\beta_1\neq0,$ then choosing $x=1, \ y=\sqrt{\frac{\beta_1}{\beta_4}},$ we have the representative $\left\langle \nabla_2,\nabla_1+\nabla_4 \right\rangle;$

\item $\alpha_3=\alpha_5=0,\ \beta_3\neq0,$ then choosing $x=1, \ y=\frac{\beta_3}{\beta_4},$ we have the family of representatives $\left\langle \nabla_2,\alpha \nabla_1+ \nabla_3+\nabla_4 \right\rangle;$

        \item $\alpha_3=0,\ \alpha_5\neq0,$ then choosing $x=\frac{\alpha_5}{\alpha_2},$ we have the family of representatives 
        $\left\langle \nabla_2+\nabla_5,\alpha \nabla_1+\beta \nabla_3+\nabla_4 \right\rangle;$
        \item $\alpha_3\neq0,\ \alpha_5=\beta_3=\beta_1=0,$ then choosing $x=\frac{\alpha_3}{\alpha_2},$ we have the representative $\left\langle \nabla_2+\nabla_3,\nabla_4 \right\rangle;$
        \item $\alpha_3\neq0,\ \alpha_5=\beta_3=0,\ \beta_1\neq0,$ then choosing $x=\frac{\alpha_3}{\alpha_2}, \ y=\frac{\alpha_3\sqrt{\beta_1}}{\alpha_2\sqrt{\beta_4}},$ we have the representative $\left\langle \nabla_2+\nabla_3,\nabla_1+\nabla_4 \right\rangle;$
        \item $\alpha_3\neq0,\ \alpha_5=0,\ \beta_3\neq0,$ then choosing $x=\frac{\alpha_3}{\alpha_2}, \ y=\frac{\alpha_3\beta_3}{\alpha_2\beta_4},$
         we have the family of representatives $\left\langle \nabla_2+\nabla_3,\alpha \nabla_1+ \nabla_3+\nabla_4 \right\rangle;$
         
        \item $\alpha_3\neq0,\ \alpha_5\neq0,$ then choosing $x=\frac{\alpha_3}{\alpha_2}, \ y=\frac{\alpha_3}{\alpha_5},$ we have the family of representatives 
        $\left\langle \nabla_2+\nabla_3+\nabla_5,\alpha \nabla_1+\beta \nabla_3+\nabla_4\right\rangle.$
    \end{enumerate}
    
    \item $\beta_5=0, \ \beta_4=0,\ \beta_3\neq0.$ 
        \begin{enumerate}
        \item $\alpha_5=\beta_1=\alpha_4=0,$ then we can suppose $\alpha_3^*=0$ and choosing $x=1,\ z=-\frac{\alpha_1}{\alpha_2},$ we have the representative $\left\langle \nabla_2,\nabla_3 \right\rangle;$
        \item $\alpha_5=\beta_1=0,\ \alpha_4\neq0,$ then choosing $x=1,\ y=\frac{\alpha_2}{\alpha_4},\ z=-\frac{\alpha_1}{\alpha_2},$ we have the representative $\left\langle \nabla_2+\nabla_4, \nabla_3 \right\rangle;$
        \item $\alpha_5=0,\ \beta_1\neq0,\ \alpha_4=0$ then choosing $x=1,\ y=\frac{\beta_1}{\beta_3},\  z=\frac{\alpha_3\beta_1-\alpha_1\beta_3}{\alpha_2\beta_3},$ we have the representative $\left\langle \nabla_2,\nabla_1+\nabla_3 \right\rangle;$
        \item $\alpha_5=0,\ \beta_1\neq0,\ \alpha_4\neq0,$ then choosing $x=\frac{\alpha_4\beta_1}{\alpha_2\beta_3},\ y=\frac{\alpha_4\beta_1^2}{\alpha_2\beta_3^2},\ z=\frac{\alpha_4\beta_1(\alpha_3\beta_1-\alpha_1\beta_3)}{\alpha_2^2\beta_3^2},$ we have the representative $ \left\langle \nabla_2+\nabla_4, \nabla_1+\nabla_3 \right\rangle;$

\item $\alpha_5\neq0,\ \alpha_2\beta_3\neq\alpha_5\beta_1, \ \alpha_4=0,$ then choosing $x=1,$ 
$y=\frac{\alpha_2}{\alpha_5},$ 
$z=\frac{\alpha_3\beta_1-\alpha_1\beta_3}{\alpha_2\beta_3-\alpha_5\beta_1},$ we have the family of  representatives$ \left\langle \nabla_2+\nabla_5,\alpha \nabla_1+\nabla_3 \right\rangle_{\alpha\neq1};$

        \item $\alpha_5\neq0, \ \alpha_2\beta_3\neq\alpha_5\beta_1,\ \alpha_4\neq0,$ then choosing 
        $x=\frac{\alpha_4}{\alpha_5},$ 
        $y=\frac{\alpha_2\alpha_4}{\alpha_5^2},$ 
        $z=\frac{\alpha_4(\alpha_3\beta_1-\alpha_1\beta_3)}{\alpha_5(\alpha_2\beta_3-\alpha_5\beta_1)},$ we have the family of  representatives $\left\langle \nabla_2+\nabla_4+\nabla_5, \alpha\nabla_1+\nabla_3 \right\rangle_{\alpha\neq1
        };$
        
        \item $\alpha_5\neq0,\ \alpha_2\beta_3=\alpha_5\beta_1,\ \alpha_2\alpha_3=\alpha_1\alpha_5,\ \alpha_4=0,$ then choosing $x=1,$ 
        $y=\frac{\alpha_2}{\alpha_5},$ 
        $z=-\frac{\alpha_1}{\alpha_2},$ we have the representative $ \left\langle \nabla_2+\nabla_5, \nabla_1+\nabla_3 \right\rangle;$
        
        \item $\alpha_5\neq0,\ \alpha_2\beta_3=\alpha_5\beta_1,\ \alpha_2\alpha_3=\alpha_1\alpha_5,\ \alpha_4\neq0,$ then choosing 
        $x=\frac{\alpha_4}{\alpha_5},$ 
        $y=\frac{\alpha_2\alpha_4}{\alpha_5^2},$ 
        $ z=-\frac{\alpha_1\alpha_4}{\alpha_2\alpha_5},$ we have the representative $ \left\langle \nabla_2+\nabla_4+\nabla_5, \nabla_1+\nabla_3 \right\rangle;$
 
 \item $\alpha_5\neq0,\ \alpha_2\beta_3=\alpha_5\beta_1,\ \alpha_2\alpha_3\neq\alpha_1\alpha_5,$ then choosing $x=\frac{\alpha_2\alpha_3-\alpha_1\alpha_5}{\alpha_2^2},$ 
 $ y=\frac{\alpha_2\alpha_3-\alpha_1\alpha_5}{\alpha_2\alpha_5},$ 
 $ z=-\frac{\alpha_1(\alpha_2\alpha_3-\alpha_1\alpha_5)}{\alpha_2^3},$ we have the family of  representatives $ \left\langle \nabla_2+\nabla_3+\alpha \nabla_4+\nabla_5, \nabla_1+\nabla_3 \right\rangle.$

    \end{enumerate}

 \item $\beta_5=0, \beta_4=0,\ \beta_3=0,\ \beta_1\neq0,$ then we can suppose $\alpha_1^*=0$ and consider following subcases: 
        \begin{enumerate}
        \item $\alpha_5=\alpha_4=\alpha_3=0,$ then we have the representative $ \left\langle \nabla_2,\nabla_1 \right\rangle;$
        \item $\alpha_5=\alpha_4=0, \ \alpha_3\neq0,$ then choosing $x=\frac{\alpha_3}{\alpha_2},\ y=1,$ we have the representative $ \left\langle \nabla_2+\nabla_3,\nabla_1 \right\rangle;$
        \item $\alpha_5=0,\ \alpha_4\neq0, \ \alpha_3=0,$ then choosing $x=1,\ y=\frac{\alpha_2}{\alpha_4},$ we have the representative $ \left\langle \nabla_2+\nabla_4,\nabla_1 \right\rangle;$
        \item $\alpha_5=0,\ \alpha_4\neq0, \ \alpha_3\neq0,$ then choosing $x=\frac{\alpha_3}{\alpha_2},\ y=\frac{\alpha_3^2}{\alpha_2\alpha_4},$ we have the representative $ \left\langle \nabla_2+\nabla_3+\nabla_4,\nabla_1 \right\rangle;$
        \item $\alpha_5\neq0,\ \alpha_4=0,$ then choosing $x=1,\ y=\frac{\alpha_2}{\alpha_5},\ z=-\frac{\alpha_3}{\alpha_5},$ we have the representative $ \left\langle \nabla_2+\nabla_5,\nabla_1 \right\rangle;$
        \item $\alpha_5\neq0,\ \alpha_4\neq0,$ then choosing $x=\frac{\alpha_4}{\alpha_5},\ y=\frac{\alpha_2\alpha_4}{\alpha_5^2},\ z=-\frac{\alpha_3\alpha_4}{\alpha_5^2},$ we have the representative $ \left\langle \nabla_2+\nabla_4+\nabla_5,\nabla_1 \right\rangle.$
    \end{enumerate}

\end{enumerate}

\item $\alpha_2=0,$ then $\alpha_5\neq0$ and we have
\begin{longtable}{lcllcl}
$\alpha^*_1$&$=$&$x^2\alpha_1,$ & $\beta^*_1$&$=$&$x^2\beta_1,$ \\
$\alpha^*_2$&$=$&$0,$ & $\beta^*_2$&$=$&$0,$\\ 
$\alpha^*_3$&$=$&$y(x\alpha_3+z\alpha_5),$ & $\beta^*_3$&$=$&$xy\beta_3,$\\ 
$\alpha_4^*$&$=$&$y^2\alpha_4,$ & $\beta_4^*$&$=$&$y^2\beta_4,$ \\
$\alpha_5^*$&$=$&$xy^2\alpha_5,$ & $\beta_5^*$&$=$&$0.$ \\
\end{longtable}

\begin{enumerate}
    \item $\beta_4\neq0,$ then we can suppose $\alpha_4^*=0$ and consider following subcases: 
    \begin{enumerate}
        \item $\alpha_1=\beta_3=\beta_1=0,$ then choosing $x=1,\  z=-\frac{\alpha_3}{\alpha_5},$ we have the representative $ \left\langle \nabla_5,\nabla_4 \right\rangle;$
        \item $\alpha_1=\beta_3=0,\ \beta_1\neq0,$ then choosing $x=1,\  y=\sqrt{\frac{\beta_1}{\beta_4}},\ z=-\frac{\alpha_3}{\alpha_5},$ we have the representative $ \left\langle \nabla_5,\nabla_1+\nabla_4 \right\rangle;$

\item $\alpha_1=0,\ \beta_3\neq0,$ then choosing $x=1,\  y=\frac{\beta_3}{\beta_4},\ z=-\frac{\alpha_3}{\alpha_5},$ we have the family of representatives $ \left\langle \nabla_5,\alpha\nabla_1+\nabla_3+\nabla_4 \right\rangle;$

        \item $\alpha_1\neq0,\ \beta_1=\beta_3=0,$ then choosing $x=\alpha_1\alpha_5,\  y=\alpha_1,\ z=-\alpha_1\alpha_3,$ we have the family of representatives $ \left\langle \nabla_1+\nabla_5,\nabla_4 \right\rangle;$
                \item $\alpha_1\neq0 ,\ \beta_3=0,\ \beta_1\neq0,$ then choosing $x=\frac{\alpha_1\beta_4}{\alpha_5\beta_1},\ y=\frac{\alpha_1\sqrt{\beta_4}}{\alpha_5\sqrt{\beta_1}},\ z=-\frac{\alpha_1\alpha_3\beta_4}{\alpha_5^2\beta_1},$ we have the representative $ \left\langle \nabla_1+\nabla_5,\nabla_1+\nabla_4 \right\rangle;$
        \item $\alpha_1\neq0,\ \beta_3\neq0,$ then choosing $x=\frac{\alpha_1\beta_4^2}{\alpha_5\beta_3^2},\  y=\frac{\alpha_1\beta_4}{\alpha_5\beta_3},\ z=-\frac{\alpha_1\alpha_3\beta_4^2}{\alpha_5^2\beta_3},$ we have the family of representatives 
        $ \left\langle \nabla_1+\nabla_5,\alpha \nabla_1+ \nabla_3+\nabla_4 \right\rangle.$

    \end{enumerate}
    \item $\beta_4=0,\ \beta_3\neq0.$ 
    \begin{enumerate}
        \item $\alpha_4=\beta_1=\alpha_1=0,$ then choosing $x=1,\  z=-\frac{\alpha_3}{\alpha_5},$ we have the representative $ \left\langle \nabla_5,\nabla_3 \right\rangle;$
        \item $\alpha_4=\beta_1=0,\ \alpha_1\neq0,$ then choosing $x=\frac{\alpha_5}{\alpha_1},\ y=1,\ z=-\frac{\alpha_3}{\alpha_1},$ we have the representative $ \left\langle \nabla_5+\nabla_1,\nabla_3 \right\rangle;$
        \item $\alpha_4=0,\ \beta_1\neq0,$ then choosing $x=1,\ y=\frac{\beta_1}{\beta_3},\ z=\frac{\alpha_1\beta_3-\alpha_3\beta_1}{\alpha_5\beta_1},$ we have the representative $ \left\langle \nabla_5,\nabla_1+\nabla_3 \right\rangle;$
        \item $\alpha_4\neq0,\ \beta_1=\alpha_1=0,$ then choosing $x=\frac{\alpha_4}{\alpha_5},\ z=-\frac{\alpha_3\alpha_4}{\alpha_5^2},$ we have the representative $ \left\langle \nabla_4+\nabla_5,\nabla_3 \right\rangle;$
        \item $\alpha_4\neq0,\ \beta_1=0,\ \alpha_1\neq0,$ then choosing $x=\frac{\alpha_4}{\alpha_5},\ y=\frac{\sqrt{\alpha_1\alpha_4}}{\alpha_5},\  z=-\frac{\alpha_3\alpha_4}{\alpha_5^2},$ we have the representative $ \left\langle \nabla_1+\nabla_4+\nabla_5,\nabla_3 \right\rangle;$
        \item $\alpha_4\neq0,\ \beta_1\neq0,$ then choosing $x=\frac{\alpha_4}{\alpha_5},\ y=\frac{\alpha_4\beta_1}{\alpha_5\beta_3},\  z=\frac{\alpha_4(\alpha_1\beta_3-\alpha_3\beta_1)}{\alpha_5^2\beta_1},$ we have the representative $ \left\langle \nabla_4+\nabla_5,\nabla_1+\nabla_3 \right\rangle.$
    \end{enumerate}
     \item $\beta_4=\beta_3=0,\ \beta_1\neq 0,$ then we can suppose $\alpha_1^*=0$ and choosing $z=-\frac{x\alpha_3}{\alpha_5},$ obtain $\alpha_3^*=0.$ 
    \begin{enumerate}
        \item $\alpha_4=0,$ then we have the representative $ \left\langle \nabla_5,\nabla_1 \right\rangle;$
        \item $\alpha_4\neq0,$ then choosing $x=\frac{\alpha_4}{\alpha_5},$ we have the representative $ \left\langle \nabla_4+\nabla_5,\nabla_1 \right\rangle.$
    \end{enumerate}
\end{enumerate}
\end{enumerate}

Now we have the following distinct orbits:
\begin{center}
$\left\langle \nabla_2,\nabla_5 \right\rangle,$ 
$\left\langle \nabla_2,\nabla_1+\nabla_5 \right\rangle,$ 
$\left\langle \nabla_2,\nabla_4+\nabla_5 \right\rangle,$ 
$\left\langle \nabla_2,\nabla_1+\nabla_4+\nabla_5 \right\rangle,$ 
$\left\langle \nabla_2+\nabla_4,\nabla_5 \right\rangle,$ 
$\left\langle \nabla_2+\nabla_4,\nabla_1+\nabla_5 \right\rangle,$ 
$\left\langle \nabla_2+\nabla_4,\alpha\nabla_1+ \nabla_4+\nabla_5 \right\rangle,$ 
$\left\langle \nabla_2+\nabla_3,\alpha \nabla_4+\nabla_5 \right\rangle,$ 
$\left\langle \nabla_2+\nabla_3,\nabla_1+\alpha \nabla_4+\nabla_5 \right\rangle,$ 
$\left\langle \nabla_2+\nabla_3+\nabla_4,\alpha \nabla_1+\beta \nabla_4+\nabla_5 \right\rangle,$ 
$\left\langle \nabla_1+\nabla_2,\nabla_5 \right\rangle,$ 
$\left\langle \nabla_1+\nabla_2,\nabla_4+\nabla_5\right\rangle,$ 
$\left\langle \nabla_1+\nabla_2,\nabla_1+\alpha \nabla_4+\nabla_5 \right\rangle,$ 
$\left\langle  \nabla_1+\nabla_2+\nabla_4,  \alpha \nabla_1+\beta \nabla_4+\nabla_5 \right\rangle,$ 
$\left\langle  \nabla_1+\nabla_2+\nabla_3+\alpha\nabla_4, \beta\nabla_1+\gamma\nabla_4+\nabla_5  \right\rangle,$
$\left\langle \nabla_2,\nabla_4 \right\rangle,$ 
$\left\langle \nabla_2,\nabla_1+\nabla_4 \right\rangle,$ 
$\left\langle \nabla_2,\alpha \nabla_1+ \nabla_3+\nabla_4 \right\rangle,$ 
$\left\langle \nabla_2+\nabla_5,\alpha \nabla_1+\beta \nabla_3+\nabla_4 \right\rangle,$
$\left\langle \nabla_2+\nabla_3,\nabla_4 \right\rangle,$ 
$\left\langle \nabla_2+\nabla_3,\nabla_1+\nabla_4 \right\rangle,$ 
$\left\langle \nabla_2+\nabla_3,\alpha \nabla_1+ \nabla_3+\nabla_4 \right\rangle,$ 
$\left\langle \nabla_2+\nabla_3+\nabla_5,  \alpha \nabla_1+\beta \nabla_3+\nabla_4  \right\rangle,$ 
$\left\langle \nabla_2,\nabla_3 \right\rangle,$ 
$\left\langle \nabla_2+\nabla_4, \nabla_3 \right\rangle,$ 
$\left\langle \nabla_2,\nabla_1+\nabla_3 \right\rangle,$ 
$ \left\langle \nabla_2+\nabla_4, \nabla_1+\nabla_3 \right\rangle$ 
$ \left\langle \nabla_2+\nabla_5,\alpha \nabla_1+\nabla_3 \right\rangle,$ 
$\left\langle  \nabla_2+\nabla_4+\nabla_5,   \alpha\nabla_1+\nabla_3  \right\rangle,$ 
 $ \left\langle  \nabla_2+\nabla_3+\alpha \nabla_4+\nabla_5,   \nabla_1+\nabla_3  \right\rangle,$ 
$ \left\langle \nabla_2,\nabla_1 \right\rangle,$ 
$ \left\langle \nabla_2+\nabla_3,\nabla_1 \right\rangle,$ 
$ \left\langle \nabla_2+\nabla_4,\nabla_1 \right\rangle,$ 
$ \left\langle \nabla_2+\nabla_3+\nabla_4,\nabla_1 \right\rangle,$ 
$ \left\langle \nabla_2+\nabla_5,\nabla_1 \right\rangle,$ 
$ \left\langle \nabla_2+\nabla_4+\nabla_5,\nabla_1 \right\rangle,$ 
$ \left\langle \nabla_5,\nabla_4 \right\rangle,$ 
$ \left\langle \nabla_5,\nabla_1+\nabla_4 \right\rangle,$ 
$ \left\langle \nabla_5,\alpha\nabla_1+\nabla_3+\nabla_4 \right\rangle,$ 
$ \left\langle \nabla_1+\nabla_5,\nabla_4 \right\rangle,$ 
$ \left\langle \nabla_1+\nabla_5,\nabla_1+\nabla_4 \right\rangle,$ 
$ \left\langle \nabla_1+\nabla_5,\alpha \nabla_1+ \nabla_3+\nabla_4 \right\rangle,$ 
$ \left\langle \nabla_5,\nabla_3 \right\rangle,$ 
$ \left\langle \nabla_5+\nabla_1,\nabla_3 \right\rangle,$ 
$ \left\langle \nabla_5,\nabla_1+\nabla_3 \right\rangle,$ 
$ \left\langle \nabla_4+\nabla_5,\nabla_3 \right\rangle,$ 
$ \left\langle \nabla_1+\nabla_4+\nabla_5,\nabla_3 \right\rangle,$ 
$ \left\langle \nabla_4+\nabla_5,\nabla_1+\nabla_3 \right\rangle,$ 
$ \left\langle \nabla_5,\nabla_1 \right\rangle,$ 
$ \left\langle \nabla_4+\nabla_5,\nabla_1 \right\rangle.$ 
\end{center}

Hence, we have the following new $5$-dimensional nilpotent Novikov algebras (see section \ref{secteoA}):

\begin{center}
 ${\rm N}_{173},$ 
${\rm N}_{174},$ 
${\rm N}_{175},$ 
${\rm N}_{176},$ 
${\rm N}_{177},$
${\rm N}_{178},$ 
${\rm N}_{179}^{\alpha},$
${\rm N}_{180}^{\alpha},$
${\rm N}_{181}^{\alpha},$
${\rm N}_{182}^{\alpha, \beta},$ 
${\rm N}_{183},$
${\rm N}_{184},$
${\rm N}_{185}^{\alpha},$
${\rm N}_{186}^{\alpha, \beta},$
${\rm N}_{187}^{\alpha, \beta, \gamma},$
${\rm N}_{188},$ 
${\rm N}_{189},$ 
${\rm N}_{62}^{0,0},$  ${\rm N}_{64}^{\alpha \notin \{0, 1\}, 0},$
${\rm N}_{190},$
${\rm N}_{191}^{\alpha, \beta},$ 
${\rm N}_{192},$
${\rm N}_{193},$
${\rm N}_{194}^{\alpha},$
${\rm N}_{195}^{\alpha, \beta},$
${\rm N}_{89}^{0}$
${\rm N}_{90}^{0}$
${\rm N}_{196},$ 
${\rm N}_{197}^{\alpha},$
${\rm N}_{198}^{\alpha},$ 
${\rm N}_{199}^{\alpha},$
${\rm N}_{200},$ 
${\rm N}_{201},$ 
${\rm N}_{202},$ 
${\rm N}_{203},$ 
${\rm N}_{204},$ 
${\rm N}_{205},$ 
${\rm N}_{206},$
${\rm N}_{207},$ 
${\rm N}_{208}^{\alpha},$ 
${\rm N}_{114}^{0},$ 
${\rm N}_{209},$ 
${\rm N}_{210}^{\alpha},$ 
${\rm N}_{211},$ 
${\rm N}_{212},$
${\rm N}_{213},$ 
${\rm N}_{214},$ 
${\rm N}_{215},$ 
${\rm N}_{216},$
${\rm N}_{217},$
${\rm N}_{218}.$

\end{center}

It should be noted that the family of orbits $\left\langle \nabla_2,\alpha \nabla_1+\nabla_3 +\nabla_4\right\rangle,$ gives the algebra ${\rm N}_{62}^{0,0}$ in case of $\alpha =0,$ the parametric family  ${\rm N}_{64}^{\frac {\alpha-1} {\alpha},0}$ in case of $\alpha \notin \{ 0, 1\}$ and the algebra ${\rm N}_{190}$ in case of $\alpha =1.$ 

%\subsubsection{Central extensions of ${\mathcal N}^{3}_{01}$ and ${\mathcal N}^{3}_{02}(\lambda)$.}  

%Since the dimension of second cohomology space of the algebras
%${\mathcal N}^{3}_{01}$ and ${\mathcal N}^{3}_{02}(\lambda)$ are two dimensional, we have that there is only one two dimensional extensions of these algebras.

%\begin{longtable}{lllllllll}
%\hline
%${\mathbb N}_{88}$ & $:$& $e_1 e_1 = e_2$  &  $e_1e_2=e_4$ & $e_1e_3= e_5$ & $e_2 e_1 = e_3$ & $e_3e_1= -e_5$\\
%\hline
%${\mathbb N}_{89}^{\lambda}$ & $:$& $e_1 e_1 = e_2$ & $e_1 e_2 = e_3$ &   $e_1e_3=(2-\lambda)e_5$ & $e_2e_1=\lambda e_3 + e_4$ &  $e_2e_2=\lambda e_5$ & $e_3e_1= \lambda e_5$\\
%\hline
%\end{longtable}

\section{Classification theorem for $5$-dimensional nilpotent Novikov algebras}\label{secteoA}
The algebraic classification of complex $5$-dimensional nilpotent Novikov algebras consists of two parts:
\begin{enumerate}
    \item $5$-dimensional algebras with identity $xyz=0$ (also known as $2$-step nilpotent algebras) are the intersection of all varieties of algebras defined by a family of polynomial identities   of degree three or more; for example, it is in the intersection of associative, Zinbiel, Leibniz, etc, algebras. All these algebras can be obtained as central extensions of zero-product algebras. The geometric classification of $2$-step nilpotent algebras is given in \cite{ikp20}. It is the reason why we are not interested in it.
    
     \item $5$-dimensional nilpotent (non-$2$-step nilpotent) Novikov algebras, which are central extensions of  nilpotent Novikov algebras with nonzero products of a smaller dimension. These algebras are classified by several steps:

    \begin{enumerate}
        \item complex split   $5$-dimensional nilpotent Novikov  algebras are classified in  \cite{kkk19};

        \item complex non-split   $5$-dimensional nilpotent commutative associative algebras  are listed in \cite{krs20};
                
        \item  complex  one-generated $5$-dimensional nilpotent Novikov algebras  are classified in \cite{ckkk20};
               
        \item complex non-split   non-one-generated   $5$-dimensional nilpotent non-commutative Novikov algebras  are classified in Theorem A (see below).
            \end{enumerate}
  \end{enumerate}

\begin{theoremA}%\label{teorA}
Let ${\mathbb N}$ be a complex non-split  non-one-generated $5$-dimensional  nilpotent (non-$2$-step nilpotent) non-commutative  Novikov algebra.
Then ${\mathbb N}$ is isomorphic to one algebra from the following list:

{\tiny 
 
\begin{longtable}{ll llll}
${\rm N}_{01}$ & $: $ & $e_1e_1=e_2$  & $e_1e_4=e_5$ & $e_2e_1=e_5$ \\&& $e_3e_3=e_5$ & $e_4e_1=e_5$ \\
${\rm N}_{02}$ & $: $ & $e_1e_1=e_2$ & $e_2e_1=e_5$ & $e_3e_4=e_5$ & $e_4e_3=-e_5$ \\
${\rm N}_{03}$ & $: $ & $e_1e_1=e_2$ & $e_1e_3=e_5$ & $e_2e_1=e_5$ \\
&& $e_3e_1=e_5$  & $e_4e_3=e_5$ & $e_4e_4=e_5$ \\
${\rm N}_{04}^{\alpha}$ & $: $ & $e_1e_1=e_2$ & $e_1e_2=e_5$ & $e_2e_1=\alpha e_5$ \\&& $e_3e_3=e_5$ & $e_4e_1=e_5$ \\
${\rm N}_{05}$ & $: $ & $e_1e_1=e_2$ & $e_1e_2=e_5$ & $e_2e_1= e_5$ \\
&& $e_3e_1=e_5$ & $e_3e_3=e_5$ & $e_4e_4=e_5$ \\
${\rm N}_{06}$ & $: $ & $e_1e_1=e_2$ & $e_1e_2=e_5$ & $e_2e_1= e_5$ \\&& $e_3e_1=e_5$ & $e_3e_4=e_5$ & $e_4e_3=e_5$\\
${\rm N}_{07}^{\alpha}$ & $: $ & $e_1e_1=  e_2$ & $e_1e_2= \alpha e_5$ & $e_2e_1=(\alpha+1)e_5$ \\&& $e_3e_3= e_5$ & $e_4e_4=e_5$\\
${\rm N}_{08}^{\alpha}$ & $: $ & $e_1e_1=  e_2$ & $e_1e_2=  e_5$ & $e_2e_1=\alpha e_5$ \\&& $e_3e_4= e_5$ & $e_4e_3=-e_5$\\ 
${\rm N}_{09}$ & $: $ & $e_1e_1=  e_2$ & $e_1e_2=  e_5$ & $e_2e_1= - e_5$ \\&& $e_3e_1= e_5$  & $e_3e_4= e_5$ & $e_4e_3=-e_5$\\
${\rm N}_{10}^{\alpha, \beta}$ & $: $ & $e_1e_1=  e_2$ & $e_1e_2=  \alpha e_5$ & $e_2e_1= (\alpha+1) e_5$ \\&& $e_3e_3=\beta e_5$  & $e_4e_3= e_5$ & $e_4e_4=e_5$\\

${\rm N}_{11}^{\alpha}$ & $: $ & $e_1e_1=  e_2$ & $e_1e_2=  \frac{-1 + \sqrt{1-4\alpha}}{2} e_5$ &  $e_3e_3= \alpha e_5$  & $e_4e_3= e_5$ \\
& & $e_4e_4=e_5$ & $e_2e_1= \frac{1 + \sqrt{1-4\alpha}}{2} e_5$ & $e_3e_1=e_5$ &\\

${\rm N}_{12}^{\alpha\neq \frac{1}{4}}$ & $: $ & $e_1e_1=  e_2$ & $e_1e_2=  \frac{-1 - \sqrt{1-4\alpha}}{2} e_5$ &  $e_3e_3= \alpha e_5$  & $e_4e_3= e_5$ \\
& & $e_4e_4=e_5$ & $e_2e_1= \frac{1 - \sqrt{1-4\alpha}}{2} e_5$ & $e_3e_1=e_5$ \\

${\rm N}_{13}$ & $: $ & $e_1e_1=e_3$  & $e_1e_3=e_5$ &  $e_2e_1=e_5$ &$e_2e_2=e_4$  \\ 
& & $e_2e_4=e_5$& $e_2e_4=e_5$ & $e_3e_1=e_5$ \\
${\rm N}_{14}^{\alpha}$ & $: $ & $e_1e_1=e_3$  & $e_1e_2=e_5$ &  $e_1e_3=\alpha e_5$ &$e_2e_1=e_5$  \\ & & $e_2e_2=e_4$  & $e_2e_4=-(1+\alpha)e_5$ & $e_3e_1=(1+\alpha)e_5$ & $e_4e_2=-\alpha e_5$\\
${\rm N}_{15}^{\alpha}$ & $: $ & $e_1e_1=e_3$  & $e_1e_3=\alpha e_5$ & $e_2e_2=e_4$  \\ & & $e_2e_4=e_5$   & $e_3e_1=(1+\alpha)e_5$ & $e_4e_2=e_5$\\
${\rm N}_{16}^{\alpha, \beta}$ & $:$ & $e_1e_1=e_3$  & $e_1e_3=\alpha e_5$ & $e_2e_2=e_4$  \\ & & $e_2e_4=\beta e_5$   & $e_3e_1=(1+\alpha)e_5$ & $e_4e_2=(1+\beta)e_5$\\

${\rm N}_{17}$ & $: $ & $e_1e_1=e_3$  & $e_1e_2=e_3$ &  $e_1e_3=e_5$ &$e_4e_4=e_5$\\
${\rm N}_{18}$ & $: $ & $e_1e_1=e_3$  & $e_1e_2=e_3$ &  $e_1e_3=e_5$ \\&&$e_2e_4=e_5$ &$e_4e_4=e_5$\\
${\rm N}_{19}$ & $: $ & $e_1e_1=e_3$  & $e_1e_2=e_3$ &  $e_1e_3=e_5$ \\&&$e_2e_1=e_5$ &$e_4e_4=e_5$\\
${\rm N}_{20}$ & $: $ & $e_1e_1=e_3$  & $e_1e_2=e_3$ &  $e_1e_3=e_5$ \\&&$e_2e_1=e_5$ & $e_2e_4=e_5$ &$e_4e_4=e_5$\\
${\rm N}_{21}^{\alpha}$ & $: $ & $e_1e_1=e_3$  & $e_1e_2=e_3$ &  $e_1e_3=e_5$ &$e_2e_1=e_5$ \\& &$e_2e_2=e_5$  &$e_2e_4=\alpha e_5$ &$e_4e_4=e_5$\\
${\rm N}_{22}^{\alpha}$ & $: $ & $e_1e_1=e_3$  & $e_1e_2=e_3$ &  $e_1e_3=e_5$ \\&& $e_2e_2=e_5$  &$e_2e_4=\alpha e_5$ &$e_4e_4=e_5$\\
${\rm N}_{23}$ & $: $ & $e_1e_1=e_3$  & $e_1e_2=e_3$ &  $e_1e_3=e_5$ \\&& $e_2e_2=e_5$ & $e_4e_1=e_5$ &$e_4e_2=e_5$\\
${\rm N}_{24}$ & $: $ & $e_1e_1=e_3$  & $e_1e_2=e_3$ &  $e_1e_3=e_5$ & $e_2e_2=e_5$ \\
&& $e_2e_4=-e_5$ & $e_4e_1=e_5$ &$e_4e_2=e_5$\\
${\rm N}_{25}^{\alpha}$ & $: $ & $e_1e_1=e_3$  & $e_1e_2=e_3$ &  $e_1e_3=e_5$\\&& $e_2e_4=\alpha e_5$ & $e_4e_1=e_5$ &$e_4e_2=e_5$\\
${\rm N}_{26}$ & $: $ & $e_1e_1=e_3$  & $e_1e_2=e_3$ &  $e_1e_3=e_5$ \\&& $e_2e_1=e_5$ & $e_4e_2=e_5$\\
${\rm N}_{27}$ & $: $ & $e_1e_1=e_3$  & $e_1e_2=e_3$ &  $e_1e_3=e_5$ \\
&& $e_2e_2=e_5$ & $e_2e_4=-e_5$ & $e_4e_2=e_5$\\
${\rm N}_{28}^{\alpha}$ & $: $ & $e_1e_1=e_3$  & $e_1e_2=e_3$ &  $e_1e_3=e_5$ \\
&& $e_2e_4=\alpha e_5$ & $e_4e_2=e_5$\\
${\rm N}_{29}$ & $: $ & $e_1e_1=e_3$  & $e_1e_2=e_3$ &  $e_1e_3=e_5$ &  $e_4e_1=e_5$\\
${\rm N}_{30}$ & $: $ & $e_1e_1=e_3$  & $e_1e_2=e_3$ &  $e_1e_3=e_5$ \\&& $e_2e_2=e_5$ & $e_4e_1=e_5$\\
${\rm N}_{31}$ & $: $ & $e_1e_1=e_3$  & $e_1e_2=e_3$ &  $e_1e_3=e_5$ \\&& $e_2e_4=e_5$ & $e_4e_1=e_5$\\
${\rm N}_{32}$ & $: $ & $e_1e_1=e_3$  & $e_1e_2=e_3$ &  $e_1e_3=e_5$ & $e_2e_4=e_5$ \\
${\rm N}_{33}$ & $: $ & $e_1e_1=e_3$  & $e_1e_2=e_3$ &  $e_1e_3=-e_5$ & $e_2e_3=-e_5$ \\& &$e_3e_1=e_5$ &$e_3e_2=e_5$ &$e_4e_4=e_5$ \\
${\rm N}_{34}$ & $: $ & $e_1e_1=e_3$  & $e_1e_2=e_3$ &  $e_1e_3=-e_5$ & $e_1e_4=e_5$  \\ && $e_2e_3=-e_5$ &$e_3e_1=e_5$  &$e_3e_2=e_5$ &$e_4e_4=e_5$ \\
${\rm N}_{35}^{\alpha}$ & $: $ & $e_1e_1=e_3$  & $e_1e_2=e_3+e_5$ &  $e_1e_3=-e_5$ & $e_1e_4=\alpha e_5$\\ & & $e_2e_3=-e_5$ &$e_3e_1=e_5$   &$e_3e_2=e_5$ &$e_4e_4=e_5$ \\
${\rm N}_{36}$ & $: $ & $e_1e_1=e_3$  & $e_1e_2=e_3+e_5$ &  $e_1e_3=-e_5$ & $e_2e_2=e_5$ \\ && $e_2e_3=-e_5$ &$e_3e_1=e_5$   &$e_3e_2=e_5$ &$e_4e_4=e_5$ \\
${\rm N}_{37}^{\alpha}$ & $: $ & $e_1e_1=e_3$  & $e_1e_2=e_3$ &  $e_1e_3=-e_5$ & $e_1e_4=\alpha e_5$ \\ && $e_2e_3=-e_5$ &$e_3e_1=e_5$  &$e_3e_2=e_5$ &$e_4e_4=e_5$ \\
${\rm N}_{38}$ & $: $ & $e_1e_1=e_3$  & $e_1e_2=e_3+e_5$ &  $e_1e_3=-e_5$ \\ && 
$e_1e_4=e_5$  & $e_2e_2=e_5$ & $e_2e_3=-e_5$  \\ & &$e_3e_1=e_5$  &$e_3e_2=e_5$ &$e_4e_4=e_5$ \\
${\rm N}_{39}$ & $: $ & $e_1e_1=e_3$  & $e_1e_2=e_3$ &  $e_2e_3=-e_5$ \\ &&$e_3e_1=e_5$ &$e_3e_2=e_5$ &$e_4e_4=e_5$ \\
${\rm N}_{40}$ & $: $ & $e_1e_1=e_3$  & $e_1e_2=e_3+e_5$ &  $e_2e_3=-e_5$ \\ &&$e_3e_1=e_5$ &$e_3e_2=e_5$ &$e_4e_4=e_5$ \\
${\rm N}_{41}^{\alpha}$ & $: $ & $e_1e_1=e_3$  & $e_1e_2=e_3+\alpha e_5$ &  $e_1e_4=e_5$ &  $e_2e_3=-e_5$  \\ &&$e_3e_1=e_5$ &$e_3e_2=e_5$ &$e_4e_4=e_5$ \\
${\rm N}_{42}^{\alpha,\beta}$ & $: $ & $e_1e_1=e_3$  & $e_1e_2=e_3+\alpha e_5$ &  $e_1e_4=\beta e_5$ &  $e_2e_2=e_5$ \\ && $e_2e_3=-e_5$  & $e_3e_1=e_5$   &$e_3e_2=e_5$ &$e_4e_4=e_5$ \\
${\rm N}_{43}$ & $: $ & $e_1e_1=e_3$  & $e_1e_2=e_3$ &  $e_1e_3=-e_5$ & $e_1e_4=e_5$ \\ && $e_2e_3=-e_5$ &$e_2e_4=e_5$   &$e_3e_1=e_5$ &$e_3e_2=e_5$  \\
${\rm N}_{44}$ & $: $ & $e_1e_1=e_3$  & $e_1e_2=e_3+e_5$ &  $e_1e_3=-e_5$ & $e_1e_4=e_5$\\ & & $e_2e_3=-e_5$ &$e_2e_4=e_5$   &$e_3e_1=e_5$ &$e_3e_2=e_5$  \\
${\rm N}_{45}$ & $: $ & $e_1e_1=e_3$  & $e_1e_2=e_3$ &  $e_1e_3=-e_5$ \\ 
&& $e_1e_4=e_5$ & $e_2e_3=-e_5$ &$e_2e_4=e_5$  \\ 
& &$e_3e_1=e_5$ &$e_3e_2=e_5$ & $e_4e_1=e_5$ \\
${\rm N}_{46}$ & $: $ & $e_1e_1=e_3$  & $e_1e_2=e_3+e_5$ &  $e_1e_3=-e_5$ & $e_2e_3=-e_5$ \\ &&$e_2e_4=e_5$ &$e_3e_1=e_5$   &$e_3e_2=e_5$ & $e_4e_1=e_5$ \\
${\rm N}_{47}^{\alpha}$ & $: $ & $e_1e_1=e_3$  & $e_1e_2=e_3$ &  $e_1e_3=-e_5$  &
$e_2e_3=-e_5$\\ & &$e_2e_4=e_5$ &$e_3e_1=e_5$  &$e_3e_2=e_5$ & $e_4e_1=\alpha e_5$ \\
${\rm N}_{48}^{\alpha}$ & $: $ & $e_1e_1=e_3$  & $e_1e_2=e_3+e_5$ &  $e_1e_4=\alpha e_5$ & $e_2e_3=-e_5$ \\ &&$e_2e_4=e_5$ &$e_3e_1=e_5$   &$e_3e_2=e_5$ & $e_4e_1=\alpha e_5$ \\
${\rm N}_{49}^{\alpha,\beta}$ & $: $ & $e_1e_1=e_3$  & $e_1e_2=e_3$ &  $e_1e_4=\alpha e_5$ & $e_2e_3=-e_5$ \\ &&$e_2e_4=e_5$ &$e_3e_1=e_5$   &$e_3e_2=e_5$ & $e_4e_1=\beta e_5$ \\
${\rm N}_{50}$ & $: $ & $e_1e_1=e_3$  & $e_1e_2=e_3+e_5$ &  $e_1e_3=-e_5$ &  $e_1e_4=-e_5$ \\ && $e_2e_3=-e_5$ &$e_3e_1=e_5$  &$e_3e_2=e_5$ & $e_4e_1=e_5$ \\
${\rm N}_{51}$ & $: $ & $e_1e_1=e_3$  & $e_1e_2=e_3$ &  $e_1e_3=-e_5$ &  $e_1e_4=\alpha e_5$ \\ && $e_2e_3=-e_5$ &$e_3e_1=e_5$   &$e_3e_2=e_5$ & $e_4e_1=e_5$ \\
${\rm N}_{52}$ & $: $ & $e_1e_1=e_3$  & $e_1e_2=e_3+e_5$ &  $e_1e_3=-e_5$ \\ &&  
$e_1e_4=-e_5$ &  $e_2e_2=e_5$ & $e_2e_3=-e_5$ \\ & & $e_3e_1=e_5$ &$e_3e_2=e_5$ & $e_4e_1=e_5$ \\
${\rm N}_{53}^{\alpha}$ & $: $ & $e_1e_1=e_3$  & $e_1e_2=e_3$ &  $e_1e_3=-e_5$\\ & &  $e_1e_4=\alpha e_5$ &  $e_2e_2=e_5$ & $e_2e_3=-e_5$ \\ & & $e_3e_1=e_5$ &$e_3e_2=e_5$ & $e_4e_1=e_5$ \\
${\rm N}_{54}^{\alpha}$ & $: $ & $e_1e_1=e_3$  & $e_1e_2=e_3$ &  $e_1e_4=\alpha e_5$ & $e_2e_3=-e_5$\\ & & $e_3e_1=e_5$ &$e_3e_2=e_5$ &  $e_4e_1=e_5$ \\
${\rm N}_{55}^{\alpha}$ & $: $ & $e_1e_1=e_3$  & $e_1e_2=e_3$ &  $e_1e_4=\alpha e_5$ & $e_2e_2=e_5$  \\ && $e_2e_3=-e_5$ & $e_3e_1=e_5$   &$e_3e_2=e_5$ & $e_4e_1=e_5$ \\
${\rm N}_{56}$ & $: $ & $e_1e_1=e_3$  & $e_1e_2=e_3$ &  $e_1e_3=-e_5$  &  $e_1e_4=e_5$ \\&&   $e_2e_3=-e_5$  & $e_3e_1=e_5$ &$e_3e_2=e_5$ \\
${\rm N}_{57}$ & $: $ & $e_1e_1=e_3$  & $e_1e_2=e_3$ &  $e_1e_3=-e_5$ &  $e_1e_4=e_5$\\&& $e_2e_2=e_5$ &  $e_2e_3=-e_5$   & $e_3e_1=e_5$ &$e_3e_2=e_5$ \\
${\rm N}_{58}$ & $: $ & $e_1e_1=e_3$  & $e_1e_2=e_3$ &  $e_1e_4=e_5$ \\&& 
$e_2e_3=-e_5$ & $e_3e_1=e_5$ & $e_3e_2=e_5$ \\
${\rm N}_{59}$ & $: $ & $e_1e_1=e_3$  & $e_1e_2=e_3$ &  $e_1e_4=e_5$ & $e_2e_2=e_5$ \\&&  $e_2e_3=-e_5$ & $e_3e_1=e_5$ & $e_3e_2=e_5$ \\
 
${\rm N}_{60}^{(\alpha,\beta)\neq(0,0)}$ & $: $ & 
$e_1e_1=e_5$  & $e_1e_2=e_3$ &  $e_1e_4=-\beta e_5$ \\
& & $e_2e_1=e_4$  &$e_2e_2=-e_3$  &  $e_2e_3=-e_5$ \\
& & $e_2e_4=\alpha e_5$  & $e_3e_2=e_5$ & $e_4e_1=\beta e_5$ \\

${\rm N}_{61}^{(\alpha,\beta, \gamma) \neq(0,0,\gamma)}$ & $: $ & 
$e_1e_2=e_3$ &  $e_1e_3=-(\gamma +1) e_5$ & 
$e_1e_4=-\beta e_5$\\
&  & $e_2e_1=e_4$  &$e_2e_2=-e_3$ &    $e_2e_3=\gamma e_5$ \\
& & $e_2e_4=\alpha e_5$ & $e_3e_2=e_5$ & $e_4e_1=\beta e_5$   \\

${\rm N}_{62}^{(\alpha,\beta)\neq(0,0)}$ & $: $ & $e_1e_2=e_3$ &  $e_1e_3=-e_5$ & $e_1e_4=-\beta e_5$ & $e_2e_1=e_4$\\
&   &$e_2e_2=-e_3$ & $e_2e_3=e_5$   & $e_2e_4=\alpha e_5$ & $e_4e_1=\beta e_5$   \\

${\rm N}_{63}^{\alpha\neq0}$ & $: $ & $e_1e_2=e_3$ &  $e_1e_3=-e_5$ & $e_1e_4=-\alpha e_5$ & $e_2e_1=e_4$ \\
&  &$e_2e_2=-e_3+e_5$ & $e_2e_3=e_5$    & $e_2e_4=\alpha e_5$ & $e_4e_1=\alpha e_5$ \\
 
${\rm N}_{64}^{\alpha\neq1, \beta\neq 0}$ & $: $ & 
$e_1e_1=e_3$  & $e_1e_2=e_4$ & $e_1e_3=e_5$  \\&& 
$e_1e_4=\beta e_5$  &  $e_2e_1=-\alpha e_3$  & $e_2e_2=-e_4$  \\&& $e_2e_3=-\alpha e_5$  &$e_2e_4=-\beta e_5$ \\  

${\rm N}_{65}^{\alpha\neq1,\beta,\gamma,\delta}$ & $: $ & 
$e_1e_1=e_3$  & $e_1e_2=e_4$ & $e_1e_3=(\beta-1)e_5$ \\
  && $e_1e_4=\gamma e_5$  &  $e_2e_1=-\alpha e_3$  & $e_2e_2=-e_4$  & $e_2e_3=-\alpha \beta e_5$ \\
&  &$e_2e_4=-(\gamma+\delta) e_5$    &$e_3e_1=e_5$  &$e_4e_2=\delta e_5$ &   \\

${\rm N}_{66}^{\alpha\neq 1,\beta\neq \frac{1}{1-\alpha},\gamma}$ & $: $ &  $e_1e_1 = e_3$ & $e_1e_2 = e_4+e_5$ & $e_1e_3=(\beta-1) e_5$ \\
  && $e_1e_4=\frac{\beta\delta}{(\alpha-1)\beta+1}e_5$  & $e_2e_1 = -\alpha e_3$ & $e_2e_2 = -e_4$ & $e_2e_3 = - \alpha \beta e_5$ \\
& 
 & $e_2e_4= -\frac{\gamma(\alpha\beta+1)}{(\alpha-1)\beta+1} e_5$  
 & $e_3e_1= e_5$ &  $e_4e_2= \gamma e_5$\\ 

${\rm N}_{67}^{\alpha\neq1,\beta}$ & $: $ & 
 $e_1e_1 = e_3$ & $e_1e_2 = e_4+e_5$ & $e_1e_3 = \frac{\alpha}{1-\alpha}e_5$ & $e_1e_4 = \beta e_5$ \\
 &&
 $e_2e_1 = -\alpha e_3$ & 
 $e_2e_2 = -e_4$ & $e_2e_3 = \frac{\alpha}{\alpha-1}e_5$ \\&& $e_2e_4 = -\beta e_5$ & $e_3e_1 = e_5$  \\ 

${\rm N}_{68}^{\alpha, \beta}$ & $: $ & $e_1e_1=e_3$  & $e_1e_2=e_4$ & $e_1e_3=\alpha e_5$ & $e_1e_4=\beta e_5$ \\ & &  $e_2e_2=e_4$ & $e_2e_4=-(1+\beta) e_5$  & $e_4e_2=e_5$\\

${\rm N}_{69}^{\alpha}$ & $: $ & $e_1e_1=e_3$  & $e_1e_2=e_4$ & $e_1e_3=\alpha e_5$ & $e_1e_4=-e_5$  \\ & &  $e_2e_1=e_5$ &  $e_2e_2=e_4$ & $e_4e_2=e_5$\\

${\rm N}_{70}$ & $: $ & $e_1e_1=e_3$  & $e_1e_2=e_4$ & $e_1e_3=e_5$ & $e_1e_4=-e_5$ \\ & &  $e_2e_1=-e_3$ & $e_2e_2=-e_4$ & $e_2e_3=-e_5$  &$e_2e_4=e_5$\\

${\rm N}_{71}$ & $: $ & $e_1e_1=e_3$  & $e_1e_2=e_4+e_5$ & $e_1e_3=e_5$ & $e_1e_4=-e_5$ \\ &  &  $e_2e_1=-e_3-e_5$ & $e_2e_2=-e_4$ & $e_2e_3=-e_5$  &$e_2e_4=e_5$\\

${\rm N}_{72}$ & $: $ & $e_1e_1=e_3$  & $e_1e_2=e_4$ & $e_1e_3=e_5$ & $e_1e_4=-e_5$ \\ & &  $e_2e_1=-e_3+e_5$ & $e_2e_2=-e_4$ & $e_2e_3=-e_5$  &$e_2e_4=e_5$\\

${\rm N}_{73}^{\alpha}$ & $: $ & $e_1e_1=e_3$  & $e_1e_2=e_4$ & $e_1e_3=(\alpha-1)e_5$ & $e_1e_4=(1-\alpha)e_5$\\ &  &  $e_2e_1=-e_3$ & $e_2e_2=-e_4$ & $e_2e_3=(1-\alpha)e_5$  & $e_2e_4=(\alpha-1)e_5$  \\ & &$e_3e_1=e_5$ &$e_3e_2=-e_5$  &$e_4e_1=-e_5$ &$e_4e_2=e_5$ \\

${\rm N}_{74}^{\alpha}$ & $: $ & $e_1e_1=e_3$  & $e_1e_2=e_4$ & $e_1e_3=(\alpha-1)e_5$ & $e_1e_4=(1-\alpha)e_5$ \\ & &  $e_2e_1=-e_3+e_5$ & $e_2e_2=-e_4$ & $e_2e_3=(1-\alpha)e_5$  & $e_2e_4=(\alpha-1)e_5$  \\ & &$e_3e_1=e_5$ &$e_3e_2=-e_5$  &$e_4e_1=-e_5$ &$e_4e_2=e_5$ \\

${\rm N}_{75}$ & $: $ & $e_1e_1=e_3$  & $e_1e_2=e_4$ & $e_1e_4=e_5$ &  $e_2e_1=-e_3$ \\ & & $e_2e_2=-e_4$ & $e_2e_4=-e_5$ &$e_3e_1=e_5$ &$e_3e_2=-e_5$ \\ & &$e_4e_1=-e_5$ &$e_4e_2=e_5$ \\

${\rm N}_{76}$ & $: $ & $e_1e_1=e_3$  & $e_1e_2=e_4$ & $e_1e_4=e_5$ &  $e_2e_1=-e_3+e_5$ \\ & & $e_2e_2=-e_4$ & $e_2e_4=-e_5$ &$e_3e_1=e_5$ &$e_3e_2=-e_5$ \\ & &$e_4e_1=-e_5$ &$e_4e_2=e_5$ \\

${\rm N}_{77}^{\alpha}$ & $: $ & $e_1e_1=e_3$  & $e_1e_2=e_4$ & $e_1e_3=(\alpha-1)e_5$ \\
& & $e_1e_4=e_5$   &  $e_2e_1=-e_3$ & $e_2e_2=-e_4$\\
&  & $e_2e_3=-\alpha e_5$  &$e_2e_4=-e_5$   &$e_3e_1=e_5$\\

${\rm N}_{78}^{\alpha}$ & $: $ & $e_1e_1=e_3$  & $e_1e_2=e_4$ & $e_1e_3=(\alpha-1)e_5$ \\
& & $e_1e_4=e_5$  &  $e_2e_1=-e_3+e_5$ & $e_2e_2=-e_4$ \\
& & $e_2e_3=-\alpha e_5$  &$e_2e_4=-e_5$   &$e_3e_1=e_5$\\

${\rm N}_{79}^{\alpha}$ & $: $ & $e_1e_1=e_3$  & $e_1e_2=e_4$ & $e_1e_3=(\alpha+2)e_5$ & $e_1e_4=-e_5$ \\  &  &  $e_2e_1=-e_3$ & $e_2e_2=-e_4$ & $e_2e_3=-(1+\alpha)e_5$  &$e_3e_1=-2e_5$ \\ & & $e_3e_2=e_5$ &$e_4e_1=e_5$\\

${\rm N}_{80}^{\alpha}$ & $: $ & $e_1e_1=e_3$  & $e_1e_2=e_4$ & $e_1e_3=(\alpha+2)e_5$ & $e_1e_4=-e_5$ \\ & &  $e_2e_1=-e_3+e_5$ & $e_2e_2=-e_4$ & $e_2e_3=-(1+\alpha)e_5$  &$e_3e_1=-2e_5$ \\ & &$e_3e_2=e_5$ &$e_4e_1=e_5$\\

${\rm N}_{81}^{\alpha}$ & $: $ & $e_1e_1=e_3$  & $e_1e_2=e_4$ & $e_1e_3=(\alpha+2)e_5$ &  $e_2e_1=-e_3$ \\ && $e_2e_2=-e_4$ & $e_2e_3=-(1+\alpha)e_5$ & $e_2e_4=-e_5$  &$e_3e_1=-2e_5$\\ & &$e_3e_2=e_5$ &$e_4e_1=e_5$\\

${\rm N}_{82}^{\alpha}$ & $: $ & $e_1e_1=e_3$  & $e_1e_2=e_4$ & $e_1e_3=(\alpha+2)e_5$ &  $e_2e_1=-e_3+e_5$\\ & & $e_2e_2=-e_4$ & $e_2e_3=-(1+\alpha)e_5$ & $e_2e_4=-e_5$  &$e_3e_1=-2e_5$\\ & &$e_3e_2=e_5$ &$e_4e_1=e_5$\\

${\rm N}_{83}^{\alpha, \beta}$ & $: $ & $e_1e_1=e_3$  & $e_1e_2=e_4$ & $e_1e_3=\alpha e_5$ & $e_1e_4=(\beta-1)e_5$ \\ & &  $e_2e_1=-e_3$ & $e_2e_2=-e_4$ & $e_2e_3=-(1+\alpha)e_5$ & $e_2e_4=-\beta e_5$\\  & &$e_3e_2=e_5$ &$e_4e_1=e_5$\\

${\rm N}_{84}^{\alpha, \beta}$ & $: $ & $e_1e_1=e_3$  & $e_1e_2=e_4$ & $e_1e_3=\alpha e_5$ & $e_1e_4=(\beta-1)e_5$ \\ & &  $e_2e_1=-e_3+e_5$ & $e_2e_2=-e_4$ & $e_2e_3=-(1+\alpha)e_5$ & $e_2e_4=-\beta e_5$\\ & &$e_3e_2=e_5$ &$e_4e_1=e_5$\\
 
${\rm N}_{85}$ & $: $ & $e_1e_2=e_3$  & $e_1e_4=e_5$ & $e_2e_1=e_4$ \\
& & $e_2e_3=e_5$  &  $e_3e_2=-e_5$ & $e_4e_1=-e_5$ \\

${\rm N}_{86}$ & $: $ & $e_1e_1=e_5$  & $e_1e_2=e_3$  & $e_1e_4=e_5$ & $e_2e_1=e_4$ \\
& & $e_2e_3=e_5$  &  $e_3e_2=-e_5$   & $e_4e_1=-e_5$ \\

${\rm N}_{87}^{\alpha}$ & $: $ & $e_1e_1=\alpha e_5$  & $e_1e_2=e_3$  & $e_1e_4=e_5$ & $e_2e_1=e_4$ \\
& & $e_2e_2=e_5$ & $e_2e_3=e_5$   &  $e_3e_2=-e_5$ & $e_4e_1=-e_5$ \\

${\rm N}_{88}^{\alpha,\beta}$ & $: $ & $e_1e_2=e_3$  &  $e_1e_3=e_5$  & $e_1e_4=\alpha e_5$ & $e_2e_1=e_4$ \\
& & $e_2e_3=\beta e_5$  & $e_2e_4=e_5$ & $e_3e_2=-\beta e_5$ & $e_4e_1=-\alpha e_5$ \\

${\rm N}_{89}^{\alpha}$ & $: $ & $e_1e_2=e_3$  & $e_2e_1=e_4$ & $e_2e_3=\alpha e_5$ \\
&  & $e_2e_4=e_5$ & $e_3e_2=-\alpha e_5$ \\

${\rm N}_{90}^{\alpha}$ & $: $ & $e_1e_1=e_5$ & $e_1e_2=e_3$  & $e_2e_1=e_4$ \\
& & $e_2e_3=\alpha e_5$  & $e_2e_4=e_5$   &   $e_3e_2=-\alpha e_5$ \\

${\rm N}_{91}^{\alpha\neq0}$ & $: $ & $e_1e_2=e_3$  & $e_1e_4=e_5$ & $e_2e_1=e_4$ & $e_2e_3=\alpha e_5$  \\
& & $e_2e_4=e_5$    & $e_3e_2=-\alpha e_5$ & $e_4e_1=-e_5$ \\

${\rm N}_{92}^{\alpha\neq0}$ & $: $ & $e_1e_1=e_5$ & $e_1e_2=e_3$  & $e_1e_4=e_5$ & $e_2e_1=e_4$ \\
& & $e_2e_3=\alpha e_5$  & $e_2e_4=e_5$   & $e_3e_2=-\alpha e_5$ & $e_4e_1=-e_5$ \\

${\rm N}_{93}$ & $: $ & $e_1e_1=e_4$  & $e_1e_2=e_3$ & $e_1e_3=-e_5$ & $e_1e_4=2e_5$ \\ & & $e_2e_1=-e_3$ & $e_2e_2=2e_3+e_4$ & $e_2e_4=e_5$ \\

${\rm N}_{94}^{\alpha}$ & $: $ & $e_1e_1=e_4$  & $e_1e_2=e_3$ & $e_1e_3=-\alpha e_5$ & $e_1e_4=(2\alpha+1)e_5$ \\ & & $e_2e_1=-e_3$ & $e_2e_2=2e_3+e_4$ & $e_2e_3=e_5$ & $e_2e_4=\alpha e_5$ \\

${\rm N}_{95}$ & $: $ & $e_1e_1=e_4$  & $e_1e_2=e_3$ & $e_1e_3= e_5$ & $e_1e_4=-e_5$ \\ & & $e_2e_1=-e_3+e_5$ & $e_2e_2=2e_3+e_4$ & $e_2e_3=e_5$ & $e_2e_4=- e_5$ \\

${\rm N}_{96}^{\alpha}$ & $: $ & $e_1e_1=e_4$  & $e_1e_2=e_3$ & $e_1e_3= - e_5$ & $e_1e_4=\alpha e_5$ \\ & & $e_2e_1=-e_3+e_5$ & $e_2e_2=2e_3+e_4$ & $e_2e_3=(2+\alpha)e_5$ & $e_2e_4=- e_5$ \\
& & $e_3e_1=e_5$ & $e_3e_2=-(2+\alpha)e_5$ & $e_4e_1= \alpha e_5$ & $e_4e_2=e_5$  \\

${\rm N}_{97}$ & $: $ & $e_1e_1=e_4$  & $e_1e_2=e_3$ & $e_1e_3=-e_5$ \\ & &
$e_2e_1=-e_3$ & $e_2e_2=2e_3+e_4 + e_5$  & $e_2e_3=-2e_5$ & $e_2e_4=-e_5$\\
& &
$e_3e_1=e_5$ & $e_4e_1=-2e_5$ & $e_4e_2=e_5$\\

${\rm N}_{98}^{\alpha, \beta, \gamma}$ & $: $ & $e_1e_1=e_4$  & $e_1e_2=e_3$ & $e_1e_3=-(\beta +2)e_5$ & $e_1e_4=(\alpha+\beta -2\gamma)e_5$ \\ & &
$e_2e_1=-e_3$ & $e_2e_2=2e_3+e_4$ &
$e_2e_3=\alpha e_5$ &
$e_2e_4= \beta e_5$ \\
& &
$e_3e_1=e_5$ &
$e_3e_2= - (\gamma +2)e_5$ & $e_4e_1=\gamma e_5$ &
$e_4e_2= e_5$\\
 
${\rm N}_{99}^{\alpha, \beta}$ & $: $ & 
$e_1e_1=e_4$  & $e_1e_2=e_3$ &  $e_2e_2=2e_3+e_4+e_5$ & \\ &&
$e_1e_3=-(\alpha +2)e_5$ & $e_2e_1=-e_3$  & 
\multicolumn{2}{l}{$e_1e_4=(\sqrt{-2\alpha\beta-2\alpha-2\beta-1} - \beta) e_5$}  \\
&& $e_2e_4= \alpha e_5$  & $e_3e_1=e_5$  & 
\multicolumn{2}{l}{$e_2e_3=(\beta-\alpha+ \sqrt{-2\alpha\beta-2\alpha-2\beta-1}) e_5$} \\& &
$e_3e_2= - (\beta +2)e_5$ & $e_4e_1=\beta e_5$ &
$e_4e_2= e_5$\\

${\rm N}_{100}^{\alpha, \beta}$ & $: $ & $e_1e_1=e_4$  & $e_1e_2=e_3$ &  $e_2e_2=2e_3+e_4+e_5$ & \\ &&
$e_1e_3=-(\alpha +2)e_5$ & $e_2e_1=-e_3$ &  
\multicolumn{2}{l}{$e_1e_4=-(\sqrt{-2\alpha\beta-2\alpha-2\beta-1} + \beta) e_5$} \\
&& $e_2e_4= \alpha e_5$  & $e_3e_1=e_5$ &  
\multicolumn{2}{l}{$e_2e_3=(\beta-\alpha - \sqrt{-2\alpha\beta-2\alpha-2\beta-1}) e_5$}  \\&&
$e_3e_2= - (\beta +2)e_5$ & $e_4e_1=\beta e_5$ &
$e_4e_2= e_5$\\

${\rm N}_{101}^{\alpha, \beta}$ & $: $ &  $e_1e_2=e_4$ &  $e_1e_3= \alpha e_5$
&  $e_1e_4= e_5$ & $e_2e_2=e_3$  \\
&& $e_2e_3=e_5$
    & $e_2e_4=(\alpha - \beta) e_5$ & $e_4e_2=\beta e_5$\\

${\rm N}_{102}^{\alpha}$ & $: $ &  $e_1e_2=e_4$ &  $e_1e_3= e_5$
&  $e_1e_4= e_5$ & $e_2e_1=e_5$  \\
&& $e_2e_2=e_3$
    &  $e_2e_4=(1 - \alpha) e_5$ & $e_4e_2=\alpha e_5$\\

${\rm N}_{103}^{\alpha}$ & $: $ &  $e_1e_2=e_4$ &  $e_1e_3= e_5$
&  $e_1e_4= e_5$ & $e_2e_2=e_3$
\\  &   &  $e_2e_4=(1 - \alpha) e_5$ & $e_4e_2=\alpha e_5$\\

${\rm N}_{104}$ & $: $ & $e_1e_1=e_5$ & $e_1e_2=e_4$ &  $e_1e_3= e_5$\\ &
& $e_2e_1=e_4$ & $e_2e_2=e_3$ &   &  $e_2e_3=e_5$\\ 
& & $e_2e_4= e_5$ & $e_3e_1= e_5$ & $e_4e_2= e_5$ \\

${\rm N}_{105}$ & $: $ & $e_1e_1=e_5$ & $e_1e_2=e_4$ &  $e_1e_3= e_5$
& $e_2e_1=-e_4+e_5$\\ & & $e_2e_2=e_3$    &  $e_2e_4= - e_5$ & $e_3e_1= - e_5$ & $e_4e_2= e_5$ \\

${\rm N}_{106}^{\alpha}$ & $: $ & $e_1e_2=e_4$ & $e_1e_3=e_5$ & $e_2e_1=-2e_4$ & $e_2e_2=e_3$ \\ && $e_2e_3=e_5$ & $e_2e_4=e_5$ & $e_3e_2=\alpha e_5$ \\

${\rm N}_{107}^{\alpha}$ & $: $ & $e_1e_1=e_5$ & $e_1e_2=e_4$ & $e_1e_3= \frac 1 2 e_5$ & $e_2e_1= - \frac 1 2 e_4$\\ & & $e_2e_2=e_3$  & $e_2e_3=\alpha e_5$ & $e_2e_4= - e_5$ \\ && $e_3e_1= - \frac 1 2  e_5$  & $e_3e_2= e_5$  & $e_4e_2= e_5$\\

${\rm N}_{108}^{\alpha}$ & $: $ & $e_1e_2=e_4$ & $e_1e_3= \frac 1 2 e_5$ & $e_2e_1= - \frac 1 2 e_4$\\ & & $e_2e_2=e_3$   & $e_2e_3=\alpha e_5$ & $e_2e_4= - e_5$ \\ && $e_3e_1= - \frac 1 2  e_5$  & $e_3e_2= e_5$  & $e_4e_2= e_5$\\

${\rm N}_{109}$ & $: $ & $e_1e_2=e_4$ & $e_1e_3=e_5$ & $e_2e_1=- \frac 1 2 e_4 + e_5$ & $e_2e_2=e_3$ \\ &&  $e_2e_4= - \frac 1 2 e_5$ & $e_3e_2= e_5$  & $e_3e_1= - \frac 1 2  e_5$  & $e_4e_2= e_5$\\

${\rm N}_{110}^{\alpha}$ & $: $ & $e_1e_2=e_4$ & $e_1e_3=\alpha e_5$ & $e_2e_1=- \frac 1 2 e_4 $ & $e_2e_2=e_3$ \\ &&  $e_2e_4= (\alpha - \frac 3 2) e_5$ & $e_3e_2= e_5$  & $e_3e_1= - \frac 1 2  e_5$  & $e_4e_2= e_5$\\

${\rm N}_{111}^{ \alpha \neq - \frac 1 2 }$ & $: $ & $e_1e_2=e_4$  & $e_1e_3=e_5$ & $e_2e_1=\alpha e_4+e_5$ & $e_2e_2=e_3$   \\ & & $e_2e_3= e_5$ & $e_2e_4= \alpha e_5$ & $e_3e_1= \alpha  e_5$  & $e_4e_2= e_5$\\

${\rm N}_{112}^{\alpha}$ & $: $ & $e_1e_2=e_4$  & $e_1e_3=e_5$ & $e_2e_1=\alpha e_4+e_5$ & $e_2e_2=e_3$   \\ & & $e_2e_4= \alpha e_5$ & $e_3e_1= \alpha  e_5$  & $e_4e_2= e_5$\\

${\rm N}_{113}^{\alpha}$ & $: $ & $e_1e_1=e_5$ & $e_1e_2=e_4$  & $e_1e_3= - \alpha e_5$ \\ & &   $e_2e_1=\alpha e_4$ & $e_2e_2=e_3$ &  & $e_2e_3=e_5$ \\ && $e_2e_4= - e_5$ & $e_3e_1= \alpha  e_5$  & $e_4e_2= e_5$\\ 

${\rm N}_{114}^{\alpha}$ & $: $ & $e_1e_1=e_5$ & $e_1e_2=e_4$  & $e_1e_3= - \alpha e_5$  &   $e_2e_1=\alpha e_4$ \\ && $e_2e_2=e_3$ &    $e_2e_4= - e_5$ & $e_3e_1= \alpha  e_5$  & $e_4e_2= e_5$\\ 

${\rm N}_{115}^{\alpha, \beta}$ & $: $ & $e_1e_2=e_4$  & $e_1e_3= \beta e_5$ &  $e_2e_1=\alpha e_4$ & $e_2e_2=e_3$ \\ &&  $e_2e_3=e_5$  & $e_2e_4= (\beta +\alpha -1) e_5$ &$e_3e_1= \alpha  e_5$  & $e_4e_2= e_5$\\

${\rm N}_{116}^{\alpha, \beta}$ & $: $ & $e_1e_2=e_4$  & $e_1e_3= \beta e_5$ &  $e_2e_1=\alpha e_4$ & $e_2e_2=e_3$   \\ && $e_2e_4= (\beta +\alpha -1) e_5$ &$e_3e_1= \alpha  e_5$  & $e_4e_2= e_5$\\

${\rm N}_{117}^{\alpha}$ & $: $ & $e_1e_2=e_4$  &  $e_1e_3= e_5$ & $e_2e_1=\alpha e_4$\\ & & $e_2e_2=e_3$    & $e_2e_4=e_5$ & $e_3e_2=e_5$\\

${\rm N}_{118}^{\alpha}$ & $: $ & $e_1e_2=e_4$  &  $e_1e_3= e_5$ & $e_2e_1=\alpha e_4$ \\ && $e_2e_2=e_3$   & $e_2e_4=e_5$ \\

${\rm N}_{119}$ & $: $ & $e_1e_1=e_2$  & $e_1e_3=e_5$ & $e_2e_1=e_3$ & $e_4e_4=e_5$  \\
${\rm N}_{120}$ & $: $ & $e_1e_1=e_2$  & $e_1e_3=e_5$ & $e_1e_4=e_5$ & $e_2e_1=e_3$  \\

${\rm N}_{121}$ & $: $ & $e_1e_1=e_2$  & $e_1e_2=e_3$ & $e_1e_3=2e_5$ & $e_4e_1=e_5$  \\
${\rm N}_{122}$ & $: $ & $e_1e_1=e_2$  & $e_1e_2=e_3$ & $e_1e_3=2e_5$ & $e_2e_1=e_5$  \\
${\rm N}_{123}$ & $: $ & $e_1e_1=e_2$  & $e_1e_2=e_3$ & $e_1e_3=2e_5$ \\ && $e_2e_1=e_5$ & $e_4e_1=e_5$  \\
${\rm N}_{124}$ & $: $ & $e_1e_1=e_2$  & $e_1e_2=e_3$ & $e_1e_3=2e_5$ \\ && $e_2e_1=e_5$ & $e_4e_4=e_5$  \\
${\rm N}_{125}^{\alpha\neq1}$ & $: $ & $e_1e_1=e_2$  & $e_1e_2=e_3$ & $e_1e_3=(2-\alpha)e_5$ & $e_2e_1=\alpha e_3$ \\&& $e_2e_2=\alpha e_5$ & $e_3e_1=\alpha e_5$ & $e_4e_4=e_5$\\

${\rm N}_{126}^{\alpha\neq0,1}$ & $: $ & $e_1e_1=e_2$  & $e_1e_2=e_3$ & $e_1e_3=(2-\alpha)e_5$ & $e_1e_4=e_5$ \\&& $e_2e_1=\alpha e_3$ & $e_2e_2=\alpha e_5$ & $e_3e_1=\alpha e_5$\\

${\rm N}_{127}$ & $: $ & $e_1e_1=e_2$  & $e_1e_2=e_3$ & $e_1e_3=e_5$ & $e_2e_1=e_3+e_5$ \\&& $e_2e_2=e_5$ & $e_3e_1=e_5$ & $e_4e_1=e_5$\\

${\rm N}_{128}$ & $: $ & $e_1e_1=e_2$  & $e_1e_2=e_3$ & $e_1e_3=e_5$ & $e_2e_1=e_3$ \\&& $e_2e_2=e_5$ & $e_3e_1=e_5$ & $e_4e_1=e_5$\\

${\rm N}_{129}$ & $: $ & $e_1e_1=e_2$  & $e_1e_2=e_3$ & $e_1e_3=e_5$ & $e_2e_1=e_3$ \\&& $e_2e_2=e_5$ & $e_3e_1=e_5$ & $e_4e_1=e_5$ & $e_4e_4=e_5$\\
${\rm N}_{130}^{\alpha}$ & $: $ & $e_1e_1=e_2$  & $e_1e_2=e_3$ & $e_1e_3=e_5$ & $e_2e_1=e_3+e_5$ \\&& $e_2e_2=e_5$ & $e_3e_1=e_5$ & $e_4e_1=e_5$ & $e_4e_4=\alpha e_5$\\

${\rm N}_{131}^{\alpha\neq0, \beta}$ & $: $ & $e_1e_1=e_2$  & $e_1e_2=e_4$ & $e_1e_3=e_4$ & $e_1e_4=\frac{2-\alpha}{\alpha}e_5$ \\ && 
$e_2e_1=\alpha e_4$   & $e_2e_2=e_5$  & $e_2e_3=e_5$\\ & & 
$e_3e_2=-e_5$ & $e_3e_3=\beta e_5$ & $e_4e_1=e_5$ \\

${\rm N}_{132}^{\alpha\neq0,1}$ & $: $ & $e_1e_1=e_2$  & $e_1e_2=e_4$ & $e_1e_3=e_4+e_5$  
 & $e_1e_4=\frac{2-\alpha}{\alpha}e_5$\\ & & $e_2e_1=\alpha e_4$   & $e_2e_2=e_5$  
 & $e_2e_3=e_5$ \\ && $e_3e_2=-e_5$ & $e_3e_3=\frac{1}{(\alpha-1)^2} e_5$ & $e_4e_1=e_5$ \\

${\rm N}_{133}$ & $: $ & $e_1e_1=e_2$  & $e_1e_3=e_5$ & $e_1e_4=-e_5$ & $e_2e_1=e_4$ \\ && $e_2e_3=e_4$   & $e_3e_1=e_5$ & $e_3e_2=-e_5$ & $e_4e_1=e_5$ \\
${\rm N}_{134}^{\alpha}$ & $: $ & $e_1e_1=e_2$  & $e_1e_3=e_5$ & $e_1e_4=-e_5$ & $e_2e_1=e_4$\\ & & $e_2e_3=e_4$   & $e_3e_2=-e_5$ & $e_3e_3=\alpha e_5$ & $e_4e_1=e_5.$\\

${\rm N}_{135}^{\alpha}$ & $: $ & $e_1e_1=e_2$  & $e_1e_2=e_4$ & $e_1e_4=e_5$ & $e_2e_1=\alpha e_5$ \\&& $e_3e_1=e_4+e_5$ & $e_3e_2=e_5$ & $e_3e_3=2e_5$\\

${\rm N}_{136}^{\alpha}$ & $: $ & $e_1e_1=e_2$  & $e_1e_2=e_4$ & $e_1e_4=e_5$ \\&& $e_3e_1=e_4$ & $e_3e_2=e_5$ & $e_3e_3=\alpha e_5$\\

${\rm N}_{137}^{\alpha}$ & $: $ & $e_1e_1=e_2$  & $e_1e_2=e_4$ & $e_1e_4=e_5$ & $e_2e_1=e_5$ \\&& $e_3e_1=e_4$ & $e_3e_2=e_5$ & $e_3e_3=\alpha e_5$\\

${\rm N}_{138}$ & $: $ & $e_1e_2=e_3$ & $e_1e_3=e_4$ & $e_1e_4=e_5$ \\
${\rm N}_{139}$ & $: $ & $e_1e_2=e_3$ & $e_1e_3=e_4$ & $e_1e_4=e_5$ & $e_2e_2=e_5$\\
${\rm N}_{140}$ & $: $ & $e_1e_2=e_3$ & $e_1e_3=e_4$ & $e_1e_4=e_5$ & $e_2e_1=e_5$\\
${\rm N}_{141}$ & $: $ & $e_1e_2=e_3$ & $e_1e_3=e_4$ & $e_1e_4=e_5$  \\
&& $e_2e_1=e_5$ & $e_2e_2=e_5$\\
${\rm N}_{142}$ & $: $ & $e_1e_2=e_3$ & $e_1e_3=e_4$ & $e_1e_4=e_5$ \\&& $e_2e_3=e_5$ & $e_3e_2=-e_5$\\
${\rm N}_{143}$ & $: $ & $e_1e_2=e_3$ & $e_1e_3=e_4$ & $e_1e_4=e_5$ & $e_2e_2=e_5$ \\ & & $e_2e_3=e_5$ & $e_3e_2=-e_5$\\
${\rm N}_{144}^{\alpha}$ & $: $ & $e_1e_2=e_3$ & $e_1e_3=e_4$ & $e_1e_4=e_5$ & $e_2e_1=e_5$ \\&& $e_2e_2=\alpha e_5$ & $e_2e_3=e_5$ & $e_3e_2=-e_5$\\

${\rm N}_{145}^{\alpha}$ & $: $ & $e_1e_2=e_3$ & $e_1e_3=e_4$ & $e_1e_4=e_5$ \\&& $e_2e_2=e_4$ & $e_2e_3=(1-\alpha)e_5$ & $e_3e_2=\alpha e_5$\\
${\rm N}_{146}^{\alpha}$ & $: $ & $e_1e_2=e_3$ & $e_1e_3=e_4$ & $e_1e_4=e_5$ & $e_2e_1=e_5$ \\&& $e_2e_2=e_4$ & $e_2e_3=(1-\alpha)e_5$ & $e_3e_2=\alpha e_5$\\
${\rm N}_{147}^{\alpha,\beta}$ & $: $ & $e_1e_2=e_3$ & $e_1e_3=e_4$ & $e_1e_4=e_5$ & $e_2e_1=\alpha e_5$ \\&& $e_2e_2=e_4+e_5$ & $e_2e_3=(1-\beta)e_5$ & $e_3e_2=\beta e_5$\\

${\rm N}_{148}^{\alpha}$ & $:$& $e_1 e_1 = e_2$  & $e_1e_2=\alpha e_4$  & $e_2 e_1= (\alpha+1)e_4$  & $e_3e_3=e_5$  \\

${\rm N}_{149}^{\alpha}$ & $:$& $e_1 e_1 = e_2$  & $e_1e_2=\alpha e_4+e_5$  & $e_2 e_1= (\alpha+1)e_4+e_5$  & $e_3e_3=e_5$  \\

${\rm N}_{150}^{\alpha}$ & $:$& $e_1 e_1 = e_2$  & $e_1e_2=\alpha e_4$  & $e_2 e_1= (\alpha+1)e_4$ \\ & & $e_3e_1=e_5$ & $e_3e_3=e_5$  \\

${\rm N}_{151}^{\alpha}$ & $:$& $e_1 e_1 = e_2$  & $e_1e_2=\alpha e_4+e_5$  & $e_2 e_1= (\alpha+1)e_4+e_5$ \\ & & $e_3e_1=e_5$ & $e_3e_3=e_5$  \\

${\rm N}_{152}^{\alpha}$ & $:$& $e_1 e_1 = e_2$  & $e_1e_2=\alpha e_4$  & $e_1e_3=e_4$  \\ & & $e_2 e_1= (\alpha+1)e_4$ & $e_3e_1=e_4$ & $e_3e_3=e_5$  \\

${\rm N}_{153}^{\alpha}$ & $:$& $e_1 e_1 = e_2$  & $e_1e_2=\alpha e_4 + e_5$  & $e_1e_3=e_4$  \\ & & $e_2 e_1= (\alpha+1)e_4+e_5$ & $e_3e_1=e_4$ & $e_3e_3=e_5$  \\

${\rm N}_{154}^{\alpha, \beta}$ & $:$& $e_1 e_1 = e_2$  & $e_1e_2=\alpha e_4 + \beta e_5$  & $e_1e_3=e_4$  \\ & & $e_2 e_1= (\alpha+1)e_4+\beta e_5$ & $e_3e_1=e_4+e_5$ & $e_3e_3=e_5$  \\

${\rm N}_{155}^{\alpha, \beta}$ & $:$& $e_1 e_1 = e_2$  & $e_1e_2=\alpha e_4 $  & $e_1e_3=\beta e_5$  \\ & & $e_2 e_1= (\alpha+1)e_4$ & $e_3e_1=(\beta+1) e_5$    \\

${\rm N}_{156}^{\alpha}$ & $:$& $e_1 e_1 = e_2$  & $e_1e_2=\alpha e_4 $  & $e_1e_3=e_4+\alpha e_5$  \\ & & $e_2 e_1= (\alpha+1)e_4$ & $e_3e_1=e_4+(\alpha+1) e_5$ \\

${\rm N}_{157}^{\alpha, \beta}$ & $:$& $e_1 e_1 = e_2$  & $e_1e_2=\alpha e_4 + e_5 $  & $e_1e_3=\beta e_5$  \\ & & $e_2 e_1= (\alpha+1)e_4+e_5$ & $e_3e_1=(\beta+1) e_5$ \\

${\rm N}_{158}^{\alpha, \beta}$ & $:$& $e_1 e_1 = e_2$  & $e_1e_2=\alpha e_4 $  & $e_1e_3=\beta e_5$  \\ & & $e_2 e_1= (\alpha+1)e_4$ & $e_3e_1=(\beta+1) e_5$ & $e_3e_3= e_4$\\

${\rm N}_{159}^{\alpha}$ & $:$& $e_1 e_1 = e_2$  & $e_1e_2=\alpha e_4+e_5 $  & $e_2 e_1= (\alpha+1)e_4+e_5$ \\ & & $e_3e_1=e_5$ & $e_3e_3= e_4 $\\

${\rm N}_{160}$ & $:$& $e_1 e_1 = e_2$  & $e_1e_2=e_4 $  & $e_2 e_1= e_4+e_5$ & $e_3e_1=e_5$ \\

${\rm N}_{161}^{\alpha}$ & $:$& $e_1 e_1 = e_2$  & $e_1e_2=e_4 $ & $e_1e_3=e_5$ \\ & & $e_2 e_1= e_4+e_5$ & $e_3e_1=\alpha e_5$ \\

${\rm N}_{162}$ & $:$& $e_1 e_1 = e_2$  & $e_1e_2=e_4 $ & $e_2 e_1= e_4+e_5$ & $e_3e_3= e_5$ \\

${\rm N}_{163}$ & $:$& $e_1 e_1 = e_2$  & $e_1e_2=e_4 $ & $e_1e_3=e_4 $ \\ & & $e_2 e_1= e_4+e_5$ & $e_3 e_1= e_4$ & $e_3e_3= e_5$ \\

${\rm N}_{164}^{\alpha}$ & $:$& $e_1 e_1 = e_2$  & $e_1e_2=\alpha e_4 $ & $e_1e_3=e_5 $ \\ & & $e_2 e_1= (\alpha+1)e_4$ & $e_3 e_1= e_5$ \\

${\rm N}_{165}^{\alpha}$ & $:$& $e_1 e_1 = e_2$  & $e_1e_2=\alpha e_4 $ & $e_1e_3=e_5 $ \\ & & $e_2 e_1= (\alpha+1)e_4$ & $e_3 e_1= e_5$ & $e_3 e_3= e_4$\\

${\rm N}_{166}$ & $:$& $e_1 e_1 = e_2$  & $e_1e_2= e_4 $ & $e_2 e_1= e_4$ \\ & & $e_3 e_1= e_4$ & $e_3 e_3= e_5$\\

${\rm N}_{167}^{\alpha}$ & $:$& $e_1 e_1 = e_2$  & $e_1e_2= e_4 $ & $e_1 e_3= \alpha e_5$ \\ & & $e_2 e_1= e_4$  & $e_3 e_1= (\alpha+1)e_5$ & $e_3 e_3= e_4$\\

${\rm N}_{168}$ & $:$& $e_1 e_1 = e_2$  & $e_1e_2= e_4 +e_5 $ & $e_2 e_1= e_4+e_5$  \\ & & $e_3 e_1= e_4$ & $e_3 e_3= e_5$\\

${\rm N}_{169}^{\alpha}$ & $:$& $e_1 e_1 = e_2$  & $e_1e_2= e_4$ & $e_1e_3= \alpha e_5$ \\ & & $e_2 e_1= e_4$  & $e_3 e_1= (\alpha+1)e_5$\\

${\rm N}_{170}$ & $:$& $e_1 e_1 = e_2$  & $e_1e_2= e_4$ & $e_2 e_1= e_4$ \\ & & $e_3 e_1= e_5$ & $e_3 e_3= e_5$\\

${\rm N}_{171}$ & $:$& $e_1 e_1 = e_2$  & $e_1e_2= e_4$ & $e_1 e_3= e_5$ \\ & & $e_2 e_1= e_4$ & $e_3 e_1= e_4+ e_5$ \\

${\rm N}_{172}$ & $:$& $e_1 e_1 = e_2$  & $e_1e_2= e_4$ & $e_1 e_3= e_5$ \\ & & $e_2 e_1= e_4$ & $e_3 e_1= e_4+ e_5$ & $e_3 e_3= e_4$ \\

${\rm N}_{173}$ & $:$& $e_1 e_2 = e_3$  & $e_1e_3=e_4$  & $e_2 e_3= e_5$  & $e_3e_2=-e_5$  \\

${\rm N}_{174}$ & $:$ & $e_1 e_1 = e_5$  & $e_1 e_2 = e_3$  & $e_1e_3=e_4$  \\
& & $e_2 e_3= e_5$  & $e_3e_2=-e_5$  \\

${\rm N}_{175}$ & $:$ &  $e_1 e_2 = e_3$  & $e_1e_3=e_4$ & $e_2 e_2 = e_5$  \\
& &  $e_2 e_3= e_5$  & $e_3e_2=-e_5$  \\

${\rm N}_{176}$ & $:$ & $e_1 e_1 = e_5$ &  $e_1 e_2 = e_3$  & $e_1e_3=e_4$ \\ & & $e_2 e_2 = e_5$  &  $e_2 e_3= e_5$  & $e_3e_2=-e_5$  \\

${\rm N}_{177}$ & $:$ &  $e_1 e_2 = e_3$  & $e_1e_3=e_4$ & $e_2 e_2 = e_4$  \\
& &  $e_2 e_3= e_5$  & $e_3e_2=-e_5$  \\

${\rm N}_{178}$ & $:$ & $e_1 e_1 = e_5$ &  $e_1 e_2 = e_3$  & $e_1e_3=e_4$ \\ & & $e_2 e_2 = e_4$  &  $e_2 e_3= e_5$  & $e_3e_2=-e_5$  \\

${\rm N}_{179}^{\alpha}$ & $:$ & $e_1 e_1 = \alpha e_5$ &  $e_1 e_2 = e_3$  & $e_1e_3=e_4$ \\ & & $e_2 e_2 = e_4 + e_5$  &  $e_2 e_3= e_5$  & $e_3e_2=-e_5$  \\

${\rm N}_{180}^{\alpha}$ & $:$ & $e_1 e_2 = e_3$  & $e_1e_3=e_4$
& $e_2 e_1 = e_4$ &
\\ & & $e_2 e_2 = \alpha e_5$  &  $e_2 e_3= e_5$  & $e_3e_2=-e_5$  \\

${\rm N}_{181}^{\alpha}$ & $:$ &  $e_1 e_1 = e_5$ & $e_1 e_2 = e_3$  & $e_1e_3=e_4$
& $e_2 e_1 = e_4$  
\\ & & $e_2 e_2 = \alpha e_5$  &  $e_2 e_3= e_5$  & $e_3e_2=-e_5$  \\

${\rm N}_{182}^{\alpha, \beta}$ & $:$ &  $e_1 e_1 = \alpha e_5$ & $e_1 e_2 = e_3$  & $e_1e_3=e_4$
& $e_2 e_1 = e_4$  
\\ & & $e_2 e_2 = e_4+ \beta e_5$  &  $e_2 e_3= e_5$  & $e_3e_2=-e_5$  \\

${\rm N}_{183}$ & $:$ & $e_1 e_1 = e_4$  & $e_1 e_2 = e_3$  & $e_1e_3=e_4$   \\
&& $e_2 e_3= e_5$  & $e_3e_2=-e_5$  \\

${\rm N}_{184}$ & $:$ & $e_1 e_1 = e_4$  & $e_1 e_2 = e_3$  & $e_1e_3=e_4$ \\ & & $e_2 e_2= e_5$ & $e_2 e_3= e_5$  & $e_3e_2=-e_5$  \\

${\rm N}_{185}^{\alpha}$ & $:$ & $e_1 e_1 = e_4+e_5$  & $e_1 e_2 = e_3$  & $e_1e_3=e_4$ \\ & & $e_2 e_2= \alpha e_5$ & $e_2 e_3= e_5$  & $e_3e_2=-e_5$  \\

${\rm N}_{186}^{\alpha, \beta}$ & $:$ & $e_1 e_1 = e_4+\alpha e_5$  & $e_1 e_2 = e_3$  & $e_1e_3=e_4$ \\ & & $e_2 e_2= e_4+ \beta e_5$ & $e_2 e_3= e_5$  & $e_3e_2=-e_5$  \\

${\rm N}_{187}^{\alpha, \beta, \gamma}$ & $:$ & $e_1 e_1 = e_4+\alpha e_5$  & $e_1 e_2 = e_3$  & $e_1e_3=e_4$ & $e_2e_1=e_4$ \\ & & $e_2 e_2= \beta e_4+  \gamma e_5$ & $e_2 e_3= e_5$  & $e_3e_2=-e_5$  \\

${\rm N}_{188}$ & $:$& $e_1 e_2 = e_3$  & $e_1e_3=e_4$  & $e_2 e_2= e_5$  \\

${\rm N}_{189}$ & $:$ & $e_1 e_1 = e_5$ & $e_1 e_2 = e_3$  & $e_1e_3=e_4$  & $e_2 e_2= e_5$  \\

${\rm N}_{190}$ & $:$ & $e_1 e_1 = e_5$ & $e_1 e_2 = e_3$  & $e_1e_3=e_4$  \\
&& $e_2 e_1= e_5$ & $e_2 e_2= e_5$  \\

${\rm N}_{191}^{\alpha, \beta}$ & $:$ & $e_1 e_1 = \alpha e_5$ & $e_1 e_2 = e_3$  & $e_1e_3=e_4$ & $e_2 e_1= \beta e_5$ \\ & & $e_2 e_2= e_5$ & $e_2 e_3= e_4$ & $e_3 e_2= - e_4$  \\

${\rm N}_{192}$ & $:$& $e_1 e_2 = e_3$  & $e_1e_3=e_4$ & $e_2 e_1= e_4$ & $e_2 e_2= e_5$  \\

${\rm N}_{193}$ & $:$& $e_1 e_1 = e_5$ & $e_1 e_2 = e_3$  & $e_1e_3=e_4$ \\
& & $e_2 e_1= e_4$ & $e_2 e_2= e_5$  \\

${\rm N}_{194}^{\alpha}$ & $:$& $e_1 e_1 = \alpha e_5$ & $e_1 e_2 = e_3$  & $e_1e_3=e_4$  \\
&& $e_2 e_1= e_4+e_5$ & $e_2 e_2= e_5$  \\

${\rm N}_{195}^{\alpha, \beta }$ & $:$& $e_1 e_1 = \alpha e_5$ & $e_1 e_2 = e_3$  & $e_1e_3=e_4$ & $e_2 e_1= e_4+\beta e_5$ \\ & & $e_2 e_2= e_5$ & $e_2 e_3= e_4$ & $e_3 e_2= - e_4$ \\

%${\rm N}_{195}$ & $:$& $e_1 e_2 = e_3$  & $e_1e_3=e_4$  & $e_2 e_1= e_5$  \\

%${\rm N}_{196}$ & $:$& $e_1 e_2 = e_3$  & $e_1e_3=e_4$  & $e_2 e_1= e_5$  & $e_2e_2=e_4$ \\

${\rm N}_{196}$ & $:$& $e_1e_1=e_5$ & $e_1 e_2 = e_3$  & $e_1e_3=e_4$  & $e_2 e_1= e_5$   \\

${\rm N}_{197}^{\alpha}$ & $:$& $e_1e_1= \alpha e_5$ & $e_1 e_2 = e_3$  & $e_1e_3=e_4$ \\ & & $e_2 e_1= e_5$  & $e_2 e_3= e_4$ & $e_3 e_2= - e_4$ \\

${\rm N}_{198}^{\alpha}$ & $:$& $e_1e_1= \alpha e_5$ & $e_1 e_2 = e_3$  & $e_1e_3=e_4$ & $e_2 e_1= e_5$ \\ &  & $e_2 e_2= e_4$ & $e_2 e_3= e_4$ & $e_3 e_2= - e_4$ \\

${\rm N}_{199}^{\alpha}$ & $:$& $e_1e_1= \alpha e_5$ & $e_1 e_2 = e_3$  & $e_1e_3=e_4$ & $e_2 e_1= e_4+e_5$ \\ &  & $e_2 e_2= e_4$ & $e_2 e_3= e_4$ & $e_3 e_2= - e_4$ \\

${\rm N}_{200}$ & $:$& $e_1 e_1 = e_5$ & $e_1 e_2 = e_3$  & $e_1e_3=e_4$ \\

${\rm N}_{201}$ & $:$& $e_1 e_1 = e_5$ & $e_1 e_2 = e_3$  & $e_1e_3=e_4$ & $e_2e_1=e_4$\\

${\rm N}_{202}$ & $:$& $e_1 e_1 = e_5$ & $e_1 e_2 = e_3$  & $e_1e_3=e_4$ & $e_2e_2=e_4$\\

${\rm N}_{203}$ & $:$& $e_1 e_1 = e_5$ & $e_1 e_2 = e_3$  & $e_1e_3=e_4$  \\
&& $e_2e_1=e_4$ & $e_2e_2=e_4$\\

${\rm N}_{204}$ & $:$& $e_1 e_1 = e_5$ & $e_1 e_2 = e_3$  & $e_1e_3=e_4$  \\
&& $e_2e_3=e_4$ & $e_3e_2= - e_4$\\

${\rm N}_{205}$ & $:$& $e_1 e_1 = e_5$ & $e_1 e_2 = e_3$  & $e_1e_3=e_4$ \\ & &  $e_2e_2=e_4$  & $e_2e_3=e_4$ & $e_3e_2= - e_4$\\

${\rm N}_{206}$ & $:$& $e_1 e_2 = e_3$  & $e_2e_2=e_5$  & $e_2e_3=e_4$ & $e_3e_2= - e_4$\\

${\rm N}_{207}$ & $:$& $e_1 e_1 = e_5$ & $e_1 e_2 = e_3$  & $e_2e_2=e_5$  \\
& & $e_2e_3=e_4$ & $e_3e_2= - e_4$\\

${\rm N}_{208}^{\alpha}$ & $:$& $e_1 e_1 = \alpha e_5$ & $e_1 e_2 = e_3$  & $e_2e_1=e_5$ \\ &  & $e_2e_2=e_5$  & $e_2e_3=e_4$ & $e_3e_2= - e_4$\\

%${\rm N}_{210}$ & $:$& $e_1 e_1 = e_4$ & $e_1 e_2 = e_3$  & $e_2e_2=e_5$   \\
%&& $e_2e_3=e_4$ & $e_3e_2= - e_4$\\

${\rm N}_{209}$ & $:$& $e_1 e_1 = e_4+e_5$ & $e_1 e_2 = e_3$  & $e_2e_2=e_5$   \\
&& $e_2e_3=e_4$ & $e_3e_2= - e_4$\\

${\rm N}_{210}^{\alpha}$ & $:$& $e_1 e_1 = e_4+\alpha e_5$ & $e_1 e_2 = e_3$  & $e_2e_1=e_5$ \\ & & $e_2e_2=e_5$  & $e_2e_3=e_4$ & $e_3e_2= - e_4$\\

${\rm N}_{211}$ & $:$& $e_1 e_2 = e_3$  & $e_2e_1=e_5$  & $e_2e_3=e_4$ & $e_3e_2= - e_4$\\

${\rm N}_{212}$ & $:$& $e_1 e_1 = e_4$ & $e_1 e_2 = e_3$  & $e_2e_1=e_5$  \\
& & $e_2e_3=e_4$ & $e_3e_2= - e_4$\\

${\rm N}_{213}$ & $:$& $e_1 e_1 = e_5$ & $e_1 e_2 = e_3$  & $e_2e_1=e_5$   \\
&& $e_2e_3=e_4$ & $e_3e_2= - e_4$\\

${\rm N}_{214}$ & $:$& $e_1 e_2 = e_3$  & $e_2e_1=e_5$  & $e_2e_2=e_4$   \\
&& $e_2e_3=e_4$ & $e_3e_2= - e_4$\\

${\rm N}_{215}$ & $:$& $e_1 e_1 = e_4$ & $e_1 e_2 = e_3$  & $e_2e_1=e_5$  \\ & & $e_2e_2=e_4$  & $e_2e_3=e_4$ & $e_3e_2= - e_4$\\

${\rm N}_{216}$ & $:$& $e_1 e_1 = e_5$ & $e_1 e_2 = e_3$  & $e_2e_1=e_5$  \\ & & $e_2e_2=e_4$  & $e_2e_3=e_4$ & $e_3e_2= - e_4$\\

${\rm N}_{217}$ & $:$& $e_1 e_1 = e_5$ & $e_1 e_2 = e_3$  &  $e_2e_3=e_4$ & $e_3e_2= - e_4$\\

${\rm N}_{218}$ & $:$& $e_1 e_1 = e_5$ & $e_1 e_2 = e_3$ & $e_2e_2=e_4$  \\
&&  $e_2e_3=e_4$ & $e_3e_2= - e_4$\\

\end{longtable}}

Note that 

\begin{center}

${\rm N}_{12}^{\frac 1 4} \simeq {\rm N}_{11}^{\frac 1 4},$
${\rm N}_{16}^{\alpha, \beta} \simeq {\rm N}_{16}^{\beta, \alpha},$
${\rm N}_{35}^{\alpha} \simeq {\rm N}_{35}^{-\alpha},$
${\rm N}_{37}^{\alpha} \simeq {\rm N}_{37}^{-\alpha},$
${\rm N}_{42}^{\alpha,\beta} \simeq {\rm N}_{42}^{\alpha,-\beta},$
${\rm N}_{60}^{0,0} \simeq {\rm N}_{213}^{0},$
${\rm N}_{61}^{0,0, -1} \simeq {\rm N}_{210}^{0},$
${\rm N}_{61}^{0,0,  \gamma \neq -1 } \simeq {\rm N}_{191}^{0, -\frac {1}{\gamma+1}},$
%${\rm N}_{62}^{0,0} \simeq {\rm N}_{189}^{0},$
${\rm N}_{63}^{0} \simeq {\rm N}_{194}^{0},$
${\rm N}_{64}^{0, 0} \simeq {\rm N}_{150}^{-1},$  
%${\rm N}_{64}^{\alpha\notin \{ 0, 1\}, 0} \simeq {\rm N}_{189}^{\frac 1 {1-\alpha}},$
${\rm N}_{87}^{\alpha} \simeq {\rm N}_{87}^{\frac 1 {\alpha}},  $  
${\rm N}_{88}^{\alpha,\beta} \simeq {\rm N}_{88}^{\beta,\alpha},$
%${\rm N}_{89}^{0} \simeq {\rm N}_{195},$ 
%${\rm N}_{90}^{0} \simeq {\rm N}_{196},$ 
${\rm N}_{91}^{0} \simeq {\rm N}_{197}^{0},$ ${\rm N}_{92}^{0} \simeq {\rm }_{198}^{0},$
${\rm N}_{94}^{\alpha} \simeq {\rm N}_{94}^{\frac{1}{\alpha}},$ 
${\rm N}_{96}^{\alpha} \simeq {\rm N}_{96}^{\frac{1}{2+\alpha}},$ 
${\rm N}_{98}^{\alpha,\beta,\gamma} \simeq {\rm N}_{98}^{\frac{2+\beta}{2+\gamma},\frac{\alpha-2(2+\gamma)}{2+\gamma},-\frac{3+2\gamma}{2+\gamma}},$
${\rm N}_{111}^{-\frac 1 2} \simeq {\rm N}_{112}^{-\frac 1 2},$ 
%${\rm N}_{114}^{0} \simeq {\rm N}_{210},$

${\rm N}_{125}^{1}$ and ${\rm N}_{126}^{1}$ are commutative-associative algebras,

${\rm N}_{126}^{0}$ is a split algebra.

\end{center}

\end{theoremA}


\begin{thebibliography}{99}




%%%%%% ъ
 
  
\bibitem{bn85}
Balinsky A., Novikov S.,
 Poisson brackets of hydrodynamic type, Frobenius algebras and Lie algebras,
 Soviet Mathematics Doklady, 32 (1985), 1, 228--231.



\bibitem{bc01}
Bai C., Meng D.,
The classification of Novikov algebras in low dimensions,
Journal of Physics A, 34 (2001),  8, 1581--1594.



 	\bibitem{bfk23}   Beites P.,    Fern\'andez Ouaridi A.,   Kaygorodov I., 
        The algebraic and geometric classification of transposed Poisson algebras, Revista de la Real Academia de Ciencias Exactas, Físicas y Naturales. Serie A. Matemáticas,   117 (2023), 2,  Paper 55.

  

\bibitem{bb14}
Beneš T., Burde D.,
Classification of orbit closures in the variety of three-dimensional Novikov algebras,
Journal of Algebra and Its Applications, 13 (2014),  2, 1350081, 33 pp.


\bibitem{zerui18}
Bokut L.,   Chen Yu.,  Zhang Z., 
    On free Gelfand-Dorfman-Novikov-Poisson algebras and a PBW theorem,  
    Journal of Algebra, 500 (2018), 153--170. 



\bibitem{bur06}
Burde D., 
Left-symmetric algebras, or pre-Lie algebras in geometry and physics,  
Central European Journal of Mathematics, 4 (2006),  3, 323--357.


\bibitem{bdv08}
Burde D., Dekimpe K.,  Vercammen K., 
Novikov algebras and Novikov structures on Lie algebras, 
Linear Algebra and Its Applications, 429 (2008),   1, 31--41.

\bibitem{bg13}
Burde D., de Graaf W.,
Classification of Novikov algebras,
Applicable Algebra in Engineering, Communication and Computing, 24 (2013),  1, 1--15.




 


\bibitem{cfk18}
Calder\'on A.,  Fern\'andez Ouaridi A., Kaygorodov I.,  
On the classification of bilinear maps with   radical of a fixed codimension,  
   Linear and Multilinear Algebra,   70  (2022), 18, 3553--3576.

\bibitem{ckkk20}
   Camacho L., Karimjanov I., Kaygorodov I., Khudoyberdiyev  A.,  
    One-generated nilpotent Novikov algebras, 
    Linear and Multilinear Algebra,  70  (2022),  2,
    331--365. 
    
    
     
 

  
 
  



\bibitem{chen08}
Chen L., Niu Y.,   Meng D., 
    Two kinds of Novikov algebras and their realizations, 
 Journal of Pure and Applied Algebra, 212 (2008), 4, 902--909.

\bibitem{dz11}
Dzhumadildaev A.,
 Codimension growth and non--Koszulity of Novikov operad,
 Communications in Algebra, 39 (2011),  8, 2943--2952.

\bibitem{di14}
Dzhumadildaev A., Ismailov N.,
 $S_n$- and $GL_n$-module structures on free Novikov algebras,
 Journal of Algebra, 416 (2014), 287--313.


\bibitem{dl02}
Dzhumadildaev A., Löfwall C.,
 Trees, free right-symmetric algebras, free Novikov algebras and identities,
Homology, Homotopy and Applications, 4 (2002),  2,  165--190.

\bibitem{dt05}
Dzhumadildaev A., Tulenbaev K.,
 Engel theorem for Novikov algebras, 
 Communications in Algebra, 34 (2006),  3, 883--888.


\bibitem{du17}
Duĭsengalieva B.,  Umirbaev U., 
    A wild automorphism of a free Novikov algebra. (Russian), 
    Siberian Electronic Mathematical Reports, 15 (2018), 1671--1679.


\bibitem{fkkv22}
Fern\'andez Ouaridi A.,  Kaygorodov I.,  Khrypchenko M., Volkov Yu., 
    Degenerations of nilpotent algebras,
     Journal of Pure and Applied Algebra,   226 (2022),  3, 106850.

 


\bibitem{fi89}
Filippov V.,
 A class of simple nonassociative algebras,
  Mathematical Notes, 45 (1989),  1-2, 68--71.
 

\bibitem{fi01}   
Filippov V., 
    On right-symmetric and Novikov nil algebras of bounded index, 
    Mathematical Notes, 70 (2001),  1-2, 258--263
    
 


\bibitem{gd79}
Gelfand I., Dorfman I., 
    Hamiltonian operators and algebraic structures related to them, 
   Funktsional'nyi Analiz i ego Prilozheniya, 13  (1979), 4, 13–30.


 
 
 

 

\bibitem{hac16}
Hegazi A., Abdelwahab H., Calderón Martín A.,
    The classification of $n$-dimensional non-Lie Malcev algebras with $(n-4)$-dimensional annihilator, 
    Linear Algebra and its Applications, 505 (2016), 32--56.


  \bibitem{ikp20}
 Ignatyev M.,  Kaygorodov I., Popov Yu., 
  The geometric classification of $2$-step nilpotent algebras   and applications,     Revista Matemática Complutense, 35 (2022), 3, 907–922. 


  \bibitem{jkk21}
 Jumaniyozov D., Kaygorodov I.,   Khudoyberdiyev  A.,  
The algebraic  classification of nilpotent commutative algebras, 
Electronic Research Archive,    29  (2021),   6, 3909--3993. 


 
 
 
\bibitem{kkk19}
  Karimjanov I., Kaygorodov I., Khudoyberdiyev  A.,  
The algebraic and geometric classification of nilpotent Novikov algebras,
Journal of Geometry and Physics,   143 (2019), 11--21. 

 
\bibitem{kkl20}
Kaygorodov I., Khrypchenko M., Lopes S.,
    The algebraic and geometric classification of nilpotent anticommutative algebras,
    Journal of Pure and Applied Algebra, 224 (2020), 8, 106337.

\bibitem{kkl22}  
 Kaygorodov I., Khrypchenko M., Lopes S., 
The algebraic      classification of nilpotent  algebras, 
 Journal of Algebra and its Applications,  21  (2022), 12, 2350009.


\bibitem{kkp20}
Kaygorodov I., Khrypchenko M., Popov Yu., 
The algebraic and geometric classification of nilpotent terminal algebras, Journal of Pure and Applied Algebra,  225 (2021), 6, 106625.



\bibitem{klp20} Kaygorodov I., Lopes S., P\'{a}ez-Guill\'{a}n P.,   Non-associative central extensions of null-filiform associative algebras, 
Journal of Algebra,   560  (2020),   1190--1210.




\bibitem{krs20} Kaygorodov I.,  Rakhimov I.,  Said Husain Sh. K., 
The algebraic  classification of nilpotent  associative commutative algebras,
Journal of Algebra and its Applications, 19 (2020), 11,    2050220.
 
 

  


\bibitem{kv16}
Kaygorodov I.,   Volkov Yu.,
    The variety of $2$-dimensional algebras over an algebraically closed field,
     Canadian  Journal of Mathematics,  71 (2019),  4, 819--842.

  
  
\bibitem{l20}
   Lebzioui H.,  
   On pseudo-Euclidean Novikov algebras, 
   Journal of Algebra, 564 (2020), 300--316.
   
   
   
\bibitem{mlu10}
   Makar-Limanov L.,  Umirbaev  U., 
The Freiheitssatz for Novikov algebras, 
TWMS Journal of Pure and Applied Mathematics, 2 (2011), 2, 228--235.



   
\bibitem{os94}
Osborn J.,
    Infinite-dimensional Novikov algebras of characteristic 0,
    Journal of Algebra, 167 (1994),  1, 146--167
 
     
 
\bibitem{zerui20}
 Shestakov I.,  Zhang Z., 
    Solvability and nilpotency of Novikov algebras,  
    Communications in Algebra,  48 (2020),  12, 5412--5420.
 

\bibitem{ss78}
Skjelbred T., Sund T.,
    Sur la classification des algebres de Lie nilpotentes,
    C. R. Acad. Sci. Paris Ser. A-B, 286 (1978), 5,  A241--A242.

  

\bibitem{tang12}
Tang X.,  Bai C., 
A class of non-graded left-symmetric algebraic structures on the Witt algebra,  
Mathematische Nachrichten, 285 (2012),  7, 922--935.

 



\bibitem{xu01}
Xu X.,
    Classification of simple Novikov algebras and their irreducible modules of characteristic 0,
    Journal of Algebra, 246 (2001),  2, 673--707.

\bibitem{xu97}
Xu X.,
    Novikov--Poisson algebras,
    Journal of Algebra, 190 (1997),  2, 253--279.

\bibitem{xu96}
Xu X.,
    On simple Novikov algebras and their irreducible modules,
    Journal of Algebra, 185 (1996),  3, 905--934.


\bibitem{ze87}
Zelmanov E.,
    A class of local translation-invariant Lie algebras,
    Soviet   Mathematics Doklady, 35 (1987),  6, 216--218.

 
 
 
\end{thebibliography}
\end{document}